\documentclass[11pt]{amsart}%
\usepackage{palatino, mathpazo}
\usepackage{amssymb}
\usepackage{amsfonts}
\usepackage{graphicx}
\usepackage{amsmath}
\usepackage[shortlabels,inline]{enumitem}
\usepackage{hyperref}%
\usepackage{cite}
\providecommand{\U}[1]{\protect\rule{.1in}{.1in}}
\hypersetup{
	colorlinks=true,
	linkcolor=blue,
	filecolor=black,
	urlcolor=black,
	citecolor=red,
}
\newtheorem{theorem}{Theorem}[section]

\newtheorem{corollary}[theorem]{Corollary}

\newtheorem{definition}[theorem]{Definition}
\newtheorem{example}{Example}
\newtheorem{lemma}[theorem]{Lemma}

\newtheorem{proposition}[theorem]{Proposition}
\newtheorem{remark}[theorem]{Remark}

\numberwithin{equation}{section}
\usepackage{caption}
\makeatletter
\def\fps@figure{htbp} 
\def\fnum@figure{\textbf{Fig. \thefigure}} 
\makeatother

\begin{document}

\title[Monodromy equivalence]{Monodromy Equivalence for Lam\'{e}-type Equations I: Finite-gap Structures and Cone Spherical Metrics.}

\author{Ting-Jung Kuo}
\address[Ting-Jung Kuo]{Department of mathematics, National Taiwan Normal University, Taipei, 11677, Taiwan \& National Center for Theoretical Sciences, No.1 Sec.4 Roosevelt Rd., National Taiwan University, Taipei 10617, Taiwan.}
\email{tjkuo1215@ntnu.edu.tw, tjkuo1215@gmail.com}

\author{Xuanpu Liang}
\address[Xuanpu Liang]{Laboratory of Mathematics and Complex Systems (Ministry of Education), School of Mathematical Sciences, Beijing Normal University, Beijing 100875, China.}
\email{xuanpuliang@mail.bnu.edu.cn}

\author{Ping-Hsiang Wu}
\address[Ping-Hsiang Wu]{Department of Mathematics, National Taiwan Normal University, Taipei 11677, Taiwan.}
\email{phsiangwu@gmail.com}

\begin{abstract}
Motivated by the finite-gap structure of the classical Lam\'{e} equation
(\ref{Lame n}) and its central role in mathematical physics, generalized
Lam\'{e}-type equations (\ref{GLE 3}) are investigated. For the fundamental
case $n=1$, a monodromy equivalence between the classical Lam\'{e} equation
(\ref{Lame 1}) and the generalized Lam\'{e}-type equation (\ref{GLE 4}) is
established. Two main applications are obtained: (i) the finite-gap structure
of \ (\ref{GLE 4}) is derived, together with a complete classification of the
spectral curves $\sigma_{1}$ and $\sigma_{2}$ for $\tau\in i\mathbb{R}_{>0}$;
and (ii) the monodromy equivalence is applied to the construction of cone
spherical metrics with three large conical singularities, each with cone angle
exceeding $2\pi$. A family of such metrics is shown to exhibits a blow-up
configuration, which is described explicitly in terms of the monodromy data.

\end{abstract}

\maketitle
\tableofcontents
	

\section{Introduction}

\label{Introduction}

Let $\omega_{0}$ $=$ $0,$ $\omega_{1}$ $=$ $1,$ $\omega_{2}$ $=$ $\tau$,
$\omega_{3}$ $=$ $1+\tau$, $\Lambda_{\tau}$ $=$ $\mathbb{Z}$ $\mathbb{\oplus}$
$\mathbb{\tau Z}$, and $E_{\tau}:=\mathbb{C}/\Lambda_{\tau}$ where $\tau$
$\in$ $\mathbb{H}$ $=$ $\left\{  \tau|\operatorname{Im}\tau>0\right\}  $. Let
$\wp\left(  z;\tau\right)  $ be the Weierstrass elliptic function with periods
$1$ and $\tau$. Let $\zeta\left(  z;\tau\right)  $ and $\sigma\left(  z;\tau\right)  $ be the Weierstrass zeta and sigma
function, defined by
\begin{align*}
\zeta\left(  z;\tau\right)  &:=-\int^{z}\wp(\xi;\tau)\,d\xi,\\
\sigma(z;\tau) &:=\exp\int^{z}\zeta(\xi;\tau)\,d\xi,
\end{align*}
which is entire on $\mathbb{C}$ and has a simple zero at $z=0$. 

For brevity,
we write $\wp\left(  z;\tau\right)  $, $\zeta\left(  z;\tau\right)  $, and
$\sigma(z;\tau)$ simply as $\wp\left(  z\right)  $, $\zeta(z)$, and
$\sigma(z)$, respectively. Let $p\in E_{\tau}$. Throughout this paper, we
assume
\begin{equation}
p\in E_{\tau}\backslash\left\{ \frac{\omega_{k}}{2},k=0,1,2,3\right\}
\label{Assumption 1}%
\end{equation}

Let $n\in\mathbb{N}$. We consider two linear elliptic Fuchsian differential
equations on elliptic curve $E_{\tau}$. The first is the classical Lam\'{e}
equation \cite{Ince, Whittaker-Watson}:
\begin{equation}
y^{\prime\prime}(z)=\left[  n(n+1)\wp\left(  z\right)  +\tilde{B}\right]
y(z),\text{ }\tilde{B}\in\mathbb{C}, \label{Lame n}%
\end{equation}
and the second is a generalized Lam\'{e}-type equation:
\begin{equation}
y^{\prime\prime}(z)=q_{n}(z;p,T_{1},T_{2},B)\,y(z), \label{GLEn}%
\end{equation}
where
\begin{align*}
q_{n}(z;p,T,B)  &  =n(n+1)\wp(z)+\frac{3}{4}(\wp(z+p)+\wp(z-p))\\
&  +T_{1}(\zeta(z+p)-\zeta(z))+T_{2}(\zeta(z-p)-\zeta(z))+B
\end{align*}
with parameters $T_{1},$ $T_{2},$ $B\in\mathbb{C}$, and $p\in E_{\tau
}\backslash\{\frac{\omega_{k}}{2},k=0,1,2,3\}$. Equation (\ref{Lame n}) is
known as classical Lam\'{e} equation in the literature, while (\ref{GLEn}) can
be viewed as its natural variant, distinguished by the presence of nonzero
residues at each singularity.

The monodromy problem for the integer Lam\'{e} equation (\ref{Lame n}) and its
Treibich--Verdier generalizations, is considered fundamental due to its
\textit{finite-gap} \textit{structure} or \textit{algebro-geometric
structure}. See \cite{Maier, TV}.

Finite-gap integration theory, also known as algebro-geometric integration, is
a powerful method in the theory of integrable systems for constructing and
analyzing solutions of certain nonlinear partial differential equations
(PDEs), such as: the Korteweg-de Vries (KdV) equation \cite{BBEIM, EK-CMH,
GH-Book, GH, GUW1, GW4, GW2, GW3, GW1, GW5, kri, Smirnov1989, Smirnov1994},
nonlinear Schr\"{o}dinger equation (NLS) \cite{BBEIM, ItM, ItRS, kir1, kri2,
ZS}, and the Sine-Gordon equation \cite{K,KK}. This theory is deeply rooted in
algebraic geometry, particularly the theory of compact Riemann surfaces and
their Jacobians \cite{Matveev}. Beyond mathematics, it has found significant
applications in physics, such as the stability analysis of critical droplets
in bounded spatial domains \cite{Maier, Maier-Stein}.

A central question motivating this work is whether the generalized
Lam\'{e}-type equation (\ref{GLEn}) also admits a finite-gap structure,
analogous to the classical Lam\'{e} equation (\ref{Lame n}). This series of
studies develops a rigorous framework establishing the monodromy equivalence
between the classical Lam\'{e} equation and the generalized Lam\'{e}-type
equation in the \textit{punctured non-even case} (see Case II below). This
equivalence not only resolves the finite-band question but also yields new
applications to spherical metrics with conical singularities.

The monodromy theory for the classical Lam\'{e} equation (\ref{Lame n}) had
been thoroughly studied in \cite{CLW, LW}. In this work, we focus our
attention on the investigation of the monodromy theory associated with the
generalized Lam\'{e}-type equation (\ref{GLEn}). In the present
paper\,---\,constituting Part I of our study\,---\,we focus on the case $n=1$, which
allows explicit computations and reveals the core ideas underlying the general theory.

A Lam\'{e}-type equation is called \textit{apparent }if it is free of
logarithmic singularity at each singularity. It is well known that the
classical Lam\'{e} equation (\ref{Lame n}) is apparent for every $\tilde{B}
\in\mathbb{C}$, due to the evenness of the potential $n(n+1)\wp\left(
z;\tau\right)  $ and the fact $z=0$ is a symmetric point on the elliptic curve
$E_{\tau}$. In contrast, the generalized equation (\ref{GLEn}) involves
nonzero residues and fails to be apparent unless certain explicit algebraic
conditions are satisfied. See (\ref{AP condition 1}) below.

Denote by $AP(\tau,p)$ the space of apparent parameters. For a given parameter
$\mathbf{T}$ $\mathbf{=}$ $\left(  T_{1},T_{2}\right)  $ $\in$ $AP(\tau,p)$,
we define the monodromy representation of the generalized Lam\'{e}-type
equation (\ref{GLEn}): Note that the local indices at $z=0$, $\pm p$ are $-n,$
$n+1$ and $-1/2,$ $3/2$, respectively, so the corresponding local monodromy
matrices at $z=0$, $\pm p$ are
\[
N_{0}=Id_{2\times2}\text{ and }N_{\pm p}=-Id_{2\times2}.
\]
We choose a base point $z_{0}\not \in \{0,\pm p\}$ such that $\{0,\pm
p\}+\Lambda_{\tau}$ $\not \in $ $z_{0}+\mathbb{R}$ and $z_{0}+\tau
\cdot\mathbb{R}$. Take a fundamental system of solution $Y(z;z_{0})$ of
equation (\ref{GLEn}) in a small neighborhood $U_{z_{0}}$ of $z_{0}$. Define
two global monodromy matrices $M_{1}(\mathbf{T};p)$ and $M_{2}(\mathbf{T};p)$
corresponding to the generators of the fundamental group $\pi_{1}(E_{\tau})$
via analytic continuation of $Y(z;z_{0})$ along loops associated with the
periods $\omega_{1}=1$ and $\omega_{2}=\tau$, respectively:%

\[
Y(z+\omega_{i})=M_{j}(\mathbf{T};p)Y(z).
\]
These matrices satisfy the monodromy relation
\[
M_{1}M_{2}M_{1}^{-1}M_{2}^{-1}=N_{p}N_{0}N_{-p}=Id_{2\times2}.
\]
Thus, the monodromy representation $\rho_{\tau,p}(\mathbf{T}):\pi_{1}(E_{\tau
})\rightarrow SL(2,\mathbb{C})$, is abelian, leading to two cases:

\noindent(i) \textbf{Completely Reducible case}:\textit{\ }$M_{i}
(\mathbf{T};p),$ $i$ $=$ $1,2$ can be simultaneously diagonalized as follows
\begin{equation}
M_{1}(\mathbf{T};p)=\left(
\begin{array}
[c]{cc}%
e^{-2\pi is} & 0\\
0 & e^{2\pi is}%
\end{array}
\right)  ,\text{ }M_{2}(\mathbf{T};p)=\left(
\begin{matrix}
e^{2\pi ir} & 0\\
0 & e^{-2\pi ir}%
\end{matrix}
\right)  , \label{m1}%
\end{equation}
for some $r,s\in$ $\mathbb{C}$.\textit{\medskip}

\noindent(ii) \textbf{Non-Completely Reducible case}: $M_{i}(\mathbf{T};p),$
$i$ $=$ $1,2,$ cannot be simultaneously diagonalized. Instead, they can be
normalized to:
\begin{equation}
M_{1}(\mathbf{T};p)=\pm\left(
\begin{array}
[c]{cc}%
1 & 0\\
1 & 1
\end{array}
\right)  \text{, }M_{2}(\mathbf{T};p)=\pm\left(
\begin{matrix}
1 & 0\\
\mathcal{D} & 1
\end{matrix}
\right)  ,\text{ } \label{m2}%
\end{equation}
where $\mathcal{D}$ $\in$ $\mathbb{C\cup\{ \infty\}}$. In the case
$\mathcal{D=\infty}$, it should be understood as
\[
M_{1}(\mathbf{T};p)=\pm\left(
\begin{array}
[c]{cc}%
1 & 0\\
0 & 1
\end{array}
\right)  \text{, }M_{2}(\mathbf{T};p)=\pm\left(
\begin{matrix}
1 & 0\\
1 & 1
\end{matrix}
\right)  .
\]
\textit{\medskip}

The pair $\left(  r,s\right)  $ $\in$ $\mathbb{C}$ in case (i) and
$\mathcal{D}$ $\in$ $\mathbb{C\cup\{ \infty\}}$ in case (ii) , are
respectively referred to as the monodromy data. A natural question is the
following:\textit{\medskip}

\noindent\textbf{Question:}\textit{ Let} $p$ $\in$ $E_{\tau}\backslash
\{\frac{\omega_{k}}{2},k=0,1,2,3\}$. \textit{Given }$\left(  r,s\right)
$\textit{ }$\in$\textit{ }$\mathbb{C}^{2}$\textit{ in case (i), and }
$D$\textit{ }$\in$\textit{ }$\mathbb{C}\cup\{\infty\}$\textit{ in case (ii),
does there exist a parameter }$\mathbf{T}$ $\in$ $AP(\tau,p)$\textit{ such
that the corresponding monodromy matrices }$M_{i}(\mathbf{T},p),$\textit{ }
$i$\textit{ }$=$\textit{ }$1,2$\textit{, are of the forms given in (\ref{m1})
and (\ref{m2}), respectively?\medskip}

To address the above question, we first examine the apparentness condition.
The apparentness condition for the Lam\'{e}-type equation (\ref{GLEn}) imposes
an explicit algebraic relation among the parameters $T_{1}$, $T_{2}$ and $B$.
Specifically, the condition that all singularities are apparent is equivalent
to the system:
\begin{equation}
(T_{1}+T_{2})\left(  T_{1}-T_{2}+\frac{\wp^{\prime\prime}(p)}{2\wp^{\prime
}(p)}\right)  =0 \label{AP condition 1}%
\end{equation}
and $B$ is determined by
\begin{equation}
B=\frac{T_{1}^{2}+T_{2}^{2}}{2}-\frac{T_{1}-T_{2}}{2}\zeta(2p)-\frac{3}{4}
\wp(2p)-n(n+1)\wp(p). \label{B condition 1}%
\end{equation}
Since the present paper deals with the case $n=1$, we will provide a direct
proof of these relations via explicit computation in Section
\ref{Spectral Theory Section}. The general case for arbitrary $n\in\mathbb{N}$
will be addressed in Part II of this series.

From (\ref{AP condition 1}), we have
\[
AP(\tau,p)=\left\{  (T_{1},T_{2})|\text{(\ref{AP condition 1}) holds}\right\}
\]
naturally decomposes into two components:
\[
AP_{0}(\tau,p)=\left\{  (T_{1},T_{2})|T_{1}+T_{2}=0\right\}  ,
\]
\[
AP_{1}(\tau,p)=\left\{  (T_{1},T_{2})\left\vert T_{1}-T_{2}+\frac{\wp
^{\prime\prime}(p)}{2\wp^{\prime}(p)}=0\right.  \right\}  .
\]
Clearly,
\[
AP(\tau,p)=AP_{0}(\tau,p)\oplus AP_{1}(\tau,p)
\]
and
\[
AP_{0}(\tau,p)\cap AP_{1}(\tau,p)=\left\{  \left(  -\frac{\wp^{\prime\prime
}(p)}{4\wp^{\prime}(p)},\frac{\wp^{\prime\prime}(p)}{4\wp^{\prime}(p)}\right)
\right\}
\]
is the only singular point of $AP(\tau,p)$. Accordingly, we have the following
two cases:\textit{\medskip}

\textbf{Case I (Even symmetry): }If $\mathbf{T}$ $=$ $(T_{1},T_{2})$ $\in$
$AP_{0}(\tau,p)$, then we may set
\[
\mathbf{T}=\left(  A,-A\right)  ,\text{ }A\in\mathbb{C}.
\]
\newline Under this assumption, the equation (\ref{GLEn}) becomes
\begin{equation}
y^{\prime\prime}(z)=\left[
\begin{array}
[c]{l}%
n(n+1)\wp(z)+\frac{3}{4}(\wp(z+p)+\wp(z-p))\\
+A(\zeta(z+p)-\zeta(z-p))+B
\end{array}
\right]  y(z),\text{ }z\in E_{\tau}, \label{GLE 2}%
\end{equation}
and by (\ref{B condition 1}) $B$ is rephrased as follows
\begin{equation}
B=A^{2}-\zeta(2p)A-\frac{3}{4}\wp(2p)-n(n+1)\wp(p). \label{APeven}%
\end{equation}
It is easy to verify that, in this case, the potential is always an
\textit{even elliptic} function for any $A\in\mathbb{C}$. Since the potential
is an even function, equation (\ref{GLE 2}) is invariant under the
transformation $z\rightarrow-z$. Under this transformation, the monodromy data
change as follows:

\begin{itemize}
\item In completely reducible case (i):
\begin{equation}
\left(  r(A),s(A)\right)  \rightarrow-\left(  r(A),s(A)\right)
\label{Monodeomy symmetry1}%
\end{equation}

\item In non-completely reducible case (ii):
\begin{equation}
\mathcal{D(}A\mathcal{)\rightarrow-D(}A\mathcal{)} \label{Monodromy symmetry2}%
\end{equation}
Accordingly, these monodromy data are regarded as equivalent under the
symmetry $z\longmapsto-z$.\textit{\medskip}
\end{itemize}

\textbf{Case II (Punctured Non-even symmetry): }If $\mathbf{T}$ $=$
$(T_{1},T_{2})$ $\in$ $AP_{1}(\tau,p)$, then we can write
\[
\mathbf{T}=\left(  T-\frac{\wp^{\prime\prime}(p)}{4\wp^{\prime}(p)}
,T+\frac{\wp^{\prime\prime}(p)}{4\wp^{\prime}(p)}\right)  ,\text{ }
T\in\mathbb{C}.
\]
Thus, (\ref{GLEn}) becomes
\begin{equation}
y^{\prime\prime}(z)=\left[
\begin{array}
[c]{l}%
n(n+1)\wp(z)+\frac{3}{4}(\wp(z+p)+\wp(z-p))\\
+T(\zeta(z+p)+\zeta(z-p)-2\zeta(z))\\
-\frac{\wp^{\prime\prime}(p)}{4\wp^{\prime}(p)}(\zeta(z+p)-\zeta(z-p))+B
\end{array}
\right]  y(z),\text{ }z\in E_{\tau}, \label{GLE 3}%
\end{equation}
and
\begin{equation}
B=T^{2}+\frac{\wp^{\prime\prime}(p)}{2\wp^{\prime}(p)}\zeta(p)-\frac{1}
{2}(2n^{2}+2n-3)\wp(p). \label{APnoneven}%
\end{equation}
It is clear that the potential remains elliptic, but---unlike in Case
I---\textit{it is not even unless }$T=0$\textit{. }

For generalized Lam\'{e}-type equation (\ref{GLE 3}), the transformation
$z\longmapsto-z$ corresponds to $T\longmapsto-T$, since $B(T)$ $=$ $B(-T)$. As
a result, the monodromy data of the equation with respect to $T$ and $-T$
relate as follows:

\begin{itemize}
\item In completely reducible case (i):
\begin{equation}
\left(  r(-T),s(-T)\right)  =-\left(  r(T),s(T)\right)  .
\label{Monodromy symmetry3}%
\end{equation}

\item In non-completely reducible case (ii):
\begin{equation}
\mathcal{D(-}T\mathcal{)=-D(}T\mathcal{)}\text{.} \label{Monodromy symmetry4}%
\end{equation}

\end{itemize}

When $T=0$, equation (\ref{GLE 3}) yields to an even potential, equivalent to
Case I (Even symmetry) with
\[
A=-\frac{\wp^{\prime\prime}(p)}{4\wp^{\prime}(p)}.
\]
Notably, $T=0$, which corresponds to the intersection point
\begin{equation}
\mathbf{T}_{\ast}\mathbf{=}\left(  -\frac{\wp^{\prime\prime}(p)}{4\wp^{\prime
}(p)},\frac{\wp^{\prime\prime}(p)}{4\wp^{\prime}(p)}\right)  ,
\label{singular point}%
\end{equation}
is the unique point in $AP_{1}(\tau)$ for which equation (\ref{GLE 3}) yields
an \textit{even} potential. This fact motivates the designation
\textquotedblleft\textit{Punctured Non-even Symmetry}\textquotedblright\ for
this case. A generalized Lam\'{e}-type equation (\ref{GLE 3}) associated with
$\mathbf{T}$ $\mathbf{\in}$ $AP_{1}(\tau,p)$ $\backslash$ $\{\mathbf{T}_{\ast
}\}$ is called \textbf{non-even symmetry}.\textit{\medskip}

The monodromy question above naturally divides into two distinct
cases:\textit{\medskip}

\textbf{Case I.} $\mathbf{T}$ $\mathbf{\in}$ $AP_{0}(\tau,p)$. The monodromy
theory for potentials in this case has been extensively developed in
\cite{Chen-Kuo-Lin-Hamiltonian, Chen-Kuo-Lin-Lame I, Chen-Kuo-Lin-Painleve
VI}, where deep connections were established between isomonodromic
deformations of the family (\ref{GLE 2}) and the elliptic form of the
Painlev\'{e} VI equation (EPVI).

The EPVI associated to the equation (\ref{GLE 2}) is given by
\begin{equation}
\frac{d^{2}p\left(  \tau\right)  }{d\tau^{2}}=\frac{-1}{4\pi^{2}}\sum
_{i=0}^{3}\alpha_{i}\wp^{\prime}\left(  p\left(  \tau\right)  +\frac
{\omega_{i}}{2}|\tau\right)  \label{EPVI 1}%
\end{equation}
with parameters
\[
\left(  \alpha_{0},\alpha_{1},\alpha_{2},\alpha_{3}\right)  =\left(  \frac
{1}{2}\left(  n+\frac{1}{2}\right)  ^{2},\frac{1}{8},\frac{1}{8},\frac{1}
{8}\right)  .
\]
\textit{\medskip}

\noindent\textbf{Theorem A. }\textit{(\cite{Chen-Kuo-Lin-Hamiltonian}
)}\textbf{ }\textit{The generalized Lam\'{e}-type equation in the even
symmetry case (\ref{GLE 2}) with parameters }$\left(  p(\tau),A(\tau)\right)
$, \textit{preserves its monodromy as the moduli parameter }$\tau$\textit{
varies if and only if }$p(\tau)$\textit{ is a solution of EPVI (\ref{EPVI 1}
).\medskip}

According to the characterization of the monodromy matrices described in (i)
and (ii), any solution $p^{(n)}(\tau)$ to EPVI (\ref{EPVI 1}) can be
classified as follows:\textit{\medskip}

(a) \textbf{Completely reducible solution: }$p(\tau)=p_{r,s}^{(n)}(\tau)$,
where the associated equation (\ref{GLE 2}) is completely reducible, and
$\left(  r,s\right)  \in$\textit{ }$\mathbb{C}^{2}$ is the monodromy
data.\textit{\medskip}

(b) \textbf{Non-completely reducible solution: }$p(\tau)=p_{\mathcal{D}}
^{(n)}(\tau)$, where the associated equation (\ref{GLE 2}) is non-completely
reducible and $\mathcal{D}\in$\textit{ }$\mathbb{C}\cup\{\infty\}$ is the
monodromy data.\textit{\medskip}

\noindent\textbf{Theorem B. }\textit{(\cite{Chen-Kuo-Lin-Hamiltonian}) Given
}$\left(  r,s\right)  $\textit{ }$\in$\textit{ }$\mathbb{C}^{2}$ \textit{in
case (i), or }$\mathcal{D}$\textit{ }$\in$\textit{ }$\mathbb{C}\cup\{\infty
\}$\textit{ in case (ii), there exists }$\left(  p,\text{ }A\right)  $\textit{
such that the corresponding monodromy matrices }$M_{i}(p,A),$\textit{ }
$i$\textit{ }$=$\textit{ }$1,2$\textit{, are of the forms given in (\ref{m1})
or (\ref{m2}), respectively, if and only if }$p=p_{r,s}^{(n)}(\tau)$\textit{
or }$p=p_{\mathcal{D}}^{(n)}(\tau)$\textit{, evaluated at }$\tau
$\textit{.\medskip}

\textbf{Case II.} $\mathbf{T}$ $\mathbf{\in}$ $AP_{1}(\tau,p)$. In this case,
the analysis of the associated monodromy theory becomes considerably more
intricate due to the non-vanishing residues at $\pm p$ and the lack of even
symmetry in the potential if $T\not =0$.

The primary goal of the present paper is to address the monodromy question for
$\mathbf{T}$ $\mathbf{\in}$ $AP_{1}(\tau,p)$ in the case $n$ $=$
$1$.\textit{\medskip}

When $n=1$, the classical Lam\'{e} equation and the Lam\'{e}-type equation in
Case II take the following forms respectively:
\begin{equation}
y^{\prime\prime}(z)=\left[  2\wp\left(  z\right)  +\tilde{B}\right]
y(z),\text{ }\tilde{B}\in\mathbb{C}, \label{Lame 1}%
\end{equation}
and
\begin{equation}
y^{\prime\prime}(z)=\left[
\begin{array}
[c]{l}%
2\wp(z)+\frac{3}{4}(\wp(z+p)+\wp(z-p))\\
+T(\zeta(z+p)+\zeta(z-p)-2\zeta(z))\\
-\frac{\wp^{\prime\prime}(p)}{4\wp^{\prime}(p)}(\zeta(z+p)-\zeta(z-p))+B
\end{array}
\right]  y(z) \label{GLE 4}%
\end{equation}
with
\begin{equation}
B=T^{2}+\frac{\wp^{\prime\prime}(p)}{2\wp^{\prime}(p)}\zeta(p)-\frac{1}{2}
\wp(p). \label{B condition 2}%
\end{equation}

Our main result asserts that, for any $p$ satisfying (\ref{Assumption 1}), the
generalized Lam\'{e}-type equation (\ref{GLE 4})$_{p}$ parameter $T$ and $B$
given by (\ref{B condition 2}) is monodromy equivalent to the classical
Lam\'{e} equation (\ref{Lame 1}) in the following sense:

\begin{theorem}
\label{Main Theorem1}Assume (\ref{Assumption 1}). Given\textit{ }$\left(
r,s\right)  $\textit{ }$\in$\textit{ }$\mathbb{C}^{2}$ \textit{in case (i)
(\ref{m1}), or }$\mathcal{D}$\textit{ }$\in$\textit{ }$\mathbb{C}\cup
\{\infty\}$\textit{ in case (ii) (\ref{m2}), the following statements are
equivalent:}

(1) \textit{There exists }$T$ $\in\mathbb{C}$, with $B$ given by
(\ref{B condition 2}) such that the monodromy matrices\textit{ }$M_{i}(p,T)$
$\in$ $SL(2,\mathbb{C}),$\textit{ }for $i$\textit{ }$=$\textit{ }
$1,2$\textit{, of }(\ref{GLE 4})$_{p}$ correspond to the given $\left(
r,s\right)  $\textit{ or }$\mathcal{D}$.

(2)\textit{ There exists }$\tilde{B}\in\mathbb{C}$ such that the
\textit{monodromy matrices }$M_{i}(\tilde{B})$ $\in$ $SL(2,\mathbb{C})$, for
$i=1,2,$\ of the classical Lam\'{e} equation (\ref{Lame 1}) correspond to the
same $\left(  r,s\right)  $\textit{ or }$\mathcal{D}$\textit{.}

Moreover, the correspondence between (1) and (2) is a two-to-one mapping
\[
\mathbb{C\longrightarrow C}\text{, }~~T\longmapsto\tilde{B}\text{,}
\]
given by the formula
\begin{equation}
\tilde{B}=T^{2}-2\wp(p). \label{correspondence}%
\end{equation}

\end{theorem}

\begin{remark}
Theorem \ref{Main Theorem1} holds, in fact, for arbitrary $n\in\mathbb{N}$. A
complete proof will be presented in Part II. In this paper, we focus on the
case $n=1$, providing explicit formulas and computations that serve as the
foundation for the general theory.
\end{remark}

A direct corollary of Theorem \ref{Main Theorem1} is stated as follows.

\begin{corollary}
\label{Main Corollary1}Suppose $\left(  r,s\right)  $\textit{ }$\in$\textit{
}$\mathbb{C}^{2}$ \textit{in case (i), or }$\mathcal{D}$\textit{ }$\in
$\textit{ }$\mathbb{C}\cup\{\infty\}$\textit{ in case (ii), represents the
monodromy data of }(\ref{Lame 1}) for some $\tilde{B}$\textit{ }$\in$\textit{
}$\mathbb{C}$. For each $p$ satisfying (\ref{Assumption 1}), let the
parameters $T(p)$ and $B(p)$ be given by (\ref{correspondence}) and
(\ref{B condition 2}), respectively. Then monodromy matrices\textit{ }
\[
M_{i}(p,T(p))\in SL(2,\mathbb{C}),~i=1,2,
\]
of (\ref{GLE 4})$_{p}$\textit{ }remain invariant\textit{ and coincide with the
corresponding monodromy matrices of (\ref{Lame 1}) associated with }$\tilde
{B}$\textit{.}
\end{corollary}

The next theorem emphasizes that the singular case $T(p)=0$ must occurs for
some $p$ in Corollary \ref{Main Corollary1}.

\begin{theorem}
\label{Main Theorem at singular}Let $\tau\in\mathbb{H}$ and let $\left(
T(p),B(p)\right)  $ be as in Corollary \ref{Main Corollary1}, with
(\ref{GLE 4})$_{p}$ completely reducible and having monodromy data $\left(
r,s\right)  $. Then
\[
T(p)=0\,\Leftrightarrow\,\tilde{B}=-2\wp(p)\,\Leftrightarrow\, p=\pm p_{r,s}
^{(1)}(\tau)
\]
where $p_{r,s}^{(1)}(\tau)$ denotes the solution of EPVI (\ref{EPVI 1}) with
$n$ $=$ $1$, \textit{evaluated at }$\tau$.
\end{theorem}

We now investigate the limiting behavior as $p\rightarrow\frac{\omega_{k}}{2}
$. Although the singularity $p$ in the equation (\ref{GLE 4})$_{p}$ is not
well-defined when
\[
p\in\left\{  \frac{\omega_{k}}{2},k=0,1,2,3\right\}  ,
\]
we will show that the generalized Lam\'{e}-type equation (\ref{GLE 4})$_{p}$
with $\left(  T(p),B(p)\right)  $ chosen to preserve the monodromy matrices,
converges uniformly on every compact subset of $E_{\tau}\backslash\{0\}$ as
$p\rightarrow\frac{\omega_{k}}{2}$, for any $k=0,1,2,3.$

Moreover, the monodromy matrices
\[
M_{i}^{(k)}\in SL(2,\mathbb{C}),~i=1,2,
\]
associated with the limiting Lam\'{e}-type equation as $p\rightarrow
\frac{\omega_{k}}{2},$ are determined by $M_{i}(p,T(p))$, which, by Corollary
\ref{Main Corollary1}, are independent of $p$.

\begin{theorem}
\label{Main Theorem3}Let $\left(  T(p),B(p)\right)  $ be as in Corollary
\ref{Main Corollary1}. Namely,
\[
T(p)^{2}-2\wp(p)=\tilde{B}
\]

(i) As $p\rightarrow0$, the generalized Lam\'{e}-type equation in the non-even
case (\ref{GLE 4})$_{p}$, with parameters $\left(  T(p),B(p)\right)  $
converges uniformly to the classical Lam\'{e} equation (\ref{Lame 1})
\textit{associated with }$\tilde{B}$. The associated monodromy matrices
$M_{i}^{(0)}(\tilde{B})$ of the limiting Lam\'{e} equation satisfy
\[
M_{i}^{(0)}(\tilde{B})=M_{i}(p,T(p))\text{, }i=1,2.
\]

(ii) As $p\rightarrow\frac{\omega_{k}}{2}$, for any $k=1,2,3$, the generalized
Lam\'{e}-type equation in the non-even case (\ref{GLE 4})$_{p}$, with $\left(
T(p),B(p)\right)  $ converges uniformly to the Lam\'{e}-type equation
\begin{equation}
y^{\prime\prime}(z)=\left\{  2\left(  \wp(z)+\wp(z-\frac{\omega_{k}}%
{2})\right)  +2\tilde{T}_{k}\left(  \zeta(z-\frac{\omega_{k}}{2}%
)-\zeta(z)\right)  +\tilde{B}_{k}\right\}  y(z) \label{limitting k}%
\end{equation}
where
\[
\tilde{T}_{k}=\lim_{p\rightarrow\frac{\omega_{k}}{2}}T(p),
\]
and
\[
\tilde{B}_{k}=\tilde{T}_{k}^{2}+\eta_{k}\tilde{T}_{k}-e_{k}.
\]
Furthermore, the monodromy matrices $M_{i}^{(k)}(\tilde{T}_{k}),$ $i=1,2,$
associated with (\ref{limitting k}) satisfy
\begin{equation}
M_{i}^{(k)}(\tilde{T}_{k})=\varepsilon_{i}^{(k)}M_{i}(p,T(p))\text{, }i=1,2,
\label{limiting M}%
\end{equation}
where the signs are given by
\[
\left(  \varepsilon_{1}^{(k)},\varepsilon_{2}^{(k)}\right)  =\left\{
\begin{array}
[c]{c}%
\left(  -1,1\right)  \text{ \ if \ }k=1\\
\left(  1,-1\right)  \text{ \ if \ }k=2\\
\left(  -1,-1\right)  \text{ \ if \ }k=3
\end{array}
\right.  .
\]

\end{theorem}

\textbf{Organization of the Paper:}\textit{\medskip}

\begin{itemize}
\item Section \ref{Finite-Gap Structure Section}. We extend the notion of
finite-gap potentials to include the generalized Lam\'{e}-type equation
(\ref{GLEn}), and for $n=1$, establish the finite-gap structure via Theorem
\ref{Main Theorem1} and describe the spectral arcs for $\tau\in i\mathbb{R}%
_{>0}$.\textit{\medskip}

\item Section \ref{Cone Spherical Metrics Section}. Applications to cone
spherical metrics with three large odd conical singularities are presented. A
family of such metrics exhibiting a blow-up configuration is analyzed via the
associated multiple Green function, and the configuration is described
explicitly in terms of the corresponding monodromy data.\textit{\medskip}

\item Section \ref{Spectral Theory Section}. We introduce spectral polynomials
and Baker-Akhiezer functions, adapted from KdV theory, and employ them to
study the monodromy representations for Case II (punctured non-even
case).\textit{\medskip}

\item Section \ref{Monodromy Theory and The Proofs Section}. We resolves the
monodromy problem outlined above and gives the complete proof of Theorem
\ref{Main Theorem1}, Theorem \ref{Main Theorem at singular} and Theorem
\ref{Main Theorem3}.\textit{\medskip}
\end{itemize}
\vspace{5pt}

\textbf{Acknowledgement:} Ting-Jung Kuo was supported by NSTC 113-2628-M-003-001-MY4. He is also grateful to the National Center for Theoretical Sciences (NCTS) for its constant support.
\vspace{5pt}

\section{Finite-Gap Structure}

\label{Finite-Gap Structure Section}

Finite-gap integration theory plays a central role in the analysis of
nonlinear partial differential equations arising in mathematical physics.
Motivated by the seminal works of Gesztesy and Weikard \cite{GW1, GW2, GW3,
GW4, GW5} on the KdV and AKNS hierarchies; Gesztesy, Holden, Michor, and
Teschl \cite{GH2, GHMT} on the Camassa--Holm and Ablowitz--Ladik hierarchies;
and Gesztesy, Unterkofler, and Weikard \cite{GUW1} on Calogero--Moser systems,
we extend the concept of finite-gap potentials to encompass the generalized
Lam\'{e}-type equation in the punctured non-even case (\ref{GLE 4})$_{p}$ and
establish its finite-gap property. In addition, we describe the spectral arcs
for the case $\tau\in i\mathbb{R}_{>0}$.

Throughout this section, we denote the Case II potential in equation
\eqref{GLE 4}$_{p}$ by $q(z;T)$, where $B$ is determined by \eqref{B condition 2}.

\subsection{Finite-Gap Theory}

Let $T\in\mathbb{C}$. Consider a nonconstant elliptic function $h(z;T)$ in the
variable $z$ with periods $1$ and $\tau$, subject to the following two assumptions:

\begin{itemize}
\item[(1)] For fixed $z$, $h(z;\cdot)$ is meromorphic on $\mathbb{C}$;

\item[(2)] The second-order differential equation
\begin{equation}
y^{\prime\prime}(z)=h(z;T)y(z) \label{finite gpa ode}%
\end{equation}
is assumed to be an apparent Fuchsian equation, that is, all of its
singularities are regular and the local solutions are single-valued (i.e.,
free of logarithmic terms).
\end{itemize}

Fix $z_{0}\in\mathbb{C}$ such that $h(x+z_{0};T)$ is smooth on both
$\mathbb{R}$ and $\tau\cdot\mathbb{R}$. Consider the following Hill's
equation
\begin{equation}
\label{hill equ}y^{\prime\prime}(x) =h(x+z_{0};T)y(x),\quad x\in\mathbb{R}.
\end{equation}

\begin{definition}
\label{sigma spectrum} Equation \eqref{hill equ} is said to be
\textbf{conditionally stable} if it admits a nontrivial bounded solution on
$\mathbb{R}$ (or $\tau\cdot\mathbb{R}$). The associated \textbf{spectral sets}
are defined as
\begin{align*}
\sigma_{1}  &  :=\{T\in\mathbb{C}\,|\,\text{Equation }\eqref{hill equ}\text{
admits a nontrivial bounded solution on }\mathbb{R}.\},\\
\sigma_{2}  &  :=\{T\in\mathbb{C}\,|\,\text{Equation }\eqref{hill equ}\text{
admits a nontrivial bounded solution on }\tau\cdot\mathbb{R}.\},
\end{align*}
that is, the sets of parameters $T$ for which \eqref{hill equ} is
conditionally stable on the respective domains.
\end{definition}

Let $M_{j}(T)$ denote the monodromy matrix of equation \eqref{finite gpa ode}
with respect to the shift $z\rightarrow z+\omega_{j}$. The \textit{Hill
discriminant} of \eqref{finite gpa ode} along the $\omega_{j}$ direction is
defined by
\begin{equation}
\Delta_{j}(T):=\frac{1}{2}\mathrm{tr}\,M_{j}(T),\quad T\in\mathbb{C},\quad
j=1,2.
\end{equation}
Note that $\Delta_{j}(T)$ is independent of the choice of fundamental system
of solutions. By the classical Floquet theory, $\sigma_{j}$ can be
characterized as
\begin{equation}
\sigma_{j}=\{T\in\mathbb{C}\,|\,-1\leqslant\Delta_{j}(T)\leqslant1\},\,j=1,2.
\label{spectrum lame-type}%
\end{equation}

\begin{definition}
\label{endpoint's definition} A point $T_{0}\in\sigma_{j}$ is called an
endpoint if an odd number of semiarcs of $\sigma_{j}$ meet at $T_{0}$.
\end{definition}

For any $T\in\mathbb{C}$, set
\begin{equation}
d_{j}(T):=\mathrm{ord}_{T}\left(  \Delta_{j}^{2}(\cdot)-1\right)  .
\label{degree of spectrum}%
\end{equation}
The following lemma shows that $d_{j}(T)$ completely characterizes the finite
endpoints of $\sigma_{j}$.

\begin{lemma}\label{lemma: odd degree } $T_{0}\in\mathbb{C}$ is an endpoint of $\sigma_{j}$
if and only if $d_{j}(T_{0})$ is odd.
\end{lemma}

\begin{proof}
Suppose $T_{0}\in\mathbb{C}$ is an endpoint of $\sigma_{j}$. Without the loss
of generality, assume $\Delta_{j}(T_{0})=1$. Then locally,
\[
\Delta_{j}(T)-1=l(T-T_{0})^{k}(1+O(|T-T_{0}|)),\quad l\neq0,\ k\text{ is an
odd integer}.
\]

Hence by \eqref{degree of spectrum}, we obtain
\[
d_{j}(T)=k,
\]
which is odd.

Conversely, assume $d_{j}(T_{0})=k$ is odd. Then $\Delta_{j}(T_{0})=\pm1,$ and
near $T_{0}$,
\begin{equation}
\Delta_{j}^{2}(T)-1=l(T-T_{0})^{k}(1+O(T-T_{0})),\quad l\neq0. \label{delta 1}%
\end{equation}
On the other hand,%
\[
\Delta_{j}(T)\mp1=l^{\prime}(T-T_{0})^{k^{\prime}}(1+O(T-T_{0})),\quad
l^{\prime}\neq0,\ k^{\prime}\in\mathbb{N},
\]
which implies
\begin{equation}
\Delta_{j}^{2}(T)-1=2l^{\prime}(T-T_{0})^{k^{\prime}}(1+O(T-T_{0})).
\label{delta 2}%
\end{equation}
Comparing \eqref{delta 1} and \eqref{delta 2}, we conclude that $k^{\prime}=k$
and $l^{\prime}=\frac{l}{2}$. Thus, exactly $k$ semiarcs of $\sigma_{j}$ meet
at $T_{0}$ and hence $T_{0}$ is an endpoint of $\sigma_{j}$.
\end{proof}

\begin{definition}
[finite-gap potential]\label{def finite-gap}The potential $h(z;T)$ is called
\textbf{finite-gap}, if for each j=1,2, the spectral set $\sigma_{j}$ consists
of finitely many bounded analytic arcs and finitely many semi-unbounded
analytic arcs whose infinite endpoints tend to $\infty$.
\end{definition}

\begin{example}
\cite{Matveev,TV,Weikard1999,Whittaker-Watson} Let $n_{k}\in\mathbb{N}$,
$k\in\{0,1,2,3\}$ with $\underset{0\leqslant k\leqslant3}{max}n_{k}\geqslant
1$. The Darboux-Treibich-Verdier potential $\sum_{k=0}^{3}n_{k}(n_{k}%
+1)\wp(z+\frac{\omega_{k}}{2};\tau)+T$ is a finite-gap potential in the sense
of Definition \ref{def finite-gap}. In particular, the classical Lam\'{e}
potential $n(n+1)\wp(z;\tau)+T$ with $n\in\mathbb{Z}_{>0}$ is a finite-gap
potential in the sense of Definition \ref{def finite-gap}.
\end{example}

In this section, we apply Theorem \ref{Main Theorem1} to establish the
finite-gap property of the potential $q(z;T)$ for the generalized
Lam\'{e}-type equation (\ref{GLE 4})$_{p}$. The main result is as follows.

\begin{theorem}
\label{thm finite gap} Let $\tau\in\mathbb{H}.$ Suppose $p$ satisfies
(\ref{Assumption 1}). Then $q(z;T)$ is a finite-gap potential in the sense of
Definition \ref{def finite-gap}. More precisely, for each $j\in\{1,2\}$, the
following statements hold.

\begin{itemize}
\item[(i)] The spectral set $\sigma_{j}$ admits the decomposition
\[
\sigma_{j}=\left(  \bigcup_{k=1}^{g_{j}}\sigma_{j,k}\right)  \cup\left(
\bigcup_{k=1}^{\tilde{g}_{j}}\sigma_{j,k,\infty}\right)  ,\quad g_{j}%
\in\{1,2,3\},\quad\tilde{g}_{j}\in\{1,2\},
\]
where

\item each $\sigma_{j,k}$ is a bounded simple analytic arc;

\item each $\sigma_{j,k,\infty}$ is a semi-infinite simple analytic arc
tending to $\infty$;

\item the set $\sigma_{j}$ is summetric with respect to the origin:
$\sigma_{j}=-\sigma_{j}$.

\item[(ii)] All finite endpoints of $\sigma_{j}$ are given by the zeros of
$\mathit{spectral\ polynomial}$ $Q(T)$ $(\text{The explicit expression of
}Q(T)\text{ is given in }\eqref{spectral polynomial in noneven case})$. In
particular,
\[
\left\{  T\in\sigma_{j}|\,T\text{ is an endpoint}\right\}  =\left\{
T\in\mathbb{C}|\,Q(T)=0\right\}  \setminus\{0\}.
\]


\end{itemize}

Here $T=0$ might be a double zero of $Q(T)$ in the degenerate case.
\end{theorem}

We first recall the spectral theory of classical Lam\'{e} equation. Consider
the following Hill's equation
\begin{equation}
y^{\prime\prime}(x)=(2\wp(x+x_{0})+\tilde{B})y(x),\quad x\in\mathbb{R},
\label{lame on R}%
\end{equation}
where $\wp(x+x_{0})$ is continuous on both $\mathbb{R}$ and $\tau\mathbb{R}$.
Let $L(x)$ be a fundamental system of solution of \eqref{lame on R}. Then also
$L(x+\omega_{j})$ is a fundamental system, so there exists a monodromy matrix
$M_{j}(\tilde{B})$ with
\[
L(x+\omega_{j})=M_{j}(\tilde{B})L(x),\quad j=1,2.
\]

The Hill discriminant for each $j$ of \eqref{lame on R} is defined as
\[
\tilde{\Delta}_{j}(\tilde{B}):=\frac{1}{2}\mathrm{tr}M_{j}(\tilde{B}),
\]
which is an entire function in $\tilde{B}$. The asssociated
$\mathit{spectral\ set}$ of \eqref{lame on R} is then
\begin{equation}
\tilde{\sigma}_{j}:=\tilde{\Delta}_{j}^{-1}([-1,1])=\{\tilde{B}\in
\mathbb{C}|\,-1\leqslant\tilde{\Delta}_{j}(\tilde{B})\leqslant1\},\quad j=1,2.
\label{spectrum lame}%
\end{equation}

It is known that the Lam\'{e} potential $2\wp(z)+\tilde{B}$ is an
$\mathit{algebro}$-$\mathit{geometric}$ $\mathit{finite}$-$\mathit{gap}$
potential, equivalently a solution of the stationary KdV hierarchy. Its
spectral polynomial $Q(\tilde{B})$ is
\[
Q(\tilde{B})=-4\prod_{j=1}^{3}(\tilde{B}-e_{j}).
\]
The following fundamental result was proved in \cite{GW1}.

\noindent\textbf{Theorem C.}\textit{ (\cite{GW1})} The spectral set
$\tilde{\sigma}_{j}$ consists of one bounded spectral arc $\tilde{\sigma
}_{j,1}$ and one semi-infinite arc $\tilde{\sigma}_{j,\infty}$ which tends to
$-\infty+\langle q_{j}\rangle$, with
\[
\langle q_{j}\rangle=\frac{1}{\omega_{j}}\int_{\tilde{x}_{0}}^{\tilde{x}%
_{0}+\omega_{j}}2\wp(x)dx.
\]
\textit{That is,}%
\[
\tilde{\sigma}_{j}=\tilde{\sigma}_{j,1}\cup\tilde{\sigma}_{j,\infty},
\]
\textit{and the finite endpoints of these arcs coincide precisely with the
zeros of the spectral polynomial }$Q(\tilde{B})$\textit{.\medskip}

We now turn to the investigation of the finite-gap property of the generalized
Lam\'{e}-type potential $q(z;T)$.

\begin{lemma}
\label{spectrum +-symmetric} For potential $q(z;T)$, the Hill discriminant
satisfies $\Delta_{j}(T)=\Delta_{j}(-T)$. Consequently, $\sigma_{j}%
=-\sigma_{j}$ is symmetric with respect to the origin.
\end{lemma}

\begin{proof}
Firstly we observe that $q(-z;T)=q(z;-T)$. Let $F_{0}(z)$ be a fundamental
system of solution to equation $y^{\prime\prime}(z)=q(z;T)y(z)$ such that
$F_{0}(z+\omega_{j})=M_{j}(T)F_{0}(z)$, then $F_{1}(z):=F_{0}(-z)$ solves
$y^{\prime\prime}(z)=q(z;-T)y(z)$ and satisfies $F_{1}(z+\omega_{j}%
)=M_{j}(T)^{-1} F_{1}(z)$. Since $M_{j}(T)\in\mathrm{SL}(2,\mathbb{C})$, then
we have
\[
\Delta_{j}(-T)=\frac{1}{2}\mathrm{tr}\,M_{j}(-T)=\frac{1}{2}\mathrm{tr}%
\,M_{j}(T)^{-1}=\frac{1}{2}\mathrm{tr}\,M_{j}(T)=\Delta_{j}(T).
\]

\end{proof}

Recall the definition of $d_{j}(T)$ from \eqref{degree of spectrum} and
define
\[
\tilde{d}_{j}(\tilde{B}):=\mathrm{ord}_{\tilde{B}}(\tilde{\Delta}_{j}%
^{2}(\cdot)-1)
\]
respectively. The following lemma establishes the relationship between
$d_{j}(T)$ and $\tilde{d}_{j}(\tilde{B})$.

\begin{lemma}
\label{lemma : two degree} Adapt notations above. Assume $T$ and $\tilde{B}%
\in\mathbb{C}$ satisfy \eqref{correspondence}. Then
\begin{equation}
\label{degree's relation}d_{j}(T)=
\begin{cases}
\tilde{d}_{j}(\tilde{B})\quad\text{if }T\neq0;\\
2\tilde{d}_{j}(\tilde{B})\quad\text{if }T=0.
\end{cases}
\end{equation}

\end{lemma}

\begin{proof}
By Theorem \ref{Main Theorem1}, we have
\begin{equation}
\label{Delta T and Delta B}\Delta_{j}(T) = \frac{1}{2}\mathrm{tr}M_{j}(T) =
\frac{1}{2}\mathrm{tr}M_{j}(\tilde{B}) = \tilde{\Delta}_{j}(\tilde{B}).
\end{equation}
Let $\tilde{B}_{0}=T_{0}^{2}-2\wp(p)$. Consider the local behavior of
$\tilde{\Delta}_{j}$ at $\tilde{B}_{0}$:
\[
\tilde{\Delta}_{j}(\tilde{B})-\tilde{\Delta}_{j}(\tilde{B}_{0})=l(\tilde
{B}-\tilde{B}_{0})^{k}(1+O(\tilde{B}-\tilde{B}_{0})),\quad l\neq
0,\ k\in\mathbb{N}.
\]
Using \eqref{Delta T and Delta B}, we obtain
\begin{align*}
\Delta_{j}(T)-\Delta_{j}(T_{0})  &  =l(T^{2}-T_{0}^{2})^{k}+\cdots\\
&  =
\begin{cases}
2lT_{0}(T-T_{0})^{k}(1+O(T-T_{0}))\quad & \text{if }T_{0}\neq0;\\
lT^{2k}(1+O(T)) & \text{if }T_{0}=0,
\end{cases}
\end{align*}
which implies \eqref{degree's relation} holds.
\end{proof}

Recall the spectral polynomial $Q(T)$ in
\eqref{spectral polynomial in noneven case}:
\[
Q(T)=-4\prod_{k=1}^{3}\left(  T^{2}-2\wp(p)-e_{k}\right)  .
\]
The discriminant of $Q(T)$ is given by
\[
D=\prod_{k=1}^{3}\left(  2\wp(p)+e_{k}\right)  .
\]
Thus, $Q(T)$ admits a multiple root at $T=0$ if and only if $\wp
(p)=-\frac{e_{k}}{2}$ for some $k\in\{1,2,3\}$. These observations lead
directly to the following result on the endpoints of $\sigma_{j}$.

\begin{theorem}
\label{thm :endpoints of spectrum} All endpoints of $\sigma_{j}$ are zeros of
$Q(T)$. In particular,

(i) If $Q(0)\not =0$, then%
\[
\left\{  T\in\sigma_{j}|\,T\text{ is a finite endpoint}\right\}  =\left\{
T\in\mathbb{C}|\,Q(T)=0\right\}  .
\]

(ii) If $Q(0)=0$, then%
\[
\left\{  T\in\sigma_{j}|\,T\text{ is a finite endpoint}\right\}  =\left\{
T\in\mathbb{C}|\,Q(T)=0\right\}  \setminus\{0\}
\]
Moreover, every endpoint $T$ of $\sigma_{j}$ satisfies $d_{j}(T)=1$.
\end{theorem}

\begin{proof}
From Theorem \ref{Main Theorem1}, we have

Suppose $T_{0}$ is an endpoint of $\sigma_{j}$. We claim that $\tilde{B}%
_{0}=T_{0}^{2}-2\wp(p)$ is also an endpoint of $\tilde{\sigma}_{j}$. By Lemma
\ref{lemma: odd degree } and Lemma \ref{lemma : two degree}, we have
$d_{j}(T_{0})$ is an odd integer and
\[
\tilde{d}_{j}(\tilde{B}_{0})=d_{j}(T_{0}).
\]
Hence, Lemma \ref{lemma: odd degree } implies that $\tilde{B}_{0}$ is an
endpoint of $\tilde{\sigma}_{j}$. By Theorem C, it follows that $\tilde{B}%
_{0}\in\{e_{1},e_{2},e_{3}\}$, which is equivalent to
\[
T_{0}^{2}\in\{2\wp(p)+e_{j}|\,j=1,2,3\}.
\]
Therefore, $Q(T_{0})=0$ by \eqref{spectral polynomial in noneven case}.

Conversely, suppose $T$ is a zero of $Q(T)$. Then $\tilde{B}=T^{2}-2\wp(p)$
satisfies $\tilde{Q}(\tilde{B})=0$. By Theorem C, we have $\tilde{d}%
_{j}(\tilde{B})=1$.

(i) If $T\neq0$, then Lemma \ref{lemma : two degree} gives
\[
d_{j}(T)=\tilde{d}_{j}(\tilde{B})=1,
\]
so by Lemma \ref{lemma: odd degree }, $T$ is an endpoint of $\sigma_{j}$.
\textit{\medskip}

(ii)If $T=0$, then Lemma \ref{lemma : two degree} implies
\[
d_{j}(T)=2\tilde{d}_{j}(\tilde{B})=2,
\]
which is even, and hence $T$ is not an endpoint of $\sigma_{j}$ by Lemma
\ref{lemma: odd degree }.

This completes the proof.
\end{proof}

\begin{proof}
[Proof of Theorem \ref{thm finite gap}]Assume $T$ and $\tilde{B}$ are related
by \eqref{correspondence}. Combining this with \eqref{spectrum lame-type} and
\eqref{spectrum lame}, we obtain
\begin{equation}
\label{sigma-describe}\sigma_{j}=\Delta_{j}^{-1}([-1,1])=\{T|\, T^{2}%
-2\wp(p)\in\tilde{\sigma}_{j}\}.
\end{equation}

Therefore, Theorem \ref{thm finite gap} follows directly from Theorem C ,
Lemma \ref{spectrum +-symmetric} and Theorem \ref{thm :endpoints of spectrum}.
\end{proof}

\subsection{Deformation of Spectral Sets as \textit{p} Varies}


Let $\tau\in i\mathbb{R}_{>0}$. Assume
\[
p\in(0,\frac{\tau}{2}]\cup\lbrack\frac{\tau}{2},\frac{1+\tau}{2}]\cup
\lbrack\frac{1+\tau}{2},\frac{1}{2}]\cup\lbrack\frac{1}{2},0),
\]
which is equivalent to requiring$\wp(p)\in(-\infty,+\infty)$. In this
subsection, we provide a complete description of the spectral sets $\sigma
_{1}$ and $\sigma_{2}$ for generic $p$. In particular, we prove that, as
$\wp(p)$ varies along the real axis, $\sigma_{j}$ exhibits exactly seven
distinct types of graphs. The main theorems are stated below.


\begin{figure}[htbp] 
		\centering 
		\includegraphics[width=0.8\textwidth]{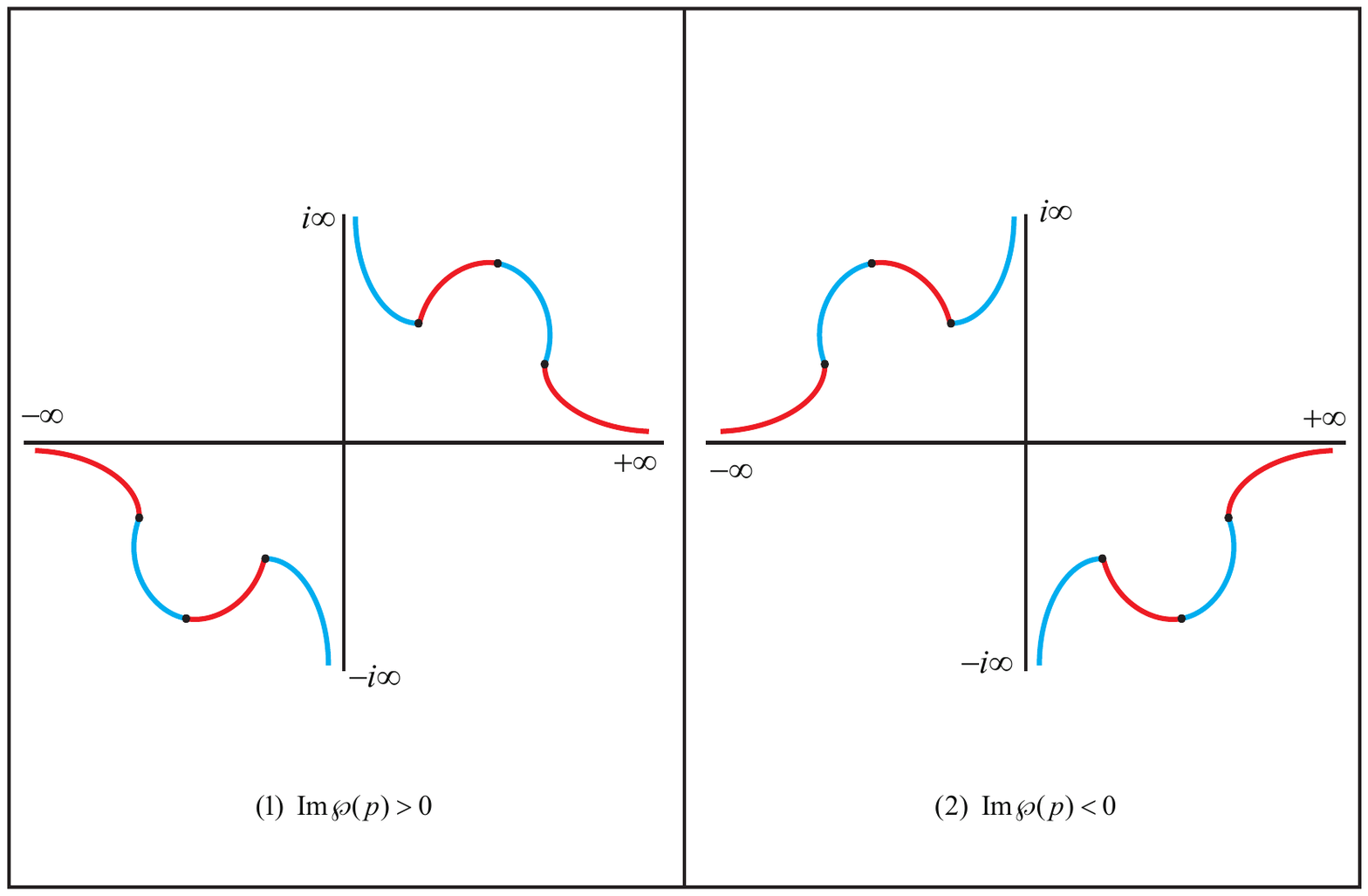} 
		\caption{Classification graphs of  of the spectral sets $\sigma_{j}$ as $\wp(p)\notin\mathbb{R}$. The blue curves denote $\sigma_{1}$ and the red curves denote $\sigma_{2}$. The dark points denote zeros of spectral polynomial $Q(T)$. }
	\label{fig:graphs of generic p} 
		
\end{figure}

\begin{theorem}
\label{thm: spectrum of lame-type, generic p} Fix $\tau\in i\mathbb{R}_{>0}$.
Suppose $\wp(p)\notin\mathbb{R}$. Then the spectral sets $\sigma_{j}$
decompose as
\[
\sigma_{j}=\left(  \bigcup_{k=1}^{2}\sigma_{j,k}\right)  \cup\left(
\bigcup_{k=1}^{2}\sigma_{j,k,\infty}\right)  ,\quad j=1,2,
\]
where each $\sigma_{j,k}$ is a bounded simple analytic arc and each
$\sigma_{j,k,\infty}$ is a semi-infinite simple analytic arc. Moreover, the
following properties hold:

\begin{itemize}
\item For $j=1,2,$ the spectral sets $\sigma_{j}$ do not intersect either the
real or the imaginary axis, namely,
\[
\sigma_{j}\cap\mathbb{R}=\sigma_{j}\cap i\mathbb{R}=\emptyset;
\]

\item Symmetry: $\sigma_{j,1}=-\sigma_{j,2}$, $\sigma_{j,1,\infty}%
=-\sigma_{j,2,\infty}$;

\item Asymptotics: the infinite endpoints of $\sigma_{1,k,\infty}$ tend to
$\pm i\infty$, while the infinite endpoints of $\sigma_{2,k,\infty}$ tend to
$\pm\infty$, respectively.
\end{itemize}

More precisely, let $\sigma_{j}=\sigma_{j}^{\prime}\cup-\sigma_{j}^{\prime}$,
where $\mathrm{Im}\sigma_{j}^{\prime}>0$. Then

\begin{itemize}
\item[(1)] If $\mathrm{Im}\wp(p)>0$, then we have $\mathrm{Re}\sigma
_{j}^{\prime}>0$;

\item[(2)] If $\mathrm{Im}\wp(p)<0$, then we have $\mathrm{Re}\sigma
_{j}^{\prime}<0$.
\end{itemize}
\end{theorem}


\begin{proof}
Lemma \ref{spectrum +-symmetric} implies that $\sigma_{j}=-\sigma_{j}$. Since
$\wp(p)\notin\mathbb{R}$, then $\wp(p)\neq-\frac{e_{j}}{2}$, which implies
$Q(0)\neq0$. By Theorem \ref{thm :endpoints of spectrum}, the finite endpoints
of $\sigma_{j}$ coincide with the zeros of $Q(T)$, namely the set $\{\pm
\sqrt{2\wp(p)+e_{j}},j=1,2,3\}$, and satisfy
\[
d_{j}\left(  \pm\sqrt{2\wp(p)+e_{j}}\right)  =1,\quad j=1,2,3.
\]
On the other hand, from Theorem D and \eqref{sigma-describe}, we have
\[
\lim\limits_{T\in\sigma_{1,k,\infty},T\to\infty}T^{2}-2\wp(p)=-\infty
\]
and
\[
\lim\limits_{T\in\sigma_{2,k,\infty},T\to\infty}T^{2}-2\wp(p)=+\infty,
\]
Then we have
\begin{equation}
\label{sigma1 infty}\lim\limits_{T\in\sigma_{1,k,\infty},T\to\infty}T=\pm
i\infty\
\end{equation}
and
\begin{equation}
\label{sigma2 infty}\lim\limits_{T\in\sigma_{2,k,\infty},T\to\infty}%
T=\pm\infty,
\end{equation}
respectively. From the above discussions, we conclude that

\begin{itemize}
\item $\sigma_{j}$ consists of two bounded spectral arcs $\sigma_{j,k}$ and
two infinte spectral arcs $\sigma_{j,k,\infty}$, $k=1,2$;

\item The infinite endpoints of $\sigma_{1,k,\infty}$ tend to $\pm i\infty$,
and the infinite endpoints of $\sigma_{2,k,\infty}$ tend to $\pm\infty$;

\item $\sigma_{j,1}=-\sigma_{j,2}$ and $\sigma_{j,1,\infty}=-\sigma
_{j,2,\infty}$.
\end{itemize}

Now assume $\mathrm{Im}\,\wp(p)>0$ and let $\sigma_{j}=\sigma_{j}^{\prime}%
\cup-\sigma_{j}^{\prime}$, where $\mathrm{Im}\,\sigma_{j}^{\prime}>0$. Then
$\mathrm{Im}\,(2\wp(p)+e_{k})>0$. By Theorem \ref{thm :endpoints of spectrum},
any finite endpoint $T$ of $\sigma_{j}^{\prime}$ satisfies
\[
\mathrm{Im}\,T^{2}=\mathrm{Im}\,(2\wp(p)+e_{k})>0,\quad\mathrm{Im}\,T>0,
\]
which implies $\mathrm{Re}\,T>0$. Since $\sigma_{j}^{\prime}\cap
\mathbb{R}=\emptyset=\sigma_{j}^{\prime}\cap i\mathbb{R}$, we conclude that
$\mathrm{Re}\,\sigma_{j}^{\prime}>0$. The same arguments applies to the case
$\mathrm{Im}\wp(p)<0$. This completes the proof.
\end{proof}

\begin{figure}[htbp] 
	\centering 
	\includegraphics[width=1.0\textwidth]{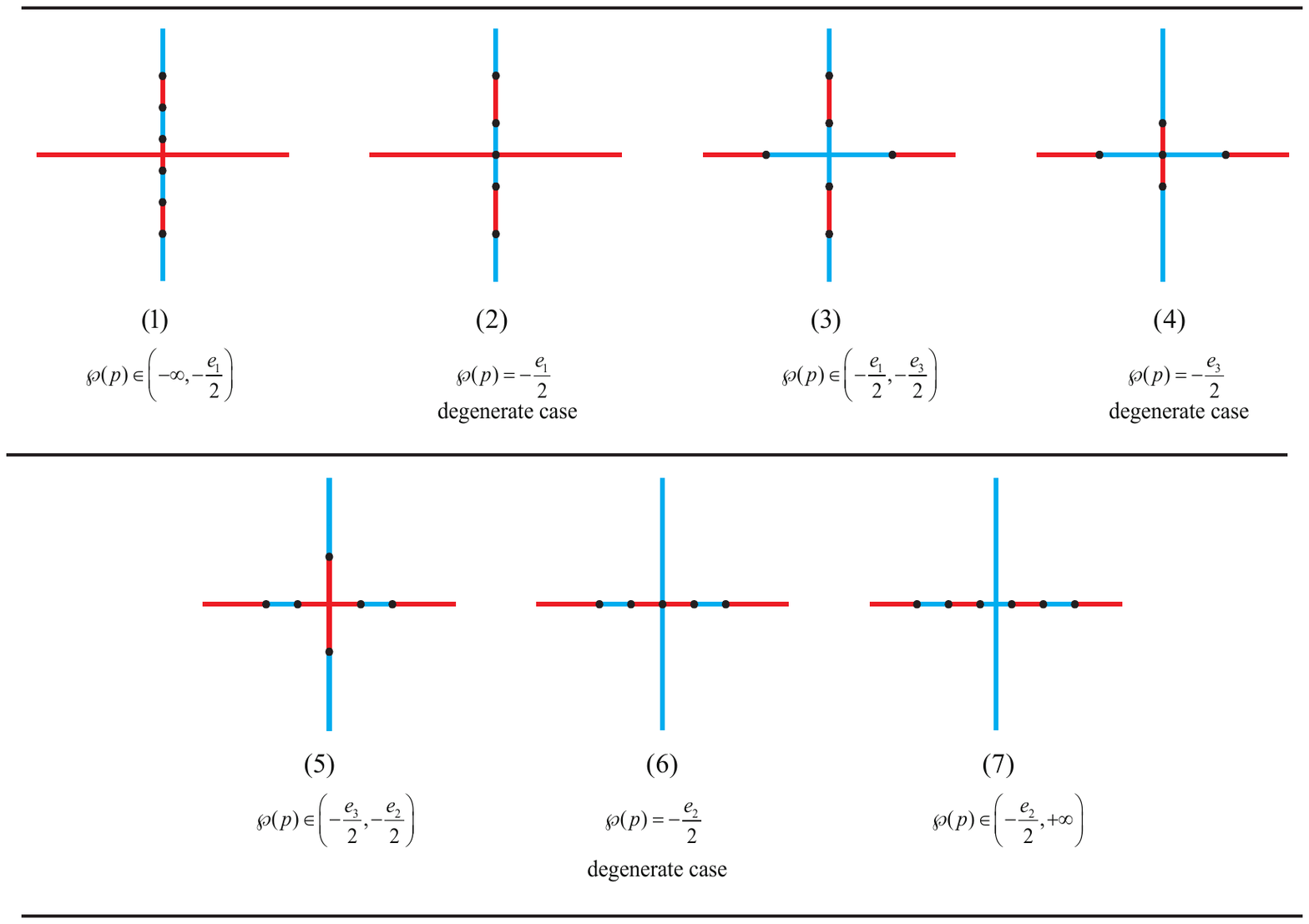} 
	\caption{Seven schematic graphs of the continuous deformation of the spectral sets $\sigma_{j}$ as stated in Theorem \ref{thm: spectrum of lame-type}. The blue curves denote $\sigma_{1}$ and the red curves denote $\sigma_{2}$. The dark points denote zeros of spectral polynomial $Q(T)$. }
	\label{fig:graphs of spectrum} 
\end{figure}

\begin{theorem}
\label{thm: spectrum of lame-type}

Fix $\tau\in i\mathbb{R}_{>0}$ and let $\wp(p)\in\mathbb{R}$. Then the
spectral arcs $\sigma_{j}$ lie in $\mathbb{R}\cup i\mathbb{R}$ and satisfy
\[
\sigma_{1}\cup\sigma_{2}=\mathbb{R}\cup i\mathbb{R},\quad\sigma_{1}\cap
\sigma_{2}=\{T\in\mathbb{C}|\,Q(T)=0\}.
\]
Moreover, the following statements hold.

\item[(1)] If $\wp(p)\in\left(  -\infty,-\frac{e_{1}(\tau)}{2}\right)  $,
then
\begin{align*}
\sigma_{1}  &  =\left(  -i\infty,-i\sqrt{-(2\wp(p)+e_{2}(\tau))}\right] \\
&  \cup\left[  -i\sqrt{-(2\wp(p)+e_{3}(\tau))},-i\sqrt{-(2\wp(p)+e_{1}(\tau
))}\right] \\
&  \cup\left[  i\sqrt{-(2\wp(p)+e_{1}(\tau))},i\sqrt{-(2\wp(p)+e_{3}(\tau
))}\right] \\
&  \cup\left[  i\sqrt{-(2\wp(p)+e_{2}(\tau))}, i\infty\right)  .
\end{align*}

\item[(2)] If $\wp(p)=-\frac{e_{1}(\tau)}{2}$, then
\begin{align*}
\sigma_{1}  &  =\left(  -i\infty,-i\sqrt{-(2\wp(p)+e_{2}(\tau))}\right] \\
&  \cup\left[  -i\sqrt{-(2\wp(p)+e_{3}(\tau))},i\sqrt{-(2\wp(p)+e_{3}(\tau
))}\right] \\
&  \cup\left[  i\sqrt{-(2\wp(p)+e_{2}(\tau))}, i\infty\right)  .
\end{align*}

\item[(3)] If $\wp(p)\in\left(  -\frac{e_{1}(\tau)}{2},-\frac{e_{3}(\tau)}%
{2}\right)  $, then
\begin{align*}
\sigma_{1}  &  =\left(  -i\infty,-i\sqrt{-(2\wp(p)+e_{2}(\tau))}\right] \\
&  \cup\left[  -i\sqrt{-(2\wp(p)+e_{3}(\tau))},i\sqrt{-(2\wp(p)+e_{3}(\tau
))}\right] \\
&  \cup\left[  i\sqrt{-(2\wp(p)+e_{2}(\tau))},i\infty\right) \\
&  \cup\left[  -\sqrt{2\wp(p)+e_{1}(\tau)},\sqrt{2\wp(p)+e_{1}(\tau)}\right]
.
\end{align*}

\item[(4)] If $\wp(p)=-\frac{e_{3}(\tau)}{2}$, then
\begin{align*}
\sigma_{1}  &  =\left(  -i\infty,-i\sqrt{-(2\wp(p)+e_{2}(\tau))}\right] \\
&  \cup\left[  i\sqrt{-(2\wp(p)+e_{2}(\tau))}, i\infty\right) \\
&  \cup\left[  -\sqrt{2\wp(p)+e_{1}(\tau)},\sqrt{2\wp(p)+e_{1}(\tau)}\right]
.
\end{align*}

\item[(5)] If $\wp(p)\in\left(  -\frac{e_{3}(\tau)}{2},-\frac{e_{2}(\tau)}%
{2}\right)  $, then
\begin{align*}
\sigma_{1}  &  =\left(  -i\infty,-i\sqrt{-(2\wp(p)+e_{2}(\tau))}\right] \\
&  \cup\left[  i\sqrt{-(2\wp(p)+e_{2}(\tau))}, i\infty\right) \\
&  \cup\left[  -\sqrt{2\wp(p)+e_{1}(\tau)},-\sqrt{2\wp(p)+e_{3}(\tau)}\right]
\\
&  \cup\left[  \sqrt{2\wp(p)+e_{3}(\tau)},-\sqrt{2\wp(p)+e_{1}(\tau)}\right]
.
\end{align*}

\item[(6)] If $\wp(p)=-\frac{e_{2}(\tau)}{2}$, then
\begin{align*}
\sigma_{1}  &  =i\mathbb{R}\cup\left[  -\sqrt{2\wp(p)+e_{1}(\tau)},-\sqrt
{2\wp(p)+e_{3}(\tau)}\right] \\
&  \cup\left[  \sqrt{2\wp(p)+e_{3}(\tau)},-\sqrt{2\wp(p)+e_{1}(\tau)}\right]
.
\end{align*}

\item[(7)] If $\wp(p)\in\left(  -\frac{e_{2}(\tau)}{2},+\infty\right)  $,
then
\begin{align*}
\sigma_{1}  &  =i\mathbb{R}\cup\left[  -\sqrt{2\wp(p)+e_{1}(\tau)},-\sqrt
{2\wp(p)+e_{3}(\tau)}\right] \\
&  \cup\left[  -\sqrt{2\wp(p)+e_{2}(\tau)},\sqrt{2\wp(p)+e_{2}(\tau)}\right]
\\
&  \cup\left[  \sqrt{2\wp(p)+e_{3}(\tau)},-\sqrt{2\wp(p)+e_{1}(\tau)}\right]
.
\end{align*}

\end{theorem}

It is known that
\[
e_{1}(\tau)>e_{3}(\tau)>e_{2}(\tau) \quad\text{for }\tau\in i\mathbb{R}_{>0}.
\]
Recall the following important conclusion concerning the spectral sets of
Lam\'{e} equation \eqref{lame on R}.

\noindent\textbf{Theorem D. }\cite{Chen-Lin-sharp nonexistence} Assume
$\tau\in i\mathbb{R}_{>0}$. Then the spectral sets $\tilde{\sigma}_{j}$ of
Lam\'{e} equation \eqref{lame on R} satisfy

\begin{itemize}
\item[(i)] $\tilde{\sigma}_{1}=(-\infty,e_{2}(\tau)]\cup[e_{3}(\tau
),e_{1}(\tau)]$;

\item[(ii)] $\tilde{\sigma}_{2}=[e_{2}(\tau),e_{3}(\tau)]\cup[e_{1}%
(\tau),+\infty)$.
\end{itemize}

Consequently, the monodromy representation of classical Lam\'{e} equation
\eqref{Lame 1} cannot be unitary for any $\tilde{B}\in\mathbb{C}$.

\begin{proof}
[Proof of Theorem \ref{thm: spectrum of lame-type}]Since $\wp(p)\in\mathbb{R}%
$, any root of $Q(T)$ is either real or purely imaginary. By Theorem C, we see
that
\[
\tilde{\sigma}_{1}\cup\tilde{\sigma}_{2}=\mathbb{R}\quad\text{and}\quad
\tilde{\sigma}_{1}\cap\tilde{\sigma}_{2}=\{e_{k},k=1,2,3\}.
\]
Then \eqref{sigma-describe} implies
\[
\sigma_{1}\cup\sigma_{2}=\mathbb{R}\cup i\mathbb{R}\quad\text{and}\quad
\sigma_{1}\cap\sigma_{2}=\{T|\,Q(T)=0\}.
\]
Analogous to the proof of Theorem \ref{thm: spectrum of lame-type, generic p},
we also have \eqref{sigma1 infty} and \eqref{sigma2 infty} hold. Together with
Lemma \ref{spectrum +-symmetric}, we conclude that

\begin{itemize}
\item $\sigma_{j}$ is symmetric with respect to $\mathbb{R}$ and $i\mathbb{R}$;

\item The infinite spectral arcs of $\sigma_{1}\ (\sigma_{2},\text{
respectively})$ tend to $\pm i\infty\ (\pm\infty)$.
\end{itemize}

Then for each case, it suffices to compute $d_{j}$ at endpoints.

\begin{itemize}
\item[(1)] It follows from $\wp(p)\in\left(  -\infty,-\frac{e_{1}(\tau)}%
{2}\right)  $ that all roots of $Q(T)$ are simple and purely imaginary. Then
we obtain from Theorem \ref{thm :endpoints of spectrum} that
\[
d_{1}(\pm\sqrt{2\wp(p)+e_{k}})=1,\quad k=1,2,3.
\]
This proves Theorem \ref{thm: spectrum of lame-type} (1). The same argument
applies to (3), (5) and (7).

\item[(2)] $\wp(p)=-\frac{e_{1}(\tau)}{2}$ implies $Q(0)=0$. From here and
Theorem \ref{thm :endpoints of spectrum}, we immediately obtain
\[
d_{1}(0)=2,\quad d_{1}(\pm\sqrt{2\wp(p)+e_{k}})=1,\quad k=2,3.
\]
This proves Theorem \ref{thm: spectrum of lame-type} (2). The same argument
applies to (4) and (6). This completes the proof.
\end{itemize}
\end{proof}

\begin{corollary}
For $\tau\in i\mathbb{R}_{>0}$ , we have $\sigma_{1}\cap\sigma_{2}%
\setminus\{T|\,Q(T)=0\}=\emptyset$. Then the monodromy representation of the
generalized Lam\'{e}-type equation (\ref{GLE 4})$_{p}$ cannot be unitary for
any $T\in\mathbb{C}$.
\end{corollary}

\begin{proof}
Note that the monodromy of the generalized Lam\'{e}-type equation
(\ref{GLE 4})$_{p}$ is unitary if and only if monodromy data $(r,s)\in
\mathbb{R}^{2}\setminus\frac{1}{2}\mathbb{Z}^{2}$, which is equivalent to
\[
T\in\sigma_{1}\cap\sigma_{2}\setminus\{T|\,Q(T)=0\}.
\]
Together Theorem \ref{thm: spectrum of lame-type, generic p} and Theorem
\ref{thm: spectrum of lame-type}, we conclude that $\sigma_{1}\cap\sigma
_{2}\setminus\{T|\,Q(T)=0\}=\emptyset$. Therefore, the monodromy cannot be
unitary for any $T\in\mathbb{C}$.
\end{proof}


\section{Cone Spherical Metrics}

\label{Cone Spherical Metrics Section}

\subsection{Existence of Cone Spherical Metrics}

\label{Main results for cone spherical metrics}

Recently, cone spherical metrics have been studied extensively; see \cite{CLW,
CCChen-Lin 1,Chen-Lin-sharp nonexistence,Chen-Kuo-Lin-8pi+8pi, Bergweiler-Eremenko-dynamics, Chen-Fu-Lin-Hitchin,
Chen-Kuo-Lin-16pi, CWWXu, Eremenko-.Gabrielov-on metrics,
Eremenko-Mondello-Panov, EGMP, LW-AnnMath, LW, LSXu} for significant
contributions in this direction. Building upon the monodromy equivalence
established in Theorem \ref{Main Theorem1}, we now apply these results to
construct cone spherical metrics with prescribed conical singularities. We
also examine related analytical aspects, highlighting how the integrable
structure of the generalized Lam\'{e}-type equation naturally governs the
geometry of such metrics.

The generalized Lam\'{e}-type equation (\ref{GLEn}) is closely related to the
following curvature equation (PDE):
\begin{equation}
\Delta u+e^{u}=8\pi n\delta_{0}+4\pi(\delta_{p}+\delta_{-p})\text{ on }%
E_{\tau}\text{,} \label{Curvature n}%
\end{equation}
which arises in the study of spherical metrics with conical singularities.

In conformal geometry, it corresponds to the problem of finding a conformal
metric of constant positive curvature on the elliptic curve $E_{\tau}$, with
conical singularities of cone angle $2\pi(2n+1)$ at $0$, and and cone angle
$4\pi$ at $\pm p$.

The approach to study the curvature equation (\ref{Curvature n}) with an odd
total angle relies on its integrability via Liouville Theorem, which states
that any solution $u(z)$ of equation (\ref{Curvature n}) is given as the form
\begin{equation}
u(z)=\log\frac{8|f^{\prime}(z)|^{2}}{(1+|f(z)|^{2})^{2}}. \label{502}%
\end{equation}
See \cite{CLW,Prajapat -Tarantello} for a proof. The function $f(z)$ is
locally meromorphic and is commonly referred to as the developing map in the literature.

Through the representation (\ref{502}), the problem of solving the curvature
equation (\ref{Curvature n})$_{p}$ reduces to studying the generalized
Lam\'{e}-type equation (\ref{GLEn})$_{p}$, with parameters, $\mathbf{T}$, $B$
satiafy the apparency condition, such that the corresponding monodromy
matrices $M_{i}(\mathbf{T;}p),$ for $i$ $=$ $1,2,$ are \textbf{unitary} up to
conjugation. The correspondence between the nonlinear curvature equation and
the linear generalized Lam\'{e}-type equation---namely, the transition from
PDE to ODE and vice versa---has become a well-established and widely used
approach in the study of spherical metrics with conical singularities.

A notable special case is the curvature equation
\begin{equation}
\Delta v+e^{v}=8\pi n\delta_{0}\text{ on }E_{\tau} \label{Curvature-Lame n}%
\end{equation}
which is associated with the classical Lam\'{e} equation (\ref{Lame n}). The
correspondence between these two equations has been established in detail in
the seminal works \cite{CLW, LW}.\textit{\medskip}

We briefly summarize the above discussion as follows.\textit{\medskip}

\noindent\textbf{Theorem E. }\textit{Assume (\ref{Assumption 1}). Then the
curvature equation associated with the singularity position }$p$\textit{,
denoted by (\ref{Curvature n})}$_{p}$\textit{ admits a solution if and only if
there exists parameters }$\mathbf{T}$\textit{, }$B\in\mathbb{C}$ \textit{such
that the corresponding generalized Lam\'{e}-type equation (\ref{GLEn})}$_{p}
$\textit{ is completely reducible, with associated monodromy data }$\left(
r,s\right)  $ $\in$ $\mathbb{R}^{2}$ $\setminus$ $\frac{1}{2}\mathbb{Z}^{2}$.

\begin{proof}
The proof of Theorem C. follows directly from the case of the model equation
(\ref{Curvature-Lame n}) and can be obtained by a straightforward
modification; we therefore omit the details. See \cite{CLW, LW, Kuo} for the proofs.
\end{proof}

It is also known that, if $u(z)$ of the form (\ref{502}) is a solution, then
the one-parameter family of functions
\begin{equation}
u_{\beta}(z)=\log\frac{8e^{2\beta}\left\vert f^{\prime}\right\vert ^{2}%
}{\left(  1+e^{2\beta}\left\vert f\right\vert ^{2}\right)  ^{2}},\text{ }%
\beta\in\mathbb{R} \label{family}%
\end{equation}
also consists of solutions to the same equation. This family of solutions
$\{u_{\beta}(z)|\beta$ $\in$ $\mathbb{R}\}$ exhibits blow-up behavior as
described below:

\begin{itemize}
\item As $\beta\rightarrow+\infty$, it blows up at \textbf{zeros} of the
developing map $f(z)$; while

\item As $\beta\rightarrow-\infty$, it blows up at \textbf{poles} of the
developing map $f(z)$.
\end{itemize}

This behavior is characteristic of bubbling solutions in geometric analysis,
and is closely related to:

\begin{itemize}
\item The study of moduli spaces of flat connections or projective structures
on Riemann surfaces;

\item The analysis of holomorphic quadratic differentials and their associated
monodromy representations;

\item Mean field equations arising in statistical mechanics models and
Chern--Simons--Higgs theory.
\end{itemize}

In particular, the parameter $\beta$ $\in$ $\mathbb{R}$ can be interpreted as
a scaling parameter along a one-parameter family in the moduli space, where
the geometry of the solution degenerates into concentrated curvature
(delta-mass) configurations.

Suppose this family of solutions $\{u_{\beta}(z)|\beta$ $\in$ $\mathbb{R}\}$
contains at least one even solution. By scaling, we may assume that the
solution corresponding to $\beta=0$, denoted $u_{0}(z)$, is even. In this
case, it can be shown that $u_{0}(z)$ is the only even solution in the family.

Accordingly, we classify the family $\{u_{\beta}(z)|\beta$ $\in$
$\mathbb{R}\}$ as follows:

\begin{itemize}
\item \noindent If it contains an even solution (necessarily unique), it is
called an \textbf{even family of solutions}.

\item \noindent Otherwise, it is called a \textbf{non-even family of
solutions}.
\end{itemize}

We remark that for the curvature equation (\ref{Curvature-Lame n}), due to the
presence of only a single conical singularity at the origin, every solution
family $\{u_{\beta}(z)|\beta$ $\in$ $\mathbb{R}\}$ must necessarily be even.

It is straightforward to verify that

\begin{enumerate}
\item Every even family of solutions arises from the generalized Lam\'{e}-type
equation (\ref{GLE 2}) in the even symmetry case, for some $A$ and $B$
satisfying (\ref{APeven}); while

\item Every non-even family of solutions corresponds to the generalized
Lam\'{e}-type equation (\ref{GLE 3}) in the non-even symmetry case, for some
$T\not =0$ and $B$ satisfying (\ref{APnoneven}).
\end{enumerate}

Now, let us focus on the fundamental case $n=1$. In this setting, the two
curvature equations take the form:
\begin{equation}
\Delta v+e^{v}=8\pi\delta_{0}\text{ on }E_{\tau}, \label{Curvature-Lame 1}%
\end{equation}
and
\begin{equation}
\Delta u+e^{u}=8\pi\delta_{0}+4\pi(\delta_{p}+\delta_{-p})\text{ on }E_{\tau
}\text{.} \label{Curvature 1}%
\end{equation}

Recall (\ref{Assumption 1}) and define
\begin{equation}
\Lambda_{even}^{(1)}=\left\{  (\tau,p)\in\mathbb{H\times}E_{\tau}\text{,
(\ref{Curvature 1})}_{\tau}\text{ admits an even family}\right\}  ,
\label{Even set}%
\end{equation}
and
\begin{equation}
\Lambda_{Noneven}^{(1)}=\left\{  (\tau,p)\in\mathbb{H\times}E_{\tau}\text{,
(\ref{Curvature 1})}_{\tau}\text{ admits a Non-even family}\right\}  .
\label{Noneven set}%
\end{equation}
Obviously, from the structure of the equation (\ref{Curvature 1}), the points
$\pm p$ must be identified in both sets. Henceforth, we always treat $p$ and
$-p$ as equivalent.

The main objective of this section is to characterize these two sets.

As a direct consequence of Theorems B and E, we obtain the following
characterization for even families to the curvature equation
(\ref{Curvature 1})$_{p}$.\textit{\medskip}

\noindent\textbf{Theorem F.(}\textit{\cite{Chen-Kuo-Lin-Painleve VI}}\textbf{)
}\textit{For each} $\tau\in\mathbb{H}$, \textit{the equation}
(\ref{Curvature 1})$_{p}$ \textit{has an even family if and only if the
singularity} $p$ \textit{is given by}
\[
p=p_{r,s}^{(1)}(\tau)\text{ for some }\left(  r,s\right)  \in\mathbb{R}%
^{2}\setminus\frac{1}{2}\mathbb{Z}^{2},
\]
\textit{where }$p_{r,s}^{(1)}(\tau)$\textit{ denotes the solution of the
Painlev\'{e} VI equation (\ref{EPVI 1})}$_{n=1}$\textit{, evaluated at }$\tau
$, \textit{corresponding to }$\left(  r,s\right)  $. \textit{Consequently,}
\begin{equation}
\Lambda_{even}^{(1)}=\left\{  (\tau,p)|\tau\in\mathbb{H}\text{, }%
p=p_{r,s}^{(1)}(\tau)\text{, }\left(  r,s\right)  \in\mathbb{R}^{2}%
\setminus\frac{1}{2}\mathbb{Z}^{2}\right\}  . \label{Even singular set}%
\end{equation}

In fact, for any $\left(  r,s\right)  \in\mathbb{C}^{2}\setminus\frac{1}
{2}\mathbb{Z}^{2}$, the solution $p_{r,s}^{(1)}(\tau)$ can be expressed
explicitly. To this end, for such a pair $(r,s)$, we introduce the fundamental
function
\begin{equation}
Z(r,s,\tau):=\zeta(r+s\tau)-r\eta_{1}(\tau)-s\eta_{2}(\tau), \label{zrs}%
\end{equation}
which depends meromorphically on $(r,s)$. Since $(r,$ $s)$ $\not \in \frac
{1}{2}\mathbb{Z}^{2}$, it is evident that $Z(r,s,\tau)$ $\not \equiv 0,\infty$
as a function of $\tau$, and is meromorphic in $\tau$.

This meromorphic function $Z(r,s,\tau)$ was initially introduced by Hecke in
\cite{Heck}. Hecke demonstrated that it is a modular form of weight one with
respect to $\Gamma(n)$ whenever $(r,$ $s)$ is an $n$-torsion point. For this
reason, $Z(r,s,\tau)$ is referred to as a premodular form.

The premodular form $Z(r,s,\tau)$ plays a central role in the monodromy
problem for the classical Lam\'{e} equation (\ref{Lame 1}) in completely
reducible case, as stated below.\textit{\medskip}

\noindent\textbf{Theorem G.}\textit{ (\cite{LW-AnnMath, LW}) Given any}
$(r,s)$ $\in\mathbb{C}^{2}\setminus\frac{1}{2}\mathbb{Z}^{2}$. \textit{There
exists }$\tilde{B}\in\mathbb{C}$ \textit{such that the classical Lam\'{e}
equation} (\ref{Lame 1})$_{\tau}$ \textit{is completely reducible and adimts
monodromy data} $\left(  r,s\right)  $\textit{ if and only if}
\begin{equation}
Z(r,s,\tau)=0. \label{Zero equation}%
\end{equation}
\textit{Consequently, the curvature equation (\ref{Curvature-Lame 1})}$_{\tau
}$\textit{ has an even family if and only if there exists }$(r,s)$\textit{
}$\in\mathbb{R}^{2}\setminus\frac{1}{2}\mathbb{Z}^{2}$\textit{ such that
(\ref{Zero equation}) above holds.\medskip}

We remark that more general premodular forms characterizing the monodromy
problem in the completely reducible case have been constructed for
(\ref{Lame n}) in \cite{LW}, and for the Treibich--Verdier generalization in
\cite{Chen-Kuo-Lin-Lame II}.

Let us denote
\[
Z=Z(r,s,\tau),\text{ }\wp=\wp(r+s\tau;\tau),\text{ }\wp^{\prime}=\wp^{\prime
}(r+s\tau;\tau).
\]
Then the solution $p_{r,s}^{(1)}(\tau)$ can be explicitly expressed by
\begin{equation}
\wp(p_{r,s}^{(1)}(\tau);\tau)=\wp+\frac{3\wp^{\prime}Z^{2}+(12\wp^{2}
-g_{2})Z+3\wp\wp^{\prime}}{2(Z^{3}-3\wp Z-\wp^{\prime})}. \label{PV6-1}%
\end{equation}
By (\ref{zrs}), we have
\begin{equation}
Z(r,s,\tau)=\pm Z(r^{\prime},s^{\prime},\tau),\text{ whenever }\left(
r^{\prime},s^{\prime}\right)  \equiv\pm\left(  r,s\right)  \text{ mod
}\mathbb{Z}^{2}. \label{Property1}%
\end{equation}
Consequently, by (\ref{PV6-1}) and the evenness of $\wp$, the solutions
satisfy
\begin{equation}
p_{r,s}^{(1)}(\tau)=\pm p_{r^{\prime},s^{\prime}}^{(1)}(\tau),
\label{Equivalent of PVI}%
\end{equation}
so $\pm p_{r^{\prime},s^{\prime}}^{(1)}(\tau)$ are identified as the same
solution for any $\left(  r^{\prime},s^{\prime}\right)  \equiv\pm\left(
r,s\right)  $ mod $\mathbb{Z}^{2}$.

We now turn our attention to the study of non-even families of solutions to
the curvature equation (\ref{Curvature 1})$_{p}$. Theorem E implies that
finding such a family is equivalent to finding a parameter $T\not =0$, with
$B$ determined by (\ref{APnoneven}), such that the equation (\ref{GLE 4}
)$_{p}$ has unitary monodromy, i.e., $(r,s)$ $\in$ $\mathbb{R}^{2}%
\setminus\frac{1}{2}\mathbb{Z}^{2}$.

\begin{theorem}
\label{Main Theorem2}Let $\tau\in\mathbb{H}$. The curvature equation
(\ref{Curvature-Lame 1}) admits an even family if and only if there exists $p$
such that the curvature equation (\ref{Curvature 1})$_{p}$ admits a non-even family.
\end{theorem}

\begin{proof}
Suppose the curvature equation (\ref{Curvature-Lame 1}) admits an even family.
There is $\tilde{B}\in\mathbb{C}$ such that the associated classical Lam\'{e}
equation (\ref{Lame 1}) is completely reducible with unitary monodromy. By the
correspondecne (\ref{correspondence}), for any $p$ satisfying $\wp
(p)\not =-\frac{1}{2}\tilde{B}$, the generalized Lam\'{e}-type equation with
parameter $T$ determined by $T^{2}=B+2\wp(p)\not =0$ is completely reducible
and admits the same unitary monodromy data. Consequently, the curvature
equation (\ref{Curvature 1})$_{p}$ admits a non-even family.

Converesly, suppose there exists $p$ such that (\ref{Curvature 1})$_{p}$
admits a non-even family. By Theorem E, there is $T\not =0\in\mathbb{C}$ such
that the generalized Lam\'{e}-type equation is completely reducible with
unitary monodromy data. Applying Theorem \ref{Main Theorem1} again, we
conclude that the curvature equation (\ref{Curvature-Lame 1}) admits an even family.
\end{proof}

When $T=0$, the equation (\ref{GLE 4})$_{p}$ with unitary monodromy admits
only an even family of solutions of the curvature (\ref{Curvature 1})$_{p}$.
In this situation, by Theorem F, the singularity $p$ is given by
$p=p_{r,s}^{(1)}(\tau),$ where $(r,s)$ $\in$ $\mathbb{R}^{2}\setminus\frac
{1}{2}\mathbb{Z}^{2}$ is the monodromy data, which by (\ref{correspondence}),
coincide with that of the classical Lam\'{e} equation (\ref{Lame 1}) with
parameter $\tilde{B}=-2\wp(p).$

Now, suppose the classical Lam\'{e} equation (\ref{Lame 1}) with parameter
$\tilde{B}$ has unitary monodromy data $(r,s)$ $\in$ $\mathbb{R}^{2}
\setminus\frac{1}{2}\mathbb{Z}^{2}$. By Corollary \ref{Main Corollary1}, for
each $p$, the parameter $T(p)$ of the generalized Lam\'{e}-type equation
(\ref{GLE 4})$_{p}$ yielding the same unitary monodromy is determined by
(\ref{correspondence}). By Theorem \ref{Main Theorem at singular}, when
$p=p_{r,s}^{(1)}(\tau)$ (equivalently, when $-2\wp(p)=\tilde{B}$), we have
$T(p)=0$. Thus, the equation equation (\ref{GLE 4})$_{p}$ reduces to the even
symmetry case.

In view of the notation (\ref{singular point}), we denote this special point
by
\begin{equation}
p=p_{\ast}:=p_{r,s}^{(1)}(\tau). \label{special point}%
\end{equation}
That is, $p_{\ast}$ is the unique point on $E_{\tau}$ such that the equation
(\ref{GLE 4})$_{p_{\ast}}$ with parameter $T=0$ shares the same monodromy data
$\left(  r,s\right)  $ as the classical Lam\'{e} equation (\ref{Lame 1}) with
parameter $\tilde{B}$. In particular, for any $p\not =p_{r,s}^{(1)}(\tau)$,
$T(p)\not =0$ and equation (\ref{GLE 4})$_{p}$ preserves exactly the same
unitary monodromy as the classical Lam\'{e} equation (\ref{Lame 1}) with
parameter $\tilde{B}$.

Now, suppose the curvature equation (\ref{Curvature 1})$_{p}$ admits a
\textbf{non-even} family $\left\{  u_{\beta}(z;p)\right\}  $, associated with
the generalized Lam\'{e}-type equation (\ref{GLE 4})$_{p}$, which has
parameter $T(p)$ and monodromy data $\left(  r,s\right)  $ $\in$
$\mathbb{R}^{2}\setminus\frac{1}{2}\mathbb{Z}^{2}$.

According to Theorem \ref{Main Theorem1}, the classical Lam\'{e} equation
(\ref{Lame 1}) with parameter
\[
\tilde{B}=T(p)^{2}-2\wp(p)
\]
also has the same monodromy data. Since both $\pm T(p)$ yield the same
$\tilde{B}$, and the symmetry $z\longmapsto-z$ of the generalized
Lam\'{e}-type equation (\ref{GLE 4})$_{p}$ is equivalent to $T\longmapsto-T$,
it follows that $\left\{  u_{\beta}(-z;p)\right\}  $ also forms a non-even
family of the curvature equation (\ref{Curvature 1})$_{p}$. Moreover, by
(\ref{Monodromy symmetry3}), this second family corresponds to the monodromy
data $\left(  -r,-s\right)  $ of (\ref{GLE 4})$_{p}$ with parameter $-T(p)$.
This relationship is illustrated in the following diagram:
\begin{equation}%
\begin{array}
[c]{ccccc}%
\text{PDE (\ref{Curvature 1})}_{p} &  & \text{GLE (\ref{GLE 4})}_{p} &  &
\text{Lam\'{e} (\ref{Lame 1})}\\
&  &  &  & \\
\left\{  u_{\beta}(z;p)\right\}  & \longleftrightarrow & \left\{
p,T(p),\left(  r,s\right)  \right\}  & \longleftrightarrow & \left\{
\tilde{B},(r,s)\right\} \\
&  & \Updownarrow &  & \\
\left\{  u_{\beta}(-z;p)\right\}  & \longleftrightarrow & \left\{
p,-T(p),\left(  -r,-s\right)  \right\}  & \longleftrightarrow & \left\{
\tilde{B},(r,s)\right\}
\end{array}
. \label{Relation}%
\end{equation}

\begin{remark}
The equality $\left\{  u_{\beta}(z;p)\right\}  =\left\{  u_{\beta
}(-z;p)\right\}  $ holds if and only if $T(p)=0$; in this case, the family is even.
\end{remark}

In the sequel, we identify $\left\{  u_{\beta}(z;p)\right\}  $ with $\left\{
u_{\beta}(-z;p)\right\}  $, as they are connected via the transformation
$z\rightarrow-z$. We denote this non-even family by
\[
\left\{  u_{\beta}(z;p)\right\}  _{(r,s)},
\]
where $\left(  r,s\right)  \in\mathbb{R}^{2}\setminus\frac{1}{2}\mathbb{Z}%
^{2}$ is the monodromy data of the associated generalized Lam\'{e}-type
equation (\ref{GLE 4})$_{p}$. With this convention, the preceding discussion
yields the following theorem.

\begin{theorem}
\label{Main Theorem4}Let $\tau\in\mathbb{H}$. Suppose the curvature equation
(\ref{Curvature-Lame 1})$_{\tau}$ admits an even family (necessarily unique)
with monodromy data $(r,s)$ $\in$ $\mathbb{R}^{2}\setminus\frac{1}
{2}\mathbb{Z}^{2}$. Then:

(i) For any $p\not =p_{r,s}^{(1)}(\tau)$, the curvature equation
(\ref{Curvature 1})$_{p}$ admits exact one non-even familiey $\left\{
u_{\beta}(z;p)\right\}  _{(r,s)}$, in the sense of (\ref{Relation}),
corresponding to $(r,s)$.

(ii) For $p=p_{\ast}:=p_{r,s}^{(1)}(\tau)$, the curvature equation
(\ref{Curvature 1})$_{p}$ admits exact one \textbf{even} family $\left\{
u_{\beta}(z;p_{\ast})\right\}  _{(r,s)}$, whose monodromy data is $(r,s)$.
\end{theorem}

By Theorem \ref{Main Theorem2} and Theorem \ref{Main Theorem4}, for each
$\tau\in\mathbb{H}$, the classification of solutions to the curvature equation
(\ref{Curvature 1})$_{p}$ with $p$ away from the symmetric points of the
elliptic curve $E_{\tau}$, reduces to the analysis of the simpler curvature
equation (\ref{Curvature-Lame 1}) with a single singular point.

The curvature equation (\ref{Curvature-Lame 1}) has been fully investigated in
\cite{LW-AnnMath}. According to Theorem E, this analysis can be reduced to
studying the associated equation (\ref{Zero equation}) for $(r,s)$\textit{
}$\in\mathbb{R}^{2}$ $\setminus\frac{1}{2}\mathbb{Z}^{2}$, which has already
been carried out in \cite{Chen-Kuo-Lin-Wang-JDG}. We will review the relevant
results below.

We first observe that the function $Z(r,s,\tau)$, as defined in (\ref{zrs}),
possesses certain modularity properties (see \cite[(4.4)]%
{Chen-Kuo-Lin-Wang-JDG} or \cite[(6.13)]{Kuo}) and satisfies (\ref{Property1}
). Consequently, it suffices to consider $Z(r,s,\tau)$ for $\left(
r,s\right)  $ $\in$ $\square$ and $\tau$ $\in$ $F_{0}$, where
\[
\square=[0,1/2]\times\lbrack0,1]\setminus\frac{1}{2}\mathbb{Z}^{2},\text{
}F_{0}=\left\{  \tau\in\mathbb{H}|0\leq\operatorname{Re}\tau\leq1,\left\vert
\tau-\frac{1}{2}\right\vert \geq\frac{1}{2}\right\}  .
\]
We further define
\[
\Delta_{0}=\left\{  \left(  r,s\right)  \left\vert 0<r,s<\frac{1}{2}%
,r+s>\frac{1}{2}\right.  \right\}  .
\]

\noindent\textbf{Theorem} \textbf{H. }\textit{(\cite[Theorem 1.3.]%
{Chen-Kuo-Lin-Wang-JDG}) Let }$\left(  r,s\right)  $ $\in$ $\square$\textit{.
Then }$Z(r,s,\tau)$\textit{\ has a zero in }$\tau\in F_{0}$\textit{\ if and
only if }$\left(  r,s\right)  $ $\in$ $\Delta_{0}$\textit{. Moreover, for each
such }$\left(  r,s\right)  $,\textit{ the zero }$\tau$ $\in$ $F_{0}%
$\textit{\ is unique. }\medskip

Define
\[
\Omega=\left\{  \tau\in F_{0}\left\vert Z(r,s,\tau)=0\text{ for some }\left(
r,s\right)  \in\Delta_{0}\right.  \right\}  .
\]
By Theorem H, there exists a real-analytic map
\[
\tau:\Delta_{0}\longrightarrow\Omega,\text{ }Z(r,s,\tau_{0}(r,s))=0,
\]
which is a bijection.

For $\left(  r,s\right)  \in\Delta_{0}$, let $\tau(r,s)\in\Omega$ denote the
unique solution of $Z(r,s,\tau)=0$. In view of the notation in
(\ref{special point}), we set
\[
p_{\ast}(\tau(r,s)):=p_{r,s}^{(1)}(\tau(r,s))\text{ for }\tau=\tau
(r,s)\in\Omega.
\]
Since $p_{\ast}(\tau(r,s))=p_{r,s}^{(1)}(\tau(r,s))$ satisfies the expression
(\ref{PV6-1}) and $Z(r,s,$ $\tau(r,s))$ $=$ $0$, it follows that
\begin{equation}
\wp(p_{\ast}(\tau(r,s)))=-\frac{1}{2}\wp(r+s\cdot\tau(r,s)).
\label{exception point}%
\end{equation}

We thus obtain the following characterization of $\Lambda_{Noneven}^{(1)}$.

\begin{theorem}
\label{Main Theorem5}The set $\Lambda_{Noneven}^{(1)}$ can be characterized as follows:

(i) $\tau=\tau(r,s)\in\Omega$ for some $\left(  r,s\right)  \in\Delta_{0}$,

(ii) For such $\tau\left(  r,s\right)  \in\Omega$, the admissible singularity
$p$ satisfies
\[
p\in E_{\tau(r,s)}\setminus\{\frac{\omega_{k}}{2},k=0,1,2,3\}\text{, }
p\not =p_{\ast}(\tau(r,s)).
\]

Equivalently,
\[
\Lambda_{Noneven}^{(1)}=\left\{  (\tau,p)\left\vert \tau=\tau(r,s)\in
\Omega\text{, }\left(  r,s\right)  \in\Delta_{0}\text{, }p\not =p_{\ast}
(\tau(r,s))\right.  \right\}  .
\]
Furthermore, the exceptional point $p_{\ast}(\tau(r,s))$ is characterized by
(\ref{exception point}).
\end{theorem}

\subsection{Blow-up Analysis of the Family in Theorem
\ref{Main Theorem4}}

\label{Blowup analysis of the non-even family}

Let $\tau=\tau(r,s)\in\Omega$. For each
\[
p\in E_{\tau(r,s)}\setminus\{\frac{\omega_{k}}{2},k=0,1,2,3\},
\]
Theorem \ref{Main Theorem4} ensures the existence of a unique non-even family
of solutions to the curvature equation (\ref{Curvature 1})$_{p}$ if
$p\not =p_{\ast}(\tau(r,s))$; while it is an even family if $p$ $=$ $p_{\ast
}(\tau(r,s))$. We denote the family obtained in Theorem \ref{Main Theorem4} by
$\left\{  u_{\beta}(z;p)\right\}  _{(r,s)}.$ As noted earlier, this family
exhibits blow-up behavior as $\beta\rightarrow\pm\infty$, respectively. The
natural problem is to determine the corresponding blow-up sets.

In this section, we investigate the blow-up sets associated with the non-even
family of the curvature equation (\ref{Curvature 1})$_{p}$, obtained in
Theorem \ref{Main Theorem4}. Specifically, we ask: for which $p$ does the
curvature equation (\ref{Curvature 1})$_{p}$ admit a family that blows up at
the singularities $\pm p$? Equivalently, when can a blow-up configuration
concentrate precisely at $p$ or $-p$?

The following theorem provides a definitive criterion for the family \\
$\left\{
u_{\beta}(z;p)\right\}  _{(r,s)}$ obtained in Theorem \ref{Main Theorem4}.

\begin{theorem}
\label{blowup at p}Let $\tau=\tau(r,s)\in\Omega$ and
\[
p\in E_{\tau(r,s)}\setminus\{\frac{\omega_{k}}{2},k=0,1,2,3\}.
\]
Then the unique family of solutions $\left\{  u_{\beta}(z;p)\right\}
_{(r,s)}$ of the curvature equation (\ref{Curvature 1})$_{p}$ obtained in
Theorem \ref{Main Theorem4} blows up at the singularity $p$ as $\beta
\rightarrow+\infty$ if and only if $p$ is determined by
\[
2p=\pm(r+s\tau)\text{ (mod }\Lambda_{\tau}\text{)}.
\]

\end{theorem}

\begin{proof}
According to Theorem \ref{a1,a2 alge equ} (ii), the assertion is equivalent to
the existence of a point $p$ such that the corresponding
\[
T(p)=\frac{\wp^{\prime\prime}(p)}{2\wp^{\prime}(p)}\text{ (up to a sign).}%
\]
Recall the relation
\[
\wp(r+s\tau)=\tilde{B}=T(p)^{2}-2\wp(p).
\]
By choosing $p$ such that
\[
2p=\pm(r+s\tau)\text{ (mod }\Lambda_{\tau}\text{),}%
\]
we obtain
\[
T(p)^{2}=\wp(2p)+2\wp(p)=\frac{\wp^{\prime\prime}(p)^{2}}{4\wp^{\prime}%
(p)^{2}},
\]
and hence, equivalently,
\[
T(p)=\frac{\wp^{\prime\prime}(p)}{2\wp^{\prime}(p)}\text{ (up to a sign).}%
\]
This proves the theorem.
\end{proof}

Next, we analyze the family of solutions of the curvature equation that blows
up at the set $\{a_{1},a_{2}\}$, where $a_{j}\not =\pm p$. In this case, we
have
\begin{equation}
a_{j}\not \in \left\{  \pm p,\frac{\omega_{k}}{2},k=0,1,2,3,\right\}  \text{
and }a_{1}\not =\pm a_{2}. \label{blowup sets cond}%
\end{equation}
When the blow-up points do not coincide with the singularities $\pm p$, the
blow-up analysis of the curvature equation (\ref{Curvature 1})$_{p}$, relies
on the classical Pohozaev identity, which characterizes the blow-up points
\cite{CCChen-Lin 1}. To make this more precise, we first introduce the Green
function $G(z;\tau)$ on $E_{\tau}$, defined by
\[
-\Delta G(z;\tau)=\delta_{0}-\frac{1}{\left\vert E_{\tau}\right\vert }\text{,
}\int_{E_{\tau}}G(z;\tau)=0,
\]
where $\left\vert E_{\tau}\right\vert $ denotes the area of the torus. This
function is even and has its only singularity at $z=0$. In the following, we
omit the dependence on $\tau$ and simply write $G(z;\tau)=G(z)$.

Let $\eta$ be the linear map
\[
\eta:E_{\tau}\rightarrow\mathbb{C}\text{ , }z\longmapsto\eta(z)
\]
defined by
\[
\eta(z):=r\eta_{1}(\tau)+s\eta_{2}(z)\text{ if }z=r+s\tau\text{, }%
r,s\in\mathbb{R}.
\]
Recall from \cite{LW-AnnMath}\ the following identity
\begin{equation}
-4\pi\partial_{z}G(z;\tau)=\zeta(z;\tau)-\eta(z):=Z(z,\tau),
\label{derivatives of G}%
\end{equation}
where
\[
Z(z,\tau)=Z(r,s,\tau)\text{ if }z=r+s\tau\text{, }r,s\in\mathbb{R}.
\]

It follows from (\ref{derivatives of G}) that for $\left(  r,s\right)  $ $\in$
$\{\left(  1/2,0\right)  ,$ $(0,1/2),$ $(1/2,1/2)\}$ mod $\mathbb{Z}^{2}$,%
\[
Z\left(  \frac{\omega_{k}}{2},\tau\right)  \equiv0
\]
for any $\tau\in\mathbb{H}$. In other words, each half-period $\frac
{\omega_{k}}{2}$ is always a critical point of $G(z;\tau)$, and these are
referred to as the \textit{trivial} critical points. As shown in
\cite{LW-AnnMath}, for any torus $E_{\tau}$, the Green function $G(z;\tau)$
either has exactly three trivial critical points, or it additionally admits a
pair of \textit{nontrivial} critical points $\pm\sigma\in E_{\tau}$, where
\[
\sigma=r+s\tau,\text{ }\left(  r,s\right)  \in\mathbb{R}^{2}\setminus\frac
{1}{2}\mathbb{Z}^{2}.
\]
Theorem G further indicates that the presence of nontrivial critical points of
$G(z;\tau)$ corresponds to flat tori characterized by the existence of even
solutions to the curvature equation (\ref{Curvature-Lame 1})$_{\tau}$. The
pair of nontrivial critical points $\sigma$ and $-\sigma$ correspond to the
blow-up points of the unique even family $\left\{  v_{\beta}(z)|\beta
\in\mathbb{R}\right\}  $ of solutions, associated with the limits\ $\beta
\rightarrow+\infty$ and $\beta\rightarrow-\infty$, respectively.

The blow-up analysis of the associated multiple Green function plays a crucial
role in understanding cone spherical metrics. Very recently, Chen, Fu and Lin
in \cite{Chen-Fu-Lin-Hitchin} studied the curvature equation
\begin{equation}
\Delta v+e^{v}=4\pi(\delta_{p}+\delta_{-p})\text{ on }E_{\tau}%
.\label{curvature2}%
\end{equation}
Since the total curvature of the curvature equation (\ref{curvature2}) is
$8\pi$, any blow-up family of solutions admits only a single blow-up point.
This allows the application of the method of anti-holomorphic dynamics
developed by Bergweiler and Eremenko \cite{Bergweiler-Eremenko-dynamics},
together with Hitchin's formula, to obtain refined results concerning the
critical points of the associated multiple Green function%
\[
\frac{1}{2}\left(  G(z-p)+G(z+p)\right)  .
\]
In contrast, for the curvature equation (\ref{Curvature 1}), the total
curvature is $16\pi$, and, generically, any blow-up family involves two
blow-up points. Hence, the analysis becomes significantly more intricate.

For each $p\in E_{\tau(r,s)}\setminus\{\frac{\omega_{k}}{2},k=0,1,2,3\}$, we
then introduce the associated \textbf{multiple Green function}, defined by
\[
G_{p}(z_{1},z_{2}):=G(z_{1}-z_{2})-\sum_{i=1}^{2}(G(z_{i})+\frac{1}{2}%
G(z_{i}-p)+\frac{1}{2}G(z_{i}+p)).
\]
This multiple Green function naturally arises in the analysis of blow-up
configurations via the Pohozaev identity: its critical points characterize the
locations of blow-up points whenever they are not equal to $\pm p$.

By differentiating and applying (\ref{derivatives of G}), we obtain that the
pair $\{a_{1},a_{2}\}$ satisfies the following relations
\begin{equation}
\frac{\wp^{\prime}(a_{1})}{2(\wp(a_{1})-\wp(p))}+\frac{\wp^{\prime}(a_{2}%
)}{2(\wp(a_{2})-\wp(p))}-\frac{\wp^{\prime}(a_{1})-\wp^{\prime}(a_{2})}%
{\wp(a_{1})-\wp(a_{2})}=-2Z(\sigma),\label{22}%
\end{equation}
and
\begin{equation}
\frac{\wp^{\prime}(a_{1})}{2(\wp(a_{1})-\wp(p))}-\frac{\wp^{\prime}(a_{2}%
)}{2(\wp(a_{2})-\wp(p))}-\frac{\wp^{\prime}(a_{1})+\wp^{\prime}(a_{2})}%
{\wp(a_{1})-\wp(a_{2})}=0,\label{23}%
\end{equation}
where
\[
\sigma=a_{1}+a_{2}=r+s\tau
\]
for a unique pair $\left(  r,s\right)  \in\square$. Moreover, by Theorem
\ref{blowup at p}, we have%
\[
2p\not =\pm\sigma.
\]
The above equations (\ref{22}) and (\ref{23}) represent the critical point
conditions for the multiple Green function $G_{p}(z_{1},z_{2})$.

To study the system of equations (\ref{22}) and (\ref{23}), we introduce the
following notations: For $j=1,2,$ set%
\begin{equation}
\left(  x_{j},y_{j}\right)  :=\left(  \wp(a_{j}),\wp^{\prime}(a_{j})\right)  .
\label{Notation}%
\end{equation}
By the classical relation for the Weierstrass elliptic function, each pair
$\left(  x_{j},y_{j}\right)  $ lies on the elliptic curve and satisfies%
\begin{equation}
y_{j}^{2}=4x_{j}^{3}-g_{2}x_{j}-g_{3},\text{ }j=1,2.
\label{elliptic relations}%
\end{equation}
With this notation, equations (\ref{22}) and (\ref{23}) can be rewritten as
the following linear system in $y_{1}$, $y_{2}:$%
\begin{equation}
\left\{
\begin{array}
[c]{l}%
\left(  \frac{1}{2(x_{1}-\wp(p))}-\frac{1}{x_{1}-x_{2}}\right)  y_{1}+\left(
\frac{1}{2(x_{2}-\wp(p))}+\frac{1}{x_{1}-x_{2}}\right)  y_{2}=-2Z(\sigma),\\
\\
\left(  \frac{1}{2(x_{1}-\wp(p))}-\frac{1}{x_{1}-x_{2}}\right)  y_{1}-\left(
\frac{1}{2(x_{2}-\wp(p))}+\frac{1}{x_{1}-x_{2}}\right)  y_{2}=0,
\end{array}
\right.  \label{Green system}%
\end{equation}
together with the elliptic curve relations (\ref{elliptic relations}).

Hence, the existence of a blow-up family of solutions to the curvature
equation (\ref{Curvature 1})$_{p}$ blowing up at $\{a_{1},a_{2}\}$ subject to
condition (\ref{blowup sets cond}) as $\beta\rightarrow\pm\infty$ is
equivalent to the existence of points $\left(  x_{1},y_{1}\right)  $ and
$\left(  x_{2},y_{2}\right)  $ on the elliptic curve defined by
(\ref{elliptic relations}) , such that $(y_{1},y_{2})$ forms a nontrivial
solution of the linear system (\ref{Green system}).

A straightforward computation shows that the determinant of the linear system
(\ref{Green system}) is given by%
\[
\det=\frac{\left(  x_{1}+x_{2}-2\wp(p)\right)  ^{2}}{2(x_{1}-x_{2})(x_{1}%
-\wp(p))(x_{2}-\wp(p))}.
\]
The constraint (\ref{blowup sets cond}) for $a_{1}$, $a_{2}$ implies that
\[
x_{1}\not =x_{2}\text{, and }x_{j}\not =\wp(p)\text{.}%
\]
Thus, the determinant vanishes if and only if
\[
\wp(p)=\frac{1}{2}(x_{1}+x_{2}).
\]
In this situation, the linear system (\ref{Green system}) degenerates and
reduces to
\[
Z(\sigma)=Z(r,s,\tau)=0.
\]

\begin{theorem}
\label{Main Theorem6}Suppose the curvature equation (\ref{Curvature 1})$_{p}$
admits a blow-up family $\{u_{\beta}(z)|\beta\in\mathbb{R}\}$, either even or
non-even, of solutions blowing up at $\{a_{1},a_{2}\}$ subject to condition
(\ref{blowup sets cond}) as $\beta\rightarrow\pm\infty$. Adopt the notation
(\ref{Notation}) and set
\[
\sigma:=a_{1}+a_{2}=r+s\tau
\]
where $\left(  r,s\right)  $ is the unique real pair in $[0,1]\times
\lbrack0,1]$ representing the monodromy data of the associated generalized
Lam\'{e}-type equation (in either the even or punctured non-even symmetry
case). Then:

(1)
\[
2p\not =\pm\sigma\text{ (mod }\Lambda_{\tau}\text{)}%
\]

(2)
\[
Z(\sigma)=Z(r,s,\tau)=0\Longleftrightarrow\wp(p)=\frac{1}{2}(x_{1}+x_{2}).
\]

\end{theorem}

\begin{remark}
The blow-up family $\{u_{\beta}(z)|\beta\in\mathbb{R}\}$ in Theorem
\ref{Main Theorem6} is \textbf{not} necessary obtained from Theorem
\ref{Main Theorem4}.
\end{remark}

Theorem \ref{Main Theorem6} and Theorem \ref{Main Theorem4} can be used to
characterize the blow-up points subject to condition (\ref{blowup sets cond}).
Suppose $\tau=\tau(r,s)\in\Omega$, that is, $Z(r,s,\tau(r,s))=0$. By Theorem
\ref{Main Theorem4}, for each $p$, there exists a family of solutions
\[
\{u_{\beta}(z;p)|\beta\in\mathbb{R}\}_{(r,s)}.
\]
This family of solutions blows up at $\{a_{1}^{\pm},a_{2}^{\pm}\}$ subject to
condition (\ref{blowup sets cond}) as $\beta\rightarrow\pm\infty$ whenever the
singularity $p$ satisfies $2p$ $\not =$ $\pm\sigma$ (mod $\Lambda_{\tau}$).
Moreover, this family $\{u_{\beta}(z;p)|\beta\in\mathbb{R}\}_{(r,s)}$ is a
non-even family provided $p\not =p_{\ast}$.

By Theorem \ref{Main Theorem6}, we obtain
\begin{equation}
2\wp(p)=x_{1}+x_{2}=\wp(a_{1})+\wp(\sigma-a_{1}). \label{blowup equation1}%
\end{equation}
Sincce $a_{1}\not =\sigma$, applying the addition formula for the Weierstrass
function,%
\begin{equation}
\label{addition formula for wp}\wp(z+w)=-\wp(z)-\wp(w)+\frac{\left(
\wp^{\prime}(z)-\wp^{\prime}(w)\right)  ^{2}}{4\left(  \wp(z)-\wp(w)\right)
^{2}},
\end{equation}
equation (\ref{blowup equation1}) is equivalent to
\begin{equation}
\left\{
\begin{array}
[c]{l}%
y_{1}^{2}+2\wp^{\prime}(\sigma)y_{1}+\wp^{\prime}(\sigma)^{2}-4(2\wp
(p)+\wp(\sigma))(\wp(\sigma)-x_{1})^{2}=0,\\
\\
y_{1}^{2}=4x_{1}^{3}-g_{2}x_{1}-g_{3}.
\end{array}
\right.  \label{blowup system1}%
\end{equation}
where $\left(  x_{1},y_{1}\right)  =\left(  \wp(a_{1}),\wp^{\prime}%
(a_{1})\right)  $.

Solving the system yields four solutions for $x_{1}$:%
\[
x_{1}=\wp(p)\pm\Delta_{+}(p,\sigma)\text{, or }\wp(p)\pm\Delta_{-}(p,\sigma),
\]
where%
\[
\Delta_{\pm}(p,\sigma):=\frac{1}{2}\sqrt{g_{2}+4\left[  \wp(p)^{2}-2\wp
(\sigma)\wp(p)-2\wp(\sigma)^{2}\pm\wp^{\prime}(\sigma)\sqrt{2\wp(p)+\wp
(\sigma)}\right]  }%
\]
By symmetry in $x_{1}$ and $x_{2}$ in the equation (\ref{blowup equation1}),
the blow-up configurations of the family $\{u_{\beta}(z;p)|\beta\in
\mathbb{R}\}_{(r,s)}$ are determined by the monodromy data $\left(
r,s\right)  $ as follows:

\begin{itemize}
\item As $\beta\rightarrow+\infty$, the blow-up set $\{a_{1}^{+},a_{2}^{+}\}$
satisfies%
\begin{equation}
\left\{  \wp(a_{1}^{+}),\wp(a_{2}^{+})\right\}  =\left\{  \wp(p)\pm\Delta
_{+}(p,\sigma)\right\}  . \label{positive blowup}%
\end{equation}

\item As $\beta\rightarrow-\infty$, the blow-up set $\{a_{1}^{-},a_{2}^{-}\}$
satisfies%
\begin{equation}
\left\{  \wp(a_{1}^{-}),\wp(a_{2}^{-})\right\}  =\left\{  \wp(p)\pm\Delta
_{-}(p,\sigma)\right\}  . \label{Negative bolowup}%
\end{equation}

\end{itemize}

\begin{theorem}
\label{Main Theorem7}Let $\tau=\tau(r,s)\in\Omega$, and $\{u_{\beta
}(z;p)|\beta\in\mathbb{R}\}_{(r,s)}$ be the family of solutions obtained in
Theorem \ref{Main Theorem4}. Suppose the singularity $p$ satisfies
\[
2p\not =\pm\sigma\text{ (mod }\Lambda_{\tau}\text{), }\sigma=r+s\tau.
\]
Then the blow-up set of $\{u_{\beta}(z;p)|\beta\in\mathbb{R}\}_{(r,s)}$ is
characterized by (\ref{positive blowup}) and (\ref{Negative bolowup}) as
$\beta\rightarrow\pm\infty$, respectively.

In particular, when $p=p_{\ast}$, by (\ref{exception point}), the blow-up sets
of the \textbf{even} family simplify to%
\[
\left\{  a_{1}^{+},a_{2}^{+}\right\}  =\left\{  -a_{1}^{-},-a_{2}^{-}\right\}
,
\]
and
\begin{equation}
\left\{  \wp(a_{1}^{+}),\wp(a_{2}^{+})\right\}  =\left\{  \wp(a_{1}^{-}%
),\wp(a_{2}^{-})\right\}  =\left\{  \frac{1}{2}\left(  -\wp(\sigma)\pm
\sqrt{g_{2}-3\wp(\sigma)^{2}}\right)  \right\}  . \label{blowup  set even}%
\end{equation}

\end{theorem}

The last statement (\ref{blowup set even}) follows immediately by substituting
$\wp(p_{\ast})=-\frac{1}{2}\wp(\sigma)$ into (\ref{positive blowup}) and
(\ref{Negative bolowup}). Moreover, the point $\sigma$ $=$ $r+s\tau$, with
$\left(  r,s\right)  \in$ $\mathbb{R}^{2}$ $\setminus\frac{1}{2}\mathbb{Z}%
^{2}$, also represents the blow-up point of the even family $\{v_{\beta
}(z)|\beta\in\mathbb{R}\}_{(r,s)}$ of the curvature equation
(\ref{Curvature-Lame 1}) as $\beta\rightarrow+\infty$ in Theorem
\ref{Main Theorem4}.

Hence Theorem \ref{Main Theorem7} establishes the correspondence between the
blow-up set of $\{u_{\beta}(z;p)|\beta\in\mathbb{R}\}_{(r,s)}$ for the
curvature equation (\ref{Curvature 1})$_{p}$ and that of the curvature
equation (\ref{Curvature-Lame 1}).

Let $\tau=\tau(r,s)\in\Omega$ and suppose $p$ satisfies
\[
2p=\pm(r+s\tau)\text{ (mod }\Lambda_{\tau}\text{)}.
\]
According to Theorem \ref{blowup at p}, the family $\{u_{\beta}(z;p)|\beta
\in\mathbb{R}\}_{(r,s)}$ blows up at the singularity $p$, removing a total
curvature of $16\pi$, as $\beta\rightarrow+\infty$. By taking $2p\not =%
\pm\sigma$ and $2p\rightarrow\pm\sigma$ (mod $\Lambda_{\tau}$), we obtain
\[
\wp(a_{1}^{+})=\wp(a_{2}^{+})=\wp(p),
\]
which is equivalent to the vanishing of the discriminant%
\[
\Delta_{+}(p,\sigma)=0.
\]
Since $\Delta_{+}(p,\sigma)=0$, we obtain%
\[
\Delta_{-}(p,\sigma)=\sqrt{\frac{g_{2}}{2}-4\wp(\sigma)^{2}-4\wp(p)\wp
(\sigma)+2\wp(p)^{2}}.
\]
By Theorem \ref{Main Theorem7}, the opposite blow-up set of this family as
$\beta\rightarrow-\infty$, can be characterized by%
\[
\left\{  \wp(a_{1}^{-}),\wp(a_{2}^{-})\right\}  =\wp(p)\pm\sqrt{\frac{g_{2}%
}{2}+2\wp(p)^{2}-4\wp(p)\wp(\sigma)-4\wp(\sigma)^{2}}.
\]

We summarize the above discussion as below.

\begin{corollary}
\label{Main corollary1}Let $\tau=\tau(r,s)\in\Omega$, and let
\[
\{u_{\beta}(z;p)|\beta\in\mathbb{R}\}_{(r,s)}%
\]
be the family of solutions obtained in Theorem \ref{Main Theorem4}. Suppose
$p$ satisfies
\[
2p=\pm(r+s\tau)\text{ (mod }\Lambda_{\tau}\text{)}.
\]
Then:

(i) As $\beta\rightarrow+\infty$, $\{u_{\beta}(z;p)|\beta\in\mathbb{R}%
\}_{(r,s)}$ blows up at $p$, removing a total curvature of $16\pi.$

(ii) As $\beta\rightarrow-\infty$, $\{u_{\beta}(z;p)|\beta\in\mathbb{R}%
\}_{(r,s)}$ blows up at two distinct points $\{a_{1}^{-},a_{2}^{-}\}$,
characterized by
\[
\left\{  \wp(a_{1}^{-}),\wp(a_{2}^{-})\right\}  =\wp(p)\pm\sqrt{\frac{g_{2}%
}{2}+2\wp(p)^{2}-4\wp(p)\wp(\sigma)-4\wp(\sigma)^{2}},
\]
where each blow-up point removes a total curvature of $8\pi.$
\end{corollary}

The two distinct blow-up points $a_{1}^{-}$, $a_{2}^{-}$ may collapse to the
singularity when $g_{2}(\tau)=0$. Indeed, consider%
\[
\tau=\rho:=e^{\pi i/3}.
\]
It is well known from \cite[Example 2.6]{LW-AnnMath} that
\[
g_{2}(\rho)=0\text{, }Z(\frac{1}{3},\frac{1}{3},\rho)=0\text{, }\wp(\frac
{1}{3}+\frac{\rho}{3};\rho)=0.
\]
That is,
\[
\rho=\tau\left(  \frac{1}{3},\frac{1}{3}\right)  \in\Omega.
\]
Suppose%
\[
p_{0}=\frac{1}{3}+\frac{\rho}{3}\text{ (mod }\Lambda_{\rho}\text{)}.
\]
Then we have
\[
\wp(p_{0};\rho)=\wp(2p_{0};\rho)=\wp(\frac{1}{3}+\frac{\rho}{3};\rho)=0.
\]
By $g_{2}(\rho)=0$, it follows that%
\[
\wp^{\prime\prime}(p_{0};\rho)=0,
\]
which implies
\[
T(p_{0})=\frac{\wp^{\prime\prime}(p_{0})}{2\wp^{\prime}(p_{0})}=0.
\]
Thus,
\[
p_{0}=p_{\ast}.
\]
Therefore, the family
\[
\{u_{\beta}(z;p_{\ast})|\beta\in\mathbb{R}\}_{\left(  \frac{1}{3},\frac{1}%
{3}\right)  }%
\]
is an \textbf{even} family of solution to the curvature equation
(\ref{Curvature 1})$_{\tau=\rho,p=p_{\ast}}$.

Moreover, $p_{\ast}$ satisfies
\[
2p_{\ast}=\frac{2}{3}+\frac{2\rho}{3}=\frac{1}{3}+\frac{\rho}{3}=\sigma\text{
(mod }\Lambda_{\rho}\text{)},
\]
It then follows from Theorem \ref{blowup at p} and Corollary
\ref{Main corollary1}, $\{u_{\beta}(z;p_{\ast})|\beta\in\mathbb{R}\}_{\left(
\frac{1}{3},\frac{1}{3}\right)  }$ blows up at $p_{\ast}$ as $\beta
\rightarrow+\infty,$ and at $-p_{\ast}$ as $\beta\rightarrow-\infty$.

\begin{corollary}
Let $\tau=\rho=e^{\pi i/3}$. Then the even family
\[
\{u_{\beta}(z;p_{\ast})|\beta\in\mathbb{R}\}_{\left(  \frac{1}{3},\frac{1}%
{3}\right)  }%
\]
to the curvature equation (\ref{Curvature 1})$_{\tau=\rho,p=p_{\ast}}$ blows
up at $p_{\ast}$ as $\beta\rightarrow+\infty,$ and at $-p_{\ast}$ as
$\beta\rightarrow-\infty$.
\end{corollary}

\subsection{Deformations of the Family in Theorem \ref{Main Theorem4}}

\label{Deformations of non-even family}

\textbf{Fix moduli parameter} $\tau=\tau(r,s)\in\Omega$. In this subsection,
we investigate the deformation theory of the family $\{u_{\beta}(z;p)|\beta
\in\mathbb{R}\}_{\left(  r,s\right)  }$ obtained in Theorem
\ref{Main Theorem4}, in the following setting: Deform
\[
p\in E_{\tau}\setminus\left\{\frac{\omega_{k}}{2},k=0,1,2,3\right\}
\]
such that $p\rightarrow\frac{\omega_{k}}{2}$ for $k=0,1,2,3$.

Recall the correspondence (\ref{correspondence}):%
\begin{equation}
\tilde{B}=\wp(r+s\cdot\tau(r,s))=T(p)^{2}-2\wp(p).
\label{Recall correspondence}%
\end{equation}
In general, we may assume the family $\{u_{\beta}(z;p)|\beta\in\mathbb{R}%
\}_{\left(  r,s\right)  }$ is non-even. We deform $p$ subject to the condition
(\ref{Recall correspondence}) such that $p\rightarrow\frac{\omega_{k}}{2}$,
$k=0,1,2,3$. That is,%
\[
T(p)^{2}=\left\{
\begin{array}
[c]{l}%
\frac{2}{p^{2}}+\tilde{B}+\frac{g_{2}}{10}p^{2}+O(p^{4})\text{ \ if \ }k=0.\\
\\
\tilde{B}+2e_{k}+\wp^{\prime\prime}(\frac{\omega_{k}}{2})(p-\frac{\omega_{k}%
}{2})^{2}+O\left(  \left(  p-\frac{\omega_{k}}{2}\right)  ^{4}\right)  \text{
\ if \ }k=1,2,3.
\end{array}
\right.
\]

\textbf{(i) The case }$k=0$\textbf{.\medskip}

By Theorem \ref{Main Theorem3} (i), the generalized Lam\'{e}-type equation
(\ref{GLE 4})$_{p}$ converges to the classical Lam\'{e} equation
\[
y^{\prime\prime}(z)=\left[  2\wp(z;\tau(r,s))+\tilde{B}\right]  y(z),
\]
which is completely reducible and has the same monodromy data $\left(
r,s\right)  \in\Delta_{0}$. Consequently, the family $\{u_{\beta}%
(z;p)|\beta\in\mathbb{R}\}_{\left(  r,s\right)  }$ converges to the unique
even family $\{v_{\beta}(z)|\beta\in\mathbb{R}\}_{\left(  r,s\right)  }$ of
the curvature equation (\ref{Curvature-Lame 1})$_{\tau=\tau(r,s)}$ as
$p\rightarrow0$ under condition (\ref{Recall correspondence}).

\textbf{(ii) The case }$k=1,2,3.$\textbf{\medskip}

By Theorem \ref{Main Theorem3} (ii), as $p\rightarrow\frac{\omega_{k}}{2}$ the
generalized Lam\'{e}-type equation (\ref{GLE 4})$_{p}$ converges to the
following Lam\'{e}-type equation
\begin{equation}
y^{\prime\prime}(z)=\left\{  2\left(  \wp(z)+\wp(z-\frac{\omega_{k}}%
{2})\right)  +2\tilde{T}_{k}\left(  \zeta(z-\frac{\omega_{k}}{2}%
)-\zeta(z)\right)  +\tilde{B}_{k}\right\}  y(z) \label{limiting lame k}%
\end{equation}
where
\[
\tilde{T}_{k}^{2}=\lim_{p\rightarrow\frac{\omega_{k}}{2}}T(p)^{2}=\tilde
{B}+2e_{k},\text{ }\tilde{B}_{k}=\tilde{T}_{k}^{2}+\eta_{k}\tilde{T}_{k}%
-e_{k}.
\]
Since the generalized Lam\'{e}-type equation (\ref{GLE 4})$_{p}$ associated
with the curvature equation (\ref{Curvature 1})$_{\tau=\tau(r,s),p}$ is
completely reducible, the limiting Lam\'{e}-type equation
(\ref{limiting lame k}) must also be completely reducible. Denote the
monodromy data of (\ref{limiting lame k}) by $\left(  \tilde{r},\tilde
{s}\right)  \in\mathbb{C}^{2}\setminus\frac{1}{2}\mathbb{Z}^{2}$. Then, by
(\ref{limiting M}), the relation between the monodromy data of (\ref{GLE 4}%
)$_{p}$ and that of (\ref{limiting lame k}) is given by
\[
\left(  r,s\right)  =\left(  \tilde{r}_{k},\tilde{s}_{k}\right)  \text{ for
}k=1,2,3,
\]
where%
\begin{equation}
\left(  \tilde{r}_{k},\tilde{s}_{k}\right)  =\left\{
\begin{array}
[c]{l}%
\left(  \tilde{r}+\frac{1}{2},\tilde{s}\right)  \text{ if }k=1,\\
\\
\left(  \tilde{r},\tilde{s}+\frac{1}{2}\right)  \text{ if }k=2,\\
\\
\left(  \tilde{r}+\frac{1}{2},\tilde{s}+\frac{1}{2}\right)  \text{ if }k=3.
\end{array}
\right.  . \label{rsk}%
\end{equation}

With this notation (\ref{rsk}), we obtain%
\[
Z(r,s,\tau(r,s))=Z(\tilde{r}_{k},\tilde{s}_{k},\tau(r,s))=0.
\]
Since $Z(\tilde{r}_{k},\tilde{s}_{k},\tau(r,s))=0$, Theorem 6.1 in
\cite[Theorem 1.3]{Kuo} implies $\tilde{T}_{k}\not =0$. By Theorem 1.3 in
\cite{Kuo}, the Lam\'{e}-type equation (\ref{limiting lame k}) yields the
unique non-even family%
\[
\left\{  v_{k,\beta}(z)\,|\,\beta\in\mathbb{R}\right\}_{\left(  \tilde{r}%
,\tilde{s}\right)  }%
\] 
of solutions to the curvature equation
\begin{equation}
\Delta v+e^{v}=8\pi\delta_{0}+8\pi\delta_{\frac{\omega_{k}}{2}}\text{ on
}E_{\tau(r,s)}. \label{curvature k}%
\end{equation}
In particular, the family $\{u_{\beta}(z;p)|\beta\in\mathbb{R}\}_{\left(
r,s\right)  }$ converges to the unique non-even family $\left\{  v_{k,\beta
}(z)|\beta\in\mathbb{R}\right\}  _{\left(  \tilde{r},\tilde{s}\right)  }$ of
the curvature equation (\ref{curvature k}).

\begin{theorem}
\label{limiting theorem}Let $\tau=\tau(r,s)\in\Omega$ be fixed and
$\{u_{\beta}(z;p)|\beta\in\mathbb{R}\}_{\left(  r,s\right)  }$ be the family
obtained in Theorem \ref{Main Theorem4}. Suppose $p\not =p_{\ast}%
\rightarrow\frac{\omega_{k}}{2}$, $k=0,1,2,3,$ subject condition
(\ref{Recall correspondence}). Then:

(i) For $k=0$, the non-even family $\{u_{\beta}(z;p)|\beta\in\mathbb{R}%
\}_{\left(  r,s\right)  }$ converges uniformly to the unique \textbf{even}
family $\{v_{\beta}(z)|\beta\in\mathbb{R}\}_{\left(  r,s\right)  }$ of the
curvature equation (\ref{Curvature-Lame 1})$_{\tau=\tau(r,s)}$.

(ii) For $k=1,2,3$, the non-even family $\{u_{\beta}(z;p)|\beta\in
\mathbb{R}\}_{\left(  r,s\right)  }$ converges uniformly to the unique
non-even family $\left\{  v_{k,\beta}(z)|\beta\in\mathbb{R}\right\}  _{\left(
\tilde{r},\tilde{s}\right)  }$ of the curvature equation (\ref{curvature k}),
where the monodromy data satisfy
\[
\left(  r,s\right)  =\left(  \tilde{r}_{k},\tilde{s}_{k}\right)
\]
with $\left(  \tilde{r}_{k},\tilde{s}_{k}\right)  $ defined in (\ref{rsk}).
\end{theorem}

\begin{remark}
This theorem shows a dichotomy in the limiting behavior at half-periods. As
$p\rightarrow0$, the family reduces to the even branch of the curvature
equation, preserving the monodromy data $\left(  r,s\right)  $. In contrast,
as $p\rightarrow\frac{\omega_{k}}{2}$ ($k=1,2,3$) , the limiting families
become non-even, with the monodromy data shifted by a half-lattice
translation. Thus, the distinction between even and non-even families is
governed by the lattice symmetries of the monodromy data under half-period
shifts. This result generalizes Theorem 1.3 in \cite{Kuo}.
\end{remark}

\section{Spectral Theory}

\label{Spectral Theory Section}

Recall the assumption \eqref{Assumption 1}. In this section, we aim to
establish the spectral theory for the generalized Lam\'{e}-type equation:
\begin{equation}
y^{\prime\prime}(z)=q(z;T_{1},T_{2},B)\,y(z)\quad\text{on}\quad E_{\tau
},\label{Fuchsian with T1,T2,B}%
\end{equation}
where $T_{1},T_{2},B\in\mathbb{C}$, and the potential $q(z;T_{1},T_{2},B)$ is
given by (\ref{GLEn})$_{n=1}$:
\begin{equation}
q(z;T_{1},T_{2},B)=\left(
\begin{array}
[c]{l}%
2\wp(z)+\dfrac{3}{4}\left(  \wp(z+p)+\wp(z-p)\right)  \\
+T_{1}\left(  \zeta(z+p)-\zeta(z)\right)  +T_{2}\left(  \zeta(z-p)-\zeta
(z)\right)  +B
\end{array}
\right)  .\label{potential with T1,T2,B}%
\end{equation}

Firstly, we apply the standard Frobenius method to derive the apparent
condition for equation \eqref{Fuchsian with T1,T2,B}$_{p}$.

\begin{lemma}
\label{apparent condition lemma} The generalized Lam\'{e}-type equation
\eqref{Fuchsian with T1,T2,B}$_{p}$ is apparent at the singularities $z=0,\pm
p$ if and only if $T_{1},T_{2},B$ satisfy:
\begin{equation}
\label{apparent condition, n=1}\left(  T_{1}+T_{2}\right)  \left(  T_{1}%
-T_{2}+\dfrac{\wp^{\prime\prime} (p)}{2\wp^{\prime}(p)}\right)  =0,
\end{equation}
where $B$ is determined by
\begin{equation}
\label{apparent condition, B}B=\dfrac{1}{2}\left(  T_{1}^{2}+T_{2}^{2}\right)
-\dfrac{1}{2}\left(  T_{1}-T_{2}\right)  \zeta(2p)-\dfrac{3}{4}\wp
(2p)-2\wp(p).
\end{equation}

\end{lemma}

\begin{proof}
The local exponents of \eqref{Fuchsian with T1,T2,B} at the singularities
$z=0,\,p,$ and $-p$ are
\[%
\begin{cases}
~-1,\,2 & \text{(at $z=0$)};\\
~-\frac{1}{2},\,\frac{3}{2} & \text{(at $z=p$)};\\
~-\frac{1}{2},\,\frac{3}{2} & \text{(at $z=-p$)}.
\end{cases}
\]

Accordingly, the singularities at $z=0,\pm p$ are apparent precisely when
there exists a local solution corresponding to the smaller exponent at each
singularity---namely $-1$ at $z=0$ and $-\frac{1}{2}$ at $z=\pm p$. These
solutions admit the following local expansions:
\begin{equation}%
\begin{split}
y_{0}(z) &  =\sum_{m=0}^{\infty}c_{m}z^{m-1},\quad c_{0}=1,\\
y_{\pm p}(z) &  =\sum_{m_{\pm p}=0}^{\infty}d_{m_{\pm p}}\left(  z\pm
p\right)  ^{m_{\pm p}-\frac{1}{2}},\quad d_{0}=1.
\end{split}
\label{0806equ1}%
\end{equation}
Substituting the expansions in \eqref{0806equ1} into
\eqref{Fuchsian with T1,T2,B}$_{p}$ and matching coefficients yield recursive
relations for the coefficients $c_{m}$ and $d_{m_{\pm p}}$. These recursion
relations impose algebraic constraints on the parameters  $T_{1},T_{2},B$,
which must satisfy the conditions given in \eqref{apparent condition, n=1} and \eqref{apparent condition, B}.

This completes the proof.
\end{proof}

Lemma \ref{apparent condition lemma} asserts that
\begin{equation}
\label{apparent space, whole}AP(\tau,p)=\left\{  \mathbf{T}=(T_{1}%
,T_{2}):\left(  T_{1}+T_{2}\right)  \left(  T_{1}-T_{2}+\dfrac{\wp
^{\prime\prime} (p)}{2\wp^{\prime}(p)}\right)  =0\right\}  ,
\end{equation}
{\normalsize which naturally decomposes into two components $AP_{i}(\tau,p)$,
$i=0,1$, as follows:
\begin{equation}
\label{apparent space, even, A}AP_{0}(\tau,p)= \left\{  (A,-A):A\in
\mathbb{C}\right\}  ,
\end{equation}
where the corresponding potential $q(z;T_{1},T_{2},B)$ is
\begin{equation}
\label{potential in even case}q(z;A)=\left(
\begin{array}
[c]{l}%
2\wp(z)+\dfrac{3}{4}\left(  \wp(z+p)+\wp(z-p)\right) \\
+A\left(  \zeta(z+p)-\zeta(z-p)\right)  +B
\end{array}
\right)  ,
\end{equation}
with $B$ determined by
\begin{equation}
\label{B,under apparent condition, even}B=A^{2}-\zeta(2p)\,A-\dfrac{3}{4}%
\wp(2p)-2\wp(p);
\end{equation}
\begin{equation}
\label{apparent space, noneven, T}AP_{1}(\tau,p)= \left\{  \left(
T-\dfrac{\wp^{\prime\prime}(p)}{4\wp^{\prime}(p)},\,T+\dfrac{\wp^{\prime
\prime}(p)}{4\wp^{\prime}(p)}\right)  :T\in\mathbb{C}\right\}
\end{equation}
where the potential $q(z;T_{1},T_{2},B)$ is given by
\begin{equation}
\label{potential in noneven case}q(z;T)=\left(
\begin{array}
[c]{c}%
2\wp(z)+\dfrac{3}{4}\left(  \wp(z+p)+\wp(z-p)\right) \\
+T\left(  \zeta(z+p)+\zeta(z-p)-2\zeta(z)\right) \\
-\dfrac{1}{4}\dfrac{\wp^{\prime\prime}(p)}{\wp^{\prime}(p)}\left(
\zeta(z+p)-\zeta(z-p)\right)  +B
\end{array}
\right)  ,
\end{equation}
with $B$ determined by
\begin{equation}
\label{B,under apparent condition, noneven}B=T^{2}+\frac{\wp^{\prime\prime
}(p)}{\,2\wp^{\prime}(p)\,}\,\zeta(p)-\dfrac{1}{2}\wp(p).
\end{equation}
These two components intersect precisely at one singular point
\begin{equation}
\label{intersection point in apparent space}\mathbf{T_{*}}=\left(  -\dfrac
{\wp^{\prime\prime}(p)}{4\wp^{\prime}(p)},\,\dfrac{\wp^{\prime\prime}(p)}%
{4\wp^{\prime}(p)}\right)  .
\end{equation}
}

We now turn to the spectral analysis of
{\normalsize \eqref{Fuchsian with T1,T2,B}$_{p}$. Since the arguments for the
even-symmetry case \eqref{apparent space, even, A} and for the punctured
non-even symmetry case \eqref{apparent space, noneven, T} are essentially
identical\thinspace---\thinspace and our primary focus is the latter\thinspace
---\thinspace we present full details only for the generalized Lam\'{e}-type
equation in the punctured non-even case \eqref{GLE 4}$_{p}$. The corresponding
results for the even-symmetry case are provided in Theorem
\ref{results for even-symmetry case}. }

To proceed, we consider the second symmetric product of equation{\normalsize
\eqref{GLE 4}$_{p}$, which is a third-order Fuchsian equation:
\begin{equation}
\Phi^{\prime\prime\prime}(z)-4q(z;T)\,\Phi^{\prime}(z)-2q^{\prime}%
(z;T)\,\Phi(z)=0,\quad\text{on}\quad E_{\tau}.\label{3rd ODE}%
\end{equation}
}

\begin{theorem}
{\normalsize \label{elliptic solution to 3rd ODE} Up to a nonzero multiple,
there exists a unique non-trivial elliptic solution $\Phi_{e}(z;T)$ of
\eqref{3rd ODE}, given by {\small {%
\begin{equation}
\label{elliptic solution to 3rd ode, non-even case*}\Phi_{e}(z;T)=\left[
\begin{array}
[c]{l}%
\wp(z)-T\left(  \zeta(z+p)+\zeta(z-p)-2\zeta(z)\right) \\
-\dfrac{1}{2}\dfrac{\,\wp^{\prime\prime}(p)\,}{\wp^{\prime}(p)}\left(
\zeta(z+p)-\zeta(z-p)\right) \\
+2\,T^{2}+\dfrac{\,\wp^{\prime\prime}(p)\,}{\wp^{\prime}(p)}\,\zeta(p)-\wp(p)
\end{array}
\right]  .
\end{equation}
} }}
\end{theorem}

{\normalsize Since the local exponents of \eqref{3rd ODE}$_{p}$ at $z=0$ and
$z=\pm p$ are, respectively, $-2,1,4$ and $-1,1,3$, we may assume $\Phi
_{e}(z;T)$ as follows: }{\small {%
\begin{equation}
\label{0806equ2}%
\begin{split}
\Phi_{e}(z;T)=c_{0}\,\wp(z)+c_{1}\left(  \zeta(z+p)-\zeta(z)\right)  +
c_{2}\left(  \zeta(z-p)-\zeta(z)\right)  +c_{3},
\end{split}
\end{equation}
}}{\normalsize for some $c_{0},c_{1},c_{2},c_{3}\in\mathbb{C}$. If $c_{0}=0$,
the local exponent of $\Phi_{e}(z;T)$ at $z=0$ forces $c_{1}=c_{2}=c_{3}=0$;
hence we may assume that $c_{0}=1$. }

{\normalsize Define the elliptic function
\begin{equation}
G(z;T):=\Phi_{e}^{\prime\prime\prime}(z;T)-4\,q(z;T)\,\Phi_{e}^{\prime
}(z;T)-2\,q^{\prime}(z;T)\,\Phi_{e}(z;T),\label{0824equ1}%
\end{equation}
where $\Phi_{e}(z;T)$ is given by \eqref{0806equ2}. Its Laurent expansion at
$z=0$ is
\begin{equation}
G(z;T)=\sum\limits_{j=-4}^{0}d_{j}(T)\,z^{j}+O(z),\label{0824equ2}%
\end{equation}
with coefficients
\begin{align}
&  d_{-4}(T)=-10\left[  2T+c_{1}+c_{2}\right]  ,\label{0824equ4}\\
&  d_{-3}(T)=4\left[
\begin{array}
[c]{c}%
2T^{2}+3T\left(  c_{1}+c_{2}\right)  \\
+2\left(  c_{3}+\left(  c_{1}-c_{2}\right)  \zeta(p)+\wp(p)\right)
\end{array}
\right]  ,\label{0824equ5}\\
&  d_{-2}(T)=-4\left[
\begin{array}
[c]{c}%
T^{2}\left(  c_{1}+c_{2}\right)  \\
+T\left(  3\wp(p)+\left(  c_{1}-c_{2}\right)  \zeta(p)+c_{3}\right)  \\
+\left(  c_{1}+c_{2}\right)  \wp(p)
\end{array}
\right]  ,\label{0824equ6}\\
&  d_{-1}(T)=4\left[
\begin{array}
[c]{c}%
\wp^{\prime\prime}(p)+\left(  c_{1}-c_{2}\right)  \wp^{\prime}(p)
\end{array}
\right]  ,\label{0824equ7}\\
&  d_{0}(T)=2\left[
\begin{array}
[c]{c}%
2T^{2}\left(  c_{1}+c_{2}\right)  \wp(p)\\
+T\left[
\begin{array}
[c]{l}%
2\wp(p)(c_{3}+(c_{1}-c_{2})\zeta(p)+2\wp(p))\\
+3\left(  c_{2}-c_{1}\right)  \wp^{\prime}(p)-\wp^{\prime\prime}(p)
\end{array}
\right]  \\
+\left(  c_{1}+c_{2}\right)  \left(  \wp^{2}(p)+\wp^{\prime\prime}(p)\right)
\end{array}
\right]  .\label{0824equ8}%
\end{align}
}

\begin{proof}
[Proof of Theorem \ref{elliptic solution to 3rd ODE}]{\normalsize With
notations as above. Since $\Phi_{e}(z;T)$ solves \eqref{3rd ODE} if and only
if $G\equiv0$, the coefficients in the Laurent expansion \eqref{0824equ2} must
satisfy $d_{j}(T)=0$ for $-4\le j\le0$. }

{\normalsize From $d_{-4}(T)=0$ and $d_{-1}(T)=0$, we obtain
\begin{equation}%
\begin{split}
~c_{1} &  =-T-\dfrac{1}{2}\dfrac{\wp^{\prime\prime}(p)}{\wp^{\prime}(p)},\\
c_{2} &  =-T+\dfrac{1}{2}\dfrac{\wp^{\prime\prime}(p)}{\wp^{\prime}(p)}.
\end{split}
\label{0824equ9}%
\end{equation}
Substituting \eqref{0824equ9} into $d_{-3}(T)=0$ yields
\begin{equation}
c_{3}=2T^{2}+\dfrac{\wp^{\prime\prime}(p)}{\wp^{\prime}(p)}\zeta
(p)-\wp(p).\label{0824equ10}%
\end{equation}
Furthermore, by \eqref{0824equ9} and \eqref{0824equ10}, $d_{-2}(T)=0$ and
$d_{0}(T)$ hold automatically. Next, we prove that $G(z;T)$ is holomorphic at
$z=\pm p$. A direct computation shows that {\small {%
\begin{equation}
G(z;T)=\frac{\,-3\bigl(\wp^{\prime}(p)+\wp^{\prime}(2p)\bigr)+\dfrac
{9\,\wp(p)\,\wp^{\prime\prime}(p)}{\wp^{\prime}(p)}-\dfrac{3\,\wp
^{\prime\prime3}}{4\wp^{\prime3}}\,}{z-p}\,T+O(1).\label{0825equ1}%
\end{equation}
} By applying the addition formula
\begin{equation}
\wp^{\prime}(2p)=-\wp^{\prime}(p)+\frac{3\,\wp^{\prime\prime}(p)\,\wp(p)}%
{\wp^{\prime}(p)}-\frac{\wp^{\prime\prime3}}{4\,\wp^{\prime3}}%
,\label{addition formula for wp'(2p)}%
\end{equation}
we find that the residue in \eqref{0825equ1} vanishes. Consequently, the
function $G$ is also holomorphic at $z=p$, }}and by symmetry, it is also
holomorphic at $z=-p${\normalsize {\small .}}

This completes the proof.
\end{proof}

{\normalsize
}

{\normalsize Define }{\small {
\begin{equation}
Q(z;T):=\dfrac{1}{2}\Phi_{e}(z;T)\,\Phi_{e}^{\prime\prime}(z;T)-\dfrac{1}%
{4}\left(  \Phi_{e}^{\prime}(z;T)\right)  ^{2}-q(z;T)\,\Phi_{e}^{2}%
(z;T).\label{spectral polynomial, def}%
\end{equation}
} }{\normalsize Since $\Phi_{e}(z;T)$ solves \eqref{3rd ODE}, a direct
computation shows that $\frac{d}{dz}Q(z;T)\equiv0$. Hence, }$Q(z;T)$ is
independent of $z$, and we may denote it simply by $Q(T)$.
{\normalsize Moreover, as $q(z;T)$ and $\Phi_{e}(z;T)$ are polynomial in $T$,
$Q(T)$ is then a polynomial in $T$. A direct computation shows that
\begin{equation}
Q(T)=-4\left(  T^{2}-2\wp(p)-e_{1}\right)  \left(  T^{2}-2\wp(p)-e_{2}\right)
\left(  T^{2}-2\wp(p)-e_{3}\right)
.\label{spectral polynomial in noneven case}%
\end{equation}
}

By analogy with the KdV hierarchy, we refer to{\normalsize  $Q(T)$ as the
\textit{spectral polynomial}. The associated \textit{spectral curve} is then
defined by
\begin{equation}
\Gamma(\tau,p):=\left\{  \,(\mathbf{T},C)\ |\ C^{2}=Q(T)\,\right\}
.\label{spectral curve}%
\end{equation}
For $P=(T,C)\in\Gamma(\tau,p)$, we define its \textit{dual point} by $P^{\ast
}:=(T,-C)\in\Gamma(\tau,p)$. }

The spectral polynomial and corresponding spectral curve for the classical
{\normalsize Lam\'{e} }equation {\normalsize \eqref{Lame 1} were derived in
\cite{Maier} and are given by
\begin{equation}
Q(\tilde{B})=-4\bigl(\tilde{B}-e_{1}\bigr)\bigl(\tilde{B}-e_{2}%
\bigr)\bigl(\tilde{B}-e_{3}%
\bigr),\label{spectral polynomial for classical Lame}%
\end{equation}
with the associated spectral curve defined as
\begin{equation}
\tilde{\Gamma}(\tau):=\{\,(\tilde{B},\tilde{C})\mid\tilde{C}^{2}=Q(\tilde
{B})\,\}.\label{spectral curve for Lame 1}%
\end{equation}
Surprisingly, in light of \eqref{spectral polynomial in noneven case} and
\eqref{spectral polynomial for classical Lame}, there exists a two-to-one
correspondence between the two spectral curves:
\begin{equation}%
\begin{split}
\Gamma(\tau,p) &  \longrightarrow\tilde{\Gamma}(\tau)\\
P &  \longmapsto\tilde{P},\qquad\tilde{B}(\tilde{P})=T^{2}(P)-2\wp(p).
\end{split}
\label{correspondence, spectral curve}%
\end{equation}
This correspondence reflects the monodromy equivalence between the classical
Lam\'{e} equation \eqref{Lame 1} and the generalized Lam\'{e}-type equation
\eqref{GLE 4}$_{p}$ as stated in Theorem \ref{Main Theorem1}. }

{\normalsize
}

{\normalsize Next, we define the \textit{Baker-Akhiezer function} for each
$P\in\Gamma(\tau,p)$ as follows: For any $P=(T(P),C(P))\in\Gamma(\tau,p)$,
define the meromorphic function
\begin{equation}
\phi(P;z):=\dfrac{\,i\,C(P)+\frac{1}{2}\Phi_{e}^{\prime}(z;T)\,}{\Phi
_{e}(z;T)},\quad z\in\mathbb{C}.\label{mero phi(P;z)}%
\end{equation}
A direct computation shows that $\phi(P;z)$ satisfies the \textit{Riccati
equation}:
\begin{equation}
\phi^{\prime}(P;z)=q(z;T)-\phi^{2}(P;z).\label{riccati equ}%
\end{equation}
}

\begin{proposition}
{\normalsize \label{residue of mero phi} Let $P=(T,C)\in\Gamma(\tau,p)$.
Suppose $z_{0}$ is a pole of $\phi(P;z)$. }

\begin{itemize}
\item[(i)] {\normalsize If $z_{0}\in\{0,p,-p\}$, one has
\[
\operatorname{Res}_{z=0}\phi(P;z)=-1,\quad\operatorname{Res}_{z=\pm p}%
\phi(P;z)=-\frac{1}{2}.
\]
}

\item[(ii)] {\normalsize If $z_{0}\notin\{0,p,-p\}$, then
\[
\operatorname{Res}_{z=z_{0}}\phi(P;z)=1.
\]
}
\end{itemize}
\end{proposition}

\begin{proof}
{\normalsize From \eqref{mero phi(P;z)}, the poles of $\phi(P;z)$ are simple
and arise either from the poles or from the zeros of $\Phi_{e}(z;T)$. If
$z_{0}$ is a pole of $\Phi_{e}(z;T)$, i.e. $z_{0}\in\{0,\pm p\}$, then since
\[
\operatorname{ord}_{z=0}\Phi_{e}(z;T)=-2, \quad\operatorname{ord}_{z=\pm
p}\Phi_{e}(z;T)=-1,
\]
it follows that
\[
\operatorname{Res}_{z=0}\phi(P;z)=-1, \quad\operatorname{Res}_{z=\pm p}%
\phi(P;z)=-\tfrac{1}{2}.
\]
If $\Phi_{e}(z;T)$ vanishes at some point $z_{0}\notin\{0,\pm p\}$, then
$z_{0}$ must be a simple zero of $\Phi_{e}(z;T)$. Substituting into the
Riccati equation \eqref{riccati equ} yields $\operatorname{Res}_{z=z_{0}}%
\phi(P;z)=1$. This completes the proof. }
\end{proof}

{\normalsize Given a base point $z_{0}\notin\{0,\pm p\}$, we define the
\textit{Baker-Akhiezer} function at $P\in\Gamma(\tau,p)$ by
\begin{equation}
\psi(P;z,z_{0}):=\exp\left(  \,\int_{z_{0}}^{z}\phi(P;\xi)\,d\xi\right)
,\quad z\in\mathbb{C},\label{Baker-Akhiezer function}%
\end{equation}
where the integration path }is chosen to avoid {\normalsize the singularities
of the meromorphic function $\phi(P;\xi)$. By Proposition
\ref{residue of mero phi}, $\psi(P;z,z_{0})$ is a multi-valued meromorphic
function whose analytic continuation around $\pm p$ produces nontrivial
monodromy, reflecting the presence of branch points at $\pm p$. }

{\normalsize By the Riccati equation \eqref{riccati equ}, it is easy to verity
that $\psi(P;z,z_{0})$ solves the generalized Lam\'{e}-type equation
\eqref{GLE 4}$_{p}$ with parameter $T(P)$, where $T=T(P)$ is the
$T$-coordinate of $P\in\Gamma(\tau,p)$. Thus, each point \textit{$P\in
\Gamma(\tau,p)$ }}corresponds to an apparent generalized Lam\'{e}-type
equation \eqref{GLE 4}$_{p}$ with parameter {\normalsize \textit{$T=T(P)$.} }

{\normalsize Recall that $P^{*}\in\Gamma(\tau,p)$ denotes the dual point of
$P$. Since $T(P)=T(P^{*})$, the Baker-Akhiezer function $\psi(P^{*};z,z_{0})$
also solves the same equation \eqref{GLE 4}$_{p}$ with parameter
$T(P)=T(P^{*})$. Hence $P$ and $P^{*}$ represent the same equation
\eqref{GLE 4}$_{p}$. }

{\normalsize In the following, we collect some properties of $\psi(P;z,z_{0})$
and $\psi(P^{\ast};z,z_{0})$ associated with the equation \eqref{GLE 4}$_{p}$.
The proofs can be found in \cite{Kuo}. Firstly,
\begin{equation}
\psi(P;z,z_{0})\,\psi(P^{\ast};z,z_{0})=\dfrac{\Phi_{e}(z;T)}{\Phi_{e}%
(z_{0};T)}.\label{properties for BA,1}%
\end{equation}
Moreover, their Wronskian is given by
\begin{equation}
W\!\left(  \psi(P;z,z_{0}),\,\psi(P^{\ast};z,z_{0})\right)  =\dfrac
{2i\,C}{\Phi_{e}(z_{0};T)},\label{properties for BA,2}%
\end{equation}
where the Wronskian is defined as $W(f,g):=f^{\prime}(z)g(z)-f(z)g^{\prime
}(z)$.}

{\normalsize Different choices of $z_{0}\notin\{0,\pm p\}$ change
$\psi(P;z,z_{0})$ and $\psi(P^{*};z,z_{0})$ only by nonzero multiples. We
therefore omit the notation $z_{0}$ and simply write
\[
\psi(P;z,z_{0})=\psi(P;z),\qquad\psi(P^{*};z,z_{0})=\psi(P^{*};z).
\]
}

{\normalsize Let $P=(T,C)\in\Gamma(\tau,p)$. For $j=1,2$, set
\begin{equation}
\label{0806equ4}%
\begin{split}
\lambda_{j}(P,z)  &  =\exp\left(  \,\int_{z}^{z+\omega_{j}}\phi(P;\xi
)\,d\xi\right) \\
&  = \exp\left(  \,\int_{z}^{z+\omega_{j}}\dfrac{\,i\,C+\frac{1}{2}\Phi
_{e}^{\prime}(\xi;T)\,}{\Phi_{e}(\xi;T)}\,d\xi\right)  ,
\end{split}
\end{equation}
where the integration path is the fundamental cycle from $z$ to $z+\omega_{j}$
avoiding poles and zeros of $\Phi_{e}(z;T)$. Since $\Phi_{e}(z;T)$ is
elliptic, the quantities $\lambda_{j}(P,z)$ are independent of $z$, and we may
write:
\[
\lambda_{1}(P,z)=\lambda_{1}(P)\quad\text{and}\quad\lambda_{2}(P,z)=\lambda
_{2}(P).
\]
From the definition of Baker-Akhiezer functions, it follows that
\begin{equation}
\label{BA, elliptic of second kind}\psi(P;z+\omega_{j})=\lambda_{j}
(P)\,\psi(P;z),\quad\text{for }\, j=1,2.
\end{equation}
Thus $\psi(P;z)$ and $\psi(P^{*};z)$ are elliptic functions of the second
kind. Moreover, by \eqref{properties for BA,1},
\begin{equation}
\label{lamba*=lambda^{-1}.}\lambda_{j}(P^{*})=\dfrac{1}{\,\lambda_{j}
(P)\,},\quad\text{for }\,j=1,2.
\end{equation}
}

{\normalsize The main result of this section is the following. }

\begin{theorem}
{\normalsize \label{main thm for spectral theory} The generalized
Lam\'{e}-type equation in the punctured non-even symmetry case
\eqref{GLE 4}$_{p}$ is completely reducible if and only if $Q(T)\neq0$. }
\end{theorem}

\begin{proof}
{\normalsize Let $P=(T,C)\in\Gamma(\tau,p)$. By \eqref{properties for BA,2},
the Baker-Akhiezer functions $\psi(P;z)$ and $\psi(P^{*};z)$ are linearly
independent if and only if $Q(T)\neq0$. Assume first that \eqref{GLE 4}$_{p}$
is completely reducible. If $Q(T)=0$, then $\psi(P;z)$ and $\psi(P^{*};z)$ are
linearly dependent, i.e.
\begin{equation}
\label{0806equ6}\psi(P;z)=c\,\psi(P^{*};z),\quad\text{for some $c\neq0$.}%
\end{equation}
Complete reducibility guarantees the existence of two linearly independent
solutions $y_{1}(z;T)$, $y_{2}(z;T)$ such that
\[
\left(
\begin{array}
[c]{c}%
y_{1}(z+\omega_{j};T)\\
y_{2}(z+\omega_{j};T)
\end{array}
\right)  = \left(
\begin{array}
[c]{cc}%
\widehat{\lambda_{j}}(P) & 0\\
0 & \dfrac{1}{\,\widehat{\lambda_{j}}(P)\,}%
\end{array}
\right)  \left(
\begin{array}
[c]{c}%
y_{1}(z;T)\\
y_{2}(z;T)
\end{array}
\right)  .
\]
Then $y_{1}(z;T)\,y_{2}(z;T)$ is an elliptic solution of the equation
\eqref{3rd ODE}. By Theorem \ref{elliptic solution to 3rd ODE} and
\eqref{properties for BA,1},
\begin{equation}
\label{0806equ7}y_{1}(z;T)\,y_{2}(z;T)=\Phi_{e}(z;T)=\psi(P;z)\,\psi
(P^{*};z)=c\,\psi^{2}(P;z).
\end{equation}
If $a$ is a zero of $\psi(P;z)$, then from \eqref{0806equ7} we may assume
$y_{2}(a;T)=0$. Thus $y_{2}(z)=\psi(P;z)$ up to a nonzero multiple, and
together with \eqref{0806equ7} this also implies $y_{1}(z)=\psi(P;z)$, a
contradiction. Hence $Q(T)\neq0$. }

{\normalsize Conversely, if $Q(T)\neq0$, then $\psi(P;z)$ and $\psi(P^{*};z)$
are linearly independent. Together with \eqref{BA, elliptic of second kind}
and \eqref{lamba*=lambda^{-1}.}, this shows that \eqref{GLE 4}$_{p}$ is
completely reducible. }
\end{proof}

To conclude this section, we present the corresponding results for the
generalized {\normalsize Lam\'{e}-type equation in the even-symmetric
setting:
\begin{equation}
y^{\prime\prime}(z)=q(z;A)\,y(z),\label{0825,even ode}%
\end{equation}
where $q(z;A)$ is given in \eqref{potential in even case} and $B$ is
determined by the apparent condition
\eqref{B,under apparent condition, even}.}

These results are summarized in the following theorem.

\begin{theorem}
{\normalsize \label{results for even-symmetry case} ~\vspace{1pt} }

{\normalsize \textit{(i)} Up to a nonzero multiple, there exists a unique
non-trivial elliptic solution $\Phi_{e}(z;A)$ of the second symmetric product
equation with respect to $q(z;A)$. The elliptic solution $\Phi_{e}(z;A)$ is
precisely given by {\small {%
\begin{equation}
\label{elliptic solution to 3rd ode, even case*}\Phi_{e}(z;A)=\left[
\begin{array}
[c]{l}%
\wp(z)+ \left(  \dfrac{A}{2}-\dfrac{3}{8}\dfrac{\wp^{\prime\prime}(p)}%
{\wp^{\prime}(p)}\right)  \left(  \zeta(z+p)-\zeta(z-p)\right) \\[10pt]%
-A^{2}+\left(  \dfrac{\wp^{\prime\prime}(p)}{\,2\wp^{\prime}(p)\,}%
-\zeta(p)\right)  A\\[10pt]%
-\wp(p)+\dfrac{3}{4}\dfrac{\,\wp^{\prime\prime}(p)\,}{\wp^{\prime}(p)}%
\zeta(p)+\dfrac{3}{16}\dfrac{\,\wp^{\prime\prime}(p)^{2}\,}{\wp^{\prime
}(p)^{2}}%
\end{array}
\right]  .
\end{equation}
} ~\vspace{1pt} }}

{\normalsize {\small \textit{(ii)} The spectral polynomial $Q(A)$ is expressed
as
\begin{equation}
\label{spectral polynomial in even case}Q(A)=-Y_{1}(A)\, Y_{2}(A),
\end{equation}
where
\begin{align*}
Y_{1}(A)  &  =A^{3} -\frac{5 \wp^{\prime\prime}(p)}{4\wp^{\prime}(p)}A^{2}
+3\left(  \wp(p) +\frac{1}{16}\frac{\wp^{\prime\prime}(p)^{2}}{\wp^{\prime
}(p)^{2}}\right)  A\\
&  \quad\quad+\frac{\wp^{\prime}(p)}{2} -\frac{9\wp^{\prime\prime}(p)}%
{4\wp^{\prime}(p)}\wp(p) +\frac{9}{64}\frac{\wp^{\prime\prime}(p)^{3}}%
{\wp^{\prime}(p)^{3}},\\
Y_{2}(A)  &  =A^{3} - \frac{\wp^{\prime\prime}(p)}{4\wp^{\prime}(p)}A^{2} -
\frac{ 5 }{16}\frac{\wp^{\prime\prime}(p)^{2}}{\wp^{\prime}(p)^{2}}A +
2\wp^{\prime}(p) - \frac{3}{64}\frac{\wp^{\prime\prime}(p)^{3}}{\wp^{\prime}(p)^{3}} .
\end{align*}
~\vspace{1pt} }}

{\normalsize {\small \textit{(iii)} The generalized Lam\'{e}-type equation in
the even-symmetry case \eqref{0825,even ode}$_{p}$ is completely reducible if
and only if $Q(A)\neq0$. }}
\end{theorem}


\section{Monodromy Theory and The Proofs}

{\normalsize \label{Monodromy Theory and The Proofs Section} }

{\normalsize In this section we analyze the monodromy representations of of
\eqref{GLE 4}$_{p}$ and present a complete proof of Theorem
\ref{Main Theorem1}. }

\subsection{Completely Reducible Case}

{\normalsize We establish the relation between complete reducibility and the
monodromy data for the generalized Lam\'{e}-type equation \eqref{GLE 4}$_{p}$.
}

{\normalsize Let $P\in\Gamma(\tau,p)$ and recall $\lambda_{1}(P),\,\lambda
_{2}(P)$ from \eqref{0806equ4}. Define $(r(P),s(P))\in\mathbb{C}%
^{2}/\mathbb{Z}^{2}$ associated with \eqref{GLE 4}$_{p}$ by
\begin{equation}
\label{defi,r(P),s(P)}\lambda_{1}(P)=e^{-2\pi is(P)}\quad\text{and}%
\quad\lambda_{2}(P)=e^{2\pi ir(P)}.
\end{equation}
It follows from \eqref{lamba*=lambda^{-1}.} and \eqref{defi,r(P),s(P)} that
\begin{equation}
\label{r(P*),s(P*)}\left(  r(P^{*}),s(P^{*})\right)  =\left(
-r(P),-s(P)\right)  .
\end{equation}
}

{\normalsize Complete reducibility is characterized by the pair $(r(P),s(P))$
as follows. }

\begin{theorem}
{\normalsize \label{Q(T) neq 0 iff. r,s notin Z/2} Let $P=(T,C)\in\Gamma
(\tau,p)$. Then $Q(T)\neq0$ if and only if $\left(  r(P),s(P)\right)
\notin\frac{1}{2}\mathbb{Z}^{2}$. Namely, the generalized Lam\'{e}-type
equation \eqref{GLE 4}$_{p}$ is completely reducible if and only if $\left(
r(P),s(P)\right)  \notin\frac{1}{2}\mathbb{Z}^{2}$. }
\end{theorem}

{\normalsize We remark that Theorem \ref{Q(T) neq 0 iff. r,s notin Z/2} also
holds for the generalized Lam\'{e}-type equation in the even-symmetry case
\eqref{0825,even ode}$_{p}$. Here, we provide the proof for the punctured
non-even symmetry case. }

\begin{proof}
[Proof of Theorem \ref{Q(T) neq 0 iff. r,s notin Z/2}]{\normalsize Suppose
that $Q(T)\neq0$. Then the Baker-Akhiezer functions $\psi(P;z)$ and
$\psi(P^{*};z)$ are linearly independent. If $(r(P),s(P))\in\tfrac{1}%
{2}\mathbb{Z}^{2}$,
\[
\bigl(\lambda_{1}(P),\lambda_{2}(P)\bigr)\in\{(1,1),(1,-1),(-1,1),(-1,-1)\}.
\]
In this case, the function
\[
\psi^{2}(P;z)+\psi^{2}(P^{*};z)
\]
is also an elliptic solution of \eqref{3rd ODE}. By Theorem
\ref{elliptic solution to 3rd ODE}, it follows that
\[
\psi(P;z)\psi(P^{*};z)=\Phi_{e}(z;T) =\psi^{2}(P;z)+\psi^{2}(P^{*};z),
\]
up to a nonzero multiple. Hence every zero of $\psi(P;z)$ is also a zero of
$\psi(P^{*};z)$, which contradicts to the linear independence of $\psi(P;z)$
and $\psi(P^{*};z)$. Therefore, $(r(P),s(P))\notin\tfrac{1}{2}\mathbb{Z}^{2}$.
}

{\normalsize Conversely, assume $(r(P),s(P))\notin\tfrac{1}{2}\mathbb{Z}^{2}$.
Without loss of generality, we may assume that $\lambda_{1}(P)\neq\pm1$. Then
\[
\lambda_{1}(P)\neq\lambda_{1}(P^{*})=\lambda_{1}(P)^{-1},
\]
which implies that $\psi(P;z)$ and $\psi(P^{*};z)$ are linearly independent by
\eqref{lamba*=lambda^{-1}.}. Hence $Q(T)\neq0$. }
\end{proof}

{\normalsize
The monodromy problem for the classical Lam\'{e} equation \eqref{Lame 1} has
been analyzed in \cite{LW,Maier}, as follows. \vspace{3pt} }

{\normalsize \noindent\textbf{Theorem I.}\textit{
({\normalsize \cite{LW,Maier}})} \textit{Given $(r,s)\in\mathbb{C}%
^{2}\setminus\mathbb{Z}^{2}$. There exists $\tilde{P}\in\tilde{\Gamma}(\tau)$
such that the classical Lam\'{e} equation \eqref{Lame 1} has $(r(\tilde
{P}),s(\tilde{P}))=(r,s)$ if and only if there exists $\tilde{B}\in\mathbb{C}$
such that
\begin{equation}%
\begin{cases}
~\wp(r+s\tau)=\tilde{B},\\
~\wp^{\prime}(r+s\tau)=\sqrt{-Q(\tilde{B})\,},\\
~Z(r,s,\tau)=0,
\end{cases}
\label{system of wp,wp',kappa, for Lame n=1}%
\end{equation}
where $Q(\tilde{B})$ and $\tilde{\Gamma}(\tau)$, defined in
\eqref{spectral polynomial for classical Lame} and
\eqref{spectral curve for Lame 1}, denote respectively the spectral polynomial
and spectral curve of \eqref{Lame 1}. Moreover, the solvability of the system
\eqref{system of wp,wp',kappa, for Lame n=1} is further reduced to the third
equation $Z(r,s,\tau)=0$.} \vspace{8pt} }

{\normalsize For the last statement, suppose $Z(r,s,\tau)=0$. Let $\tilde
{B}:=\wp(r+s\tau)$. Then the second equation of
\eqref{system of wp,wp',kappa, for Lame n=1} follows from the classical
differential equation
\begin{equation}
\label{classical differential equation}\wp^{\prime2}=4\bigl(\wp(z)-e_{1}%
\bigr)\bigl(\wp(z)-e_{2}\bigr)\bigl(\wp(z)-e_{3}\bigr).
\end{equation}
}

{\normalsize The following result establishes the monodromy problem of the
generalized Lam\'{e}-type equation in the punctured non-even symmetry case
\eqref{GLE 4}$_{p}$. }

\begin{theorem}
{\normalsize \label{main theorem for Lame-type ODE} Given $(r,s)\in
\mathbb{C}^{2}\setminus\mathbb{Z}^{2}$. There exists $P\in\Gamma(\tau)$ such
that the generalized Lam\'{e}-type equation \eqref{GLE 4}$_{p}$ has
$(r(P),s(P))=(r,s)$ if and only if there exists $T\in\mathbb{C}$ such that
\begin{equation}
\label{system of wp,wp',kappa*}%
\begin{cases}
~\wp(r+s\tau)=T^{2}-2\wp(p),\\
~\wp^{\prime}(r+s\tau)=\sqrt{-Q(T)\,},\\
~Z(r,s,\tau)=0.
\end{cases}
\end{equation}
}
\end{theorem}

{\normalsize
}

{\normalsize In view of \eqref{system of wp,wp',kappa, for Lame n=1} and
\eqref{system of wp,wp',kappa*}, there exists a two-to-one correspondence
between the parameters of the two systems:
\begin{equation}
\mathbb{C}\longrightarrow\mathbb{C},\qquad T\longmapsto\tilde{B}:=T^{2}%
-2\wp(p).\label{classical Lame to Lame-type's correspondence}%
\end{equation}
}This transformation establishes a connection between the generalized
{\normalsize Lam\'{e}-type equation \eqref{GLE 4}$_{p}$ }and the classical
{\normalsize Lam\'{e} equation \eqref{Lame 1}. }Applying this correspondence
together with Theorem I and Theorem
{\normalsize \ref{main theorem for Lame-type ODE}}, we obtain the following result.

\begin{theorem}
{\normalsize Given $(r,s)\in\mathbb{C}^{2}\setminus\frac{1}{2}\mathbb{Z}^{2}$.
Then there exists $P=(T(P),C(P))\in\Gamma(\tau,p)$ such that the generalized
Lam\'{e}-type equation \eqref{GLE 4}$_{p}$ is completely reducible with
monodromy data $(r(P),s(P))=(r,s)$, if and only if there exists $\tilde
{P}=(\tilde{B}(\tilde{P}),\tilde{C}(\tilde{P}))\in\tilde{\Gamma}(\tau)$ such
that the classical Lam\'{e} equation \eqref{Lame 1} is completely reducible
with monodromy data $(r(\tilde{P}),s(\tilde{P}))=(r,s)$. Moreover, the
parameters }$T(P)$ for {\normalsize \eqref{GLE 4}$_{p}$ }and $\tilde{B}%
(\tilde{P})$ for {\normalsize \eqref{Lame 1} }are related via the
correspondence:{\normalsize
\[
\tilde{B}(\tilde{P})=T(P)^{2}-2\wp(p).
\]
Here $r(\tilde{P}),s(\tilde{P})$ is determined by the monodromy of the
Baker--Akhiezer functions associated with the classical Lam\'{e} equation
\eqref{Lame 1} at $\tilde{P}$. }
\end{theorem}

{\normalsize To prove Theorem \ref{main theorem for Lame-type ODE}, we analyze
the monodromy representation via the Baker-Akhiezer functions. Let
$P=(T,C)\in\Gamma(\tau,p)$. From \eqref{BA, elliptic of second kind} and
\eqref{defi,r(P),s(P)}, we obtain
\begin{equation}
\label{BA, elliptic of second kind, with r(P),s(P)}%
\begin{split}
\psi(P;z+1)  &  =e^{-2\pi is(P)}\,\psi(P;z),\\
\psi(P;z+\tau)  &  =e^{2\pi ir(P)}\,\psi(P;z).
\end{split}
\end{equation}
Since $\psi(P;z)$ is an elliptic function of the second kind, it can be
expressed as follows:
\begin{equation}
\label{BA function written in Hermite}\psi(P;z)=e^{c(P)z}\, \dfrac
{\sigma(z-a_{1}(P))\,\sigma(z-a_{2}(P))}{\sigma(z)\,\sqrt{\sigma
(z+p)\,\sigma(z-p)}},
\end{equation}
where $c(P)\in\mathbb{C}$ and $a_{1}(P),a_{2}(P)\in E_{\tau}\setminus\{0\}$
are the zeros of $\psi(P;z)$. }

{\normalsize The following proposition shows that $(r(P),s(P))$ can be
determined by $a_{1}(P),a_{2}(P)$, and $c(P)$, and vice versa. }

\begin{proposition}
{\normalsize \label{0807prop1} From
\eqref{BA, elliptic of second kind, with r(P),s(P)} and
\eqref{BA function written in Hermite}, it follows that
\begin{equation}
\label{algebraic relations between a1(P),a2(P),r(P),s(P),c(P)}%
\begin{cases}
~r(P)+s(P)\tau= a_{1}(P)+a_{2}(P),\\[5pt]%
~r(P)\,\eta_{1}(\tau)+s(P)\,\eta_{2}(\tau)=c(P).
\end{cases}
\end{equation}
}
\end{proposition}

\begin{proof}
{\normalsize Recall the transformation law of the Weierstrass }$\sigma
${\normalsize  function:
\begin{equation}
\sigma(z+\omega_{j})=-e^{\eta_{j}\left(  z+\frac{\omega_{j}}{2}\right)
}\sigma(z),\quad\text{for }%
\,j=1,2.\label{transformation law of the Weierstrass sigma function}%
\end{equation}
From (\ref{BA function written in Hermite}), we obtain
\[
\psi(P;z+\omega_{j})=e^{c(P)\omega_{j}-\eta_{j}\left(  a_{1}(P)+a_{2}%
(P)\right)  }\psi(P;z),\quad\text{for }\,j=1,2.
\]
Comparing with \eqref{BA, elliptic of second kind, with r(P),s(P)}, we deduce
\begin{equation}%
\begin{cases}
~c(P)-\eta_{1}\left(  a_{1}(P)+a_{2}(P)\right)  =-2\pi i\,s(P),\\
~c(P)\tau-\eta_{2}\left(  a_{1}(P)+a_{2}(P)\right)  =2\pi i\,r(P).
\end{cases}
\label{0807equ1}%
\end{equation}
Finally, recalling the Legendre relation
\begin{equation}
\tau\eta_{1}-\eta_{2}=2\pi i,\label{Legendre relation}%
\end{equation}
we obtain \eqref{algebraic relations between a1(P),a2(P),r(P),s(P),c(P)}. }
\end{proof}

{\normalsize Next, we derive the algebraic equations for the zeros $a_{1}(P)$
and $a_{2}(P)$. Define
\begin{equation}
\label{ya}y_{a}(z):=e^{cz}\dfrac{\sigma(z-a_{1}) \sigma(z-a_{2})}
{\,\sigma(z)\sqrt{\sigma(z+p)\sigma(z-p)}\,},
\end{equation}
where $c\in\mathbb{C}$, $a_{1}, a_{2}\in E_{\tau}\setminus\{0\}$. }

\begin{theorem}
{\normalsize \label{a1,a2 alge equ} Let $y_{a}(z)$ be defined in \eqref{ya}. }

\begin{enumerate}
{\normalsize [(i)] }

\item {\normalsize Suppose $a_{1}\neq a_{2}\in E_{\tau}\setminus\{0,\pm p\}$.
Then the function $y_{a}(z)$ is a solution to \eqref{GLE 4}$_{p}$ for some
$T\in\mathbb{C}$, with $B$ given by \eqref{B condition 2}, if and only if
\begin{equation}
\label{a1 and a2's condition}\wp(a_{1})+\wp(a_{2})=2\wp(p).
\end{equation}
In this case, the constants $c$ and $T$ are determined by
\begin{align}
c  &  =\zeta(a_{1}+a_{2}),\label{c determined by a1, a2}\\
T  &  =\dfrac{\wp^{\prime}(a_{1})-\wp^{\prime}(a_{2})}{2\left(  \wp(a_{1}
)-\wp(a_{2})\right)  }. \label{T determined by a1, a2}%
\end{align}
}

\item {\normalsize Suppose $a_{1}=a_{2}=\pm p$. Then $y_{a}(z)$ is a solution
to \eqref{GLE 4}$_{p}$ for some $T\in\mathbb{C}$, with $B$ given by
\eqref{B condition 2}, if and only if
\begin{align}
c  &  =\pm\zeta(2p),\label{c determined by a1=a2=pm p}\\
T  &  =\pm\dfrac{\wp^{\prime\prime}(p)}{2\wp^{\prime}(p)}.
\label{T determined by a1=a2=pm p}%
\end{align}
}
\end{enumerate}
\end{theorem}

\begin{proof}
{\normalsize From the definition of $y_{a}(z)$ in \eqref{ya}, we obtain
\begin{align*}
\dfrac{y_{a}^{\prime}(z)}{y_{a}(z)}  &  =c+\zeta(z-a_{1})+\zeta(z-a_{2}%
)-\zeta(z)-\dfrac{1}{2}\left(  \zeta(z+p)+\zeta(z-p)\right)  ,\\
\left(  \dfrac{y_{a}^{\prime}(z)}{y_{a}(z)}\right)  ^{\prime}  &
=\wp(z)+\dfrac{1}{2}\left(  \wp(z+p)+\wp(z-p)\right)  -\wp(z-a_{1}
)-\wp(z-a_{2}).
\end{align*}
Define the elliptic function
\[
g(z):=\left(  \dfrac{y_{a}^{\prime}(z)}{y_{a}(z)}\right)  ^{\prime}+\left(
\dfrac{y_{a}^{\prime}(z)}{y_{a}(z)}\right)  ^{2}-q(z;T).
\]
Then $y_{a}(z)$ solves \eqref{GLE 4}$_{p}$ if and only if $g(z)\equiv0$.
\newline\textbf{Case (i).} $a_{1}\neq a_{2}\in E_{\tau}\setminus\{0,\pm p\}$.
}

{\normalsize Expanding $g(z)$ at $z=0$ yields
\begin{equation}
\label{0822equ5'}%
\begin{split}
g(z)  &  =\dfrac{\,2\left(  -c+T+\zeta(a_{1})+\zeta(a_{2})\right)  \,}{z}\\
&  \quad+\left(
\begin{array}
[c]{c}%
c^{2}-T^{2}-2\wp(p)+\wp(a_{1})+\wp(a_{2})\\
-\left(  2c-\zeta(a_{1})-\zeta(a_{2})\right)  \left(  \zeta(a_{1})+\zeta
(a_{2})\right)
\end{array}
\right)  +O(z).
\end{split}
\end{equation}
For $g$ to be holomorphic with $g(0)=0$, we must have
\begin{equation}
\label{0822equ6}%
\begin{cases}
~c=T+\zeta(a_{1})+\zeta(a_{2}),\\
~\wp(p)=\dfrac{1}{\,2\,}\left(  \wp(a_{1})+\wp(a_{2})\right)  .
\end{cases}
\end{equation}
Under \eqref{0822equ6}, the expansions of $g$ at $z=a_{1},a_{2}$ are
\begin{equation}
\label{0822equ7}%
\begin{split}
g(z)  &  =\dfrac{\,2T-\dfrac{\,\wp^{\prime}(a_{1})-\wp^{\prime}(a_{2})\,}%
{\wp(a_{1})-\wp(a_{2})}\,}{z-a_{1}}+O(1),\\
g(z)  &  =\dfrac{\,2T-\dfrac{\,\wp^{\prime}(a_{1})-\wp^{\prime}(a_{2})\,}%
{\wp(a_{1})-\wp(a_{2})}\,}{z-a_{2}}+O(1).\\
\end{split}
\end{equation}
Thus, holomorphicity at $a_{1},a_{2}$ requires
\begin{equation}
\label{0822equ8}T=\dfrac{\wp^{\prime}(a_{1})-\wp^{\prime}(a_{2})}{\,2\left(
\wp(a_{1})-\wp(a_{2})\right)  \,}.
\end{equation}
With \eqref{0822equ6} and \eqref{0822equ8}, we see that
\[
g(z)=\dfrac{\,-2T+\dfrac{\,\wp^{\prime}(a_{1})-\wp^{\prime}(a_{2})\,}%
{\wp(a_{1})-\wp(a_{2})}\,}{z- p}+O(1)
\]
near $z= p$. By \eqref{0822equ8}, $g$ is also holomorphic at $z= p$.
Consequently, $y_{a}(z)$ solves \eqref{GLE 4}$_{p}$ if and only if
\eqref{0822equ6} and \eqref{0822equ8} hold, with $c$ determined by
\begin{align*}
c  &  =T+\zeta(a_{1})+\zeta(a_{2})\\
&  =\dfrac{\wp^{\prime}(a_{1})-\wp^{\prime}(a_{2})}{\,2\left(  \wp(a_{1}%
)-\wp(a_{2})\right)  \,}+\zeta(a_{1})+\zeta(a_{2})=\zeta(a_{1}+a_{2}),
\end{align*}
which establishes \eqref{a1 and a2's condition},
\eqref{c determined by a1, a2}, and \eqref{T determined by a1, a2}.
\newline\textbf{Case (ii).} $a_{1}=a_{2}\in\{\pm p\}$. The case $a_{1}%
=a_{2}=-p$ is obtained by replacing $p$ with $-p$. Here, we provide the proof
for $a_{1}=a_{2}=p$. }

{\normalsize Expanding $g(z)$ at $z=0$ gives
\begin{equation}
\label{0822equ9}%
\begin{split}
g(z)  &  =\dfrac{\,-2\left(  c-T-2\zeta( p)\right)  \,}{z}\\
&  \quad+\left(  c-T-2\zeta( p)\right)  \left(  c+T-2\zeta p)\right)  +O(z).
\end{split}
\end{equation}
Thus holomorphicity requires
\begin{equation}
\label{0822equ10}c=T+2\zeta( p).
\end{equation}
In this case, $g(0)=0$ holds automatically. Expanding further at $z=p$ yields
\begin{equation}
\label{0822equ10*}g(z)=\dfrac{\, 2T-\frac{\wp^{\prime\prime}(p)}{\wp^{\prime
}(p)}\,}{z- p}+O(1),
\end{equation}
forcing
\begin{equation}
\label{0822equ11}T=\dfrac{\wp^{\prime\prime}(p)}{2\wp^{\prime}(p)}.
\end{equation}
Hence $y_{a}(z)$ is a solution precisely when \eqref{0822equ10} and
\eqref{0822equ11} hold, and
\begin{align*}
c  &  =T+2\zeta( p)=\dfrac{\wp^{\prime\prime}(p)}{2\wp^{\prime}(p)}+2\zeta(
p)=\zeta(2p),
\end{align*}
verifying \eqref{c determined by a1=a2=pm p} and
\eqref{T determined by a1=a2=pm p}. }
\end{proof}

{\normalsize For $P\in\Gamma(\tau,p)$, define
\begin{align}
\alpha(P)  &  :=r(P)+s(P)\tau,\label{alpha}\\
\kappa(\alpha(P))  &  := Z\bigl(r(P),s(P),\tau\bigr). \label{kappa}%
\end{align}
From \eqref{algebraic relations between a1(P),a2(P),r(P),s(P),c(P)}, we
obtain
\[
a_{1}(P)+a_{2}(P)=r(P)+s(P)\tau=\alpha(P).
\]
Since $c(P)=\zeta(\alpha(P))\neq\infty$, $\alpha(P)=a_{1}(P)+a_{2}(P)\neq0$.
Consequently,
\begin{align*}
\kappa(\alpha(P))  &  =\zeta\left(  a_{1}(P)+a_{2}(P)\right)  -c(P)\\
&  = \zeta\left(  a_{1}(P)+a_{2}(P)\right)  -\zeta\left(  a_{1}(P)+a_{2}
(P)\right)  =0.
\end{align*}
With the above notations, we obtain the following lemma. }

\begin{lemma}
{\normalsize \label{system lemma} Let $P=(T,C)\in\Gamma(\tau,p)$. Then
\begin{align}
\wp\bigl(\alpha(P)\bigr)  &  = T^{2}-2\wp(p),\label{wp(alpha)}\\
\wp^{\prime}\bigl(\alpha(P)\bigr)  &  = \sqrt{-Q(T)\,},\label{wp'(alpha)}\\
\kappa(\alpha(P))  &  = 0. \label{kappa=0}%
\end{align}
}
\end{lemma}

\begin{proof}
{\normalsize We prove \eqref{wp(alpha)} and \eqref{wp'(alpha)}. First, assume
$a_{1}(P)\neq a_{2}(P)$. Namely, $a_{1}(P),a_{2}(P)\notin\{0,\pm p\}$. By the
addition formula for the Weierstrass $\wp$-function
\eqref{addition formula for wp}, \eqref{wp(alpha)} follows from
\eqref{a1 and a2's condition} and \eqref{T determined by a1, a2}. If instead
$a_{1}(P)=a_{2}(P)\in\{\pm p\}$, then another addition formula gives
\[
\wp(\alpha(P))=\wp(2p)=-2\wp(p)+\dfrac{\wp^{\prime\prime2}(p)}{4\,\wp
^{\prime2}(p)}.
\]
In such case, \eqref{wp(alpha)} follows from
\eqref{T determined by a1=a2=pm p}. Finally, using the classical differential
equation \eqref{classical differential equation} together with
\eqref{spectral polynomial in noneven case}, we obtain \eqref{wp'(alpha)}. }
\end{proof}

{\normalsize Given $(r,s)\in\mathbb{C}^{2}$, set $\alpha:=r+s\tau\neq0$.
Motivated by Lemma \ref{system lemma}, we next study the solvability of the
system
\begin{equation}
\label{system of wp,wp',kappa}%
\begin{cases}
~\wp(\alpha)=T^{2}-2\wp(p),\\
~\wp^{\prime}(\alpha)=\sqrt{-Q(T)\,},\\
~\kappa(\alpha)=0.
\end{cases}
\end{equation}
}

\begin{lemma}
{\normalsize \label{data corresponding lemma} Let $(r,s)\in\mathbb{C}^{2}$
with $\alpha=r+s\tau\neq0$. }

\begin{itemize}

\item[(i)] {\normalsize If $P=(T,C)\in\Gamma(\tau,p)$, then $T$ solves
\eqref{system of wp,wp',kappa} with $(r,s) = (r(P),s(P))$. }

\item[(ii)] {\normalsize Suppose $T$ solves \eqref{system of wp,wp',kappa}.
Then there exists $P=(T,C)\in\Gamma(\tau,p)$ such that $(r,s) = (r(P),s(P))$.
}
\end{itemize}
\end{lemma}

\begin{proof}
{\normalsize Part $(i)$ follows directly from Lemma \ref{system lemma}.
\newline For part $(ii)$, let $(r,s)\in\mathbb{C}^{2}$ and suppose
$T\in\mathbb{C}$ satisfies \eqref{system of wp,wp',kappa}. Set $P=(T,C)\in
\Gamma(\tau,p)$. Since $T$ solves \eqref{system of wp,wp',kappa}, we have
\begin{align*}
\wp(r(P)+s(P)\tau)  &  =\wp(r+s\tau),\\
\wp^{\prime}(r(P)+s(P)\tau)  &  =\wp^{\prime}(r+s\tau),\\
\kappa(\alpha(P))  &  =\kappa(\alpha).
\end{align*}
These equations imply
\begin{align*}
r(P)+s(P)\tau &  =r+s\tau,\\
r(P)\,\eta_{1}(\tau) + s(P)\,\eta_{2}(\tau)  &  =r\,\eta_{1}(\tau) +
s\,\eta_{2}(\tau).
\end{align*}
Then $(r(P),s(P))=(r,s)$ follows from the Legendre relation
\eqref{Legendre relation}. }
\end{proof}

\begin{proof}
[Proof of Theorem \ref{main theorem for Lame-type ODE}]{\normalsize It follows
directly from Lemma \ref{system lemma} and Lemma
\ref{data corresponding lemma}. }
\end{proof}

{\normalsize By the correspondence
\eqref{classical Lame to Lame-type's correspondence}, the systems
\eqref{system of wp,wp',kappa, for Lame n=1} and
\eqref{system of wp,wp',kappa*} are equivalent. The following corollary
follows from the last statement of Theorem I. }

\begin{corollary}
{\normalsize \label{pre-modular form lemma} For any $(r,s)\in\mathbb{C}%
^{2}\setminus\tfrac{1}{2}\mathbb{Z}^{2}$, the generalized Lam\'{e}-type
equation \eqref{GLE 4}$_{p}$ is completely reducible with monodromy data
$(r,s)$ if and only if $Z(r,s,\tau)=0$. }
\end{corollary}

{\normalsize
}

\subsection{Non-completely Reducible Case}

{\normalsize To prove Theorem \ref{Main Theorem1}, it remains to treat the
non-completely reducible case of equation \eqref{GLE 4}$_{p}$. }

{\normalsize The proofs of the following Lemmas \ref{0822lemma1} and
\ref{0822lemma2} for determining the monodromy data $\mathcal{D}$ in the
non-completely reducible case are essentially the same for the classical
Lam\'{e} equation \eqref{Lame 1} and for the generalized Lam\'{e}-type
equation \eqref{GLE 4}$_{p}$. We therefore present the proofs only for
\eqref{GLE 4}$_{p}$. }

{\normalsize Fix a base point $x_{0}\in E_{\tau}\setminus\{0,\pm p\}$, for any
$T\in\mathbb{C}$ and $z$ sufficiently close to $x_{0}$, define
\begin{equation}
\label{chi(z;T), definition}\chi_{j}(z;T):=\int_{z}^{z+\omega_{j}}\dfrac
{1}{\,\Phi_{e}(\xi;T)\,}\,d\xi,\quad j=1,2,
\end{equation}
These integrals are locally holomorphic. Since $\Phi_{e}(z;T)$ is elliptic,
$\chi_{j}(z;T)$ is independent of $z$, for $j=1,2$. We denote them by
$\chi_{j}(T)$. }

\begin{lemma}
{\normalsize \label{0822lemma1} Suppose the generalized Lam\'{e}-type equation
\eqref{GLE 4}$_{p,\,T_{0}}$ is non-completely reducible at parameter $T_{0}$
with monodromy data $\mathcal{D}\in\mathbb{C}\cup\{\infty\}$. Then
\begin{equation}
\label{0822equ1}\mathcal{D}=\dfrac{\,\chi_{2}(T_{0})\,}{\chi_{1}(T_{0})}.
\end{equation}
}
\end{lemma}

\begin{proof}
{\normalsize Since \eqref{GLE 4}$_{p,\,T_{0}}$ is non-completely reducible,
Theorem \ref{Q(T) neq 0 iff. r,s notin Z/2} implies that the Baker-Akhiezer
function $\psi(P_{0};z)$ satisfies
\begin{equation}
\label{0810equ3}\psi(P_{0};z+\omega_{j})=\epsilon_{j}\,\psi(P_{0};z),
\quad\epsilon_{j}\in\{\pm1\}.
\end{equation}
Hence $\psi(P_{0};z)^{2}$ is elliptic and solves \eqref{3rd ODE}.
Consequently, up to a nonzero multiple,
\begin{equation}
\label{0810equ4}\Phi_{e}(z;T_{0})=\psi^{2}(P_{0};z),
\end{equation}
Let $\tilde{y}(z)$ be another solution of \eqref{GLE 4}$_{p,\,T_{0}}$ linearly
independent of $\psi(P_{0};z)$, which is not elliptic of the second kind.
Define
\begin{equation}
\label{0810equ5}\chi(z):=\dfrac{\tilde{y}(z)}{\,\psi(P_{0};z)\,}.
\end{equation}
Substituting $\tilde{y}(z)=\chi(z)\,\psi(P_{0};z)$ into
\eqref{GLE 4}$_{p,\,T_{0}}$ yields
\begin{equation}
\label{0818equ1}\dfrac{\chi^{\prime\prime}(z)}{\chi^{\prime}(z)}%
=-2\,\dfrac{\psi^{\prime}(P_{0};z)}{\psi(P_{0};z)}.
\end{equation}
Integrating \eqref{0818equ1} and using \eqref{0810equ4}, we obtain
\begin{equation}
\label{0818equ2}\chi^{\prime}(z)=\dfrac{\tilde{b}}{\psi^{2}(P_{0};z)}%
=\dfrac{b}{\Phi_{e}(z;T_{0})},
\end{equation}
for some $\tilde{b},b\in\mathbb{C}\setminus\{0\}$. Thus $\chi(z)$ is
quasi-periodic. By \eqref{0818equ2} and \eqref{chi(z;T), definition}, we have
\begin{equation}
\label{0818equ3}%
\begin{cases}
~\chi(z+\omega_{1})=\chi(z)+b\,\chi_{1}(T_{0}),\\
~\chi(z+\omega_{2})=\chi(z)+b\,\chi_{2}(T_{0}).
\end{cases}
\end{equation}
Note that $\chi_{1}(T_{0})$ and $\chi_{2}(T_{0})$ cannot vanish
simultaneously; otherwise $\tilde{y}(z)$ would be elliptic of the second
kind.\newline\textbf{Case 1.} If $\chi_{1}(T_{0})=0$, then $\chi_{2}%
(T_{0})\neq0$. Therefore $\frac{\chi_{2}(T_{0})}{\,\chi_{1}(T_{0})\,}=\infty$.
In this case, the fundamental solution
\[
Y(z)=\big(b\chi_{2}(T_{0})\,\psi(P_{0};z),\,\tilde{y}(z)\big)^{T}
\]
has monodromy matrices
\begin{equation}
\label{0823equ3}%
\begin{split}
\left(
\begin{array}
[c]{c}%
b\chi_{2}(T_{0})\,\psi(P_{0};z+\omega_{1})\\
\tilde{y}(z+\omega_{1})
\end{array}
\right)   &  =\epsilon_{1} \left(
\begin{array}
[c]{cc}%
1 & 0\\
0 & 1
\end{array}
\right)  \left(
\begin{array}
[c]{c}%
b\chi_{2}(T_{0})\,\psi(P_{0};z)\\
\tilde{y}(z)
\end{array}
\right)  ,\\
\left(
\begin{array}
[c]{c}%
b\chi_{2}(T_{0})\,\psi(P_{0};z+\omega_{2})\\
\tilde{y}(z+\omega_{2})
\end{array}
\right)   &  =\epsilon_{2} \left(
\begin{array}
[c]{cc}%
1 & 0\\
1 & 1
\end{array}
\right)  \left(
\begin{array}
[c]{c}%
b\chi_{2}(T_{0})\,\psi(P_{0};z)\\
\tilde{y}(z)
\end{array}
\right)  ,
\end{split}
\end{equation}
where $\epsilon_{1},\epsilon_{2}\in\{\pm1\}$. It follows from \eqref{0823equ3}
that
\[
\frac{\chi_{2}(T_{0})}{\,\chi_{1}(T_{0})\,}=\mathcal{D}=\infty.
\]
\textbf{Case 2.} If $\chi_{1}(T_{0})\neq0$, then $\frac{\chi_{2}(T_{0}%
)}{\,\chi_{1}(T_{0})\,}\in\mathbb{C}$, and
\[
Y(z)=\big(\chi_{1}(T_{0})\psi(P_{0};z),\,\tilde{y}(z)\big)^{T}
\]
forms a fundamental system. Its monodromy matrices satisfy
\begin{equation}
\label{0823equ4}%
\begin{split}
\left(
\begin{array}
[c]{c}%
b\chi_{1}(T_{0})\,\psi(P_{0};z+\omega_{1})\\
\tilde{y}(z+\omega_{1})
\end{array}
\right)   &  =\epsilon_{1} \left(
\begin{array}
[c]{cc}%
1 & 0\\
1 & 1
\end{array}
\right)  \left(
\begin{array}
[c]{c}%
b\chi_{1}(T_{0})\,\psi(P_{0};z)\\
\tilde{y}(z)
\end{array}
\right)  ,\\
\left(
\begin{array}
[c]{c}%
b\chi_{1}(T_{0})\,\psi(P_{0};z+\omega_{2})\\
\tilde{y}(z+\omega_{2})
\end{array}
\right)   &  =\epsilon_{2} \left(
\begin{array}
[c]{cc}%
1 & 0\\
\frac{\chi_{2}(T_{0})}{\,\chi_{1}(T_{0})\,} & 1
\end{array}
\right)  \left(
\begin{array}
[c]{c}%
b\chi_{1}(T_{0})\,\psi(P_{0};z)\\
\tilde{y}(z)
\end{array}
\right)  ,
\end{split}
\end{equation}
where $\epsilon_{1},\epsilon_{2}\in\{\pm1\}$. It follows from \eqref{0823equ4}
that
\[
\frac{\chi_{2}(T_{0})}{\,\chi_{1}(T_{0})\,}=\mathcal{D}.
\]
This completes the proof. }
\end{proof}

{\normalsize
}

\begin{lemma}
{\normalsize \label{0822lemma2} Suppose the generalized Lam\'{e}-type equation
\eqref{GLE 4}$_{p,\,T_{0}}$ is non-completely reducible at parameter $T_{0}$
with monodromy data $\mathcal{D}\in\mathbb{C}\cup\{\infty\}$. Then
$\mathcal{D}$ can be determined by
\begin{equation}
\label{monodromy data, non-completely reducible}\mathcal{D}=\lim
\limits_{\substack{P=(T,C)\,\rightarrow\, P_{0}\\C^{2}=Q(T)\neq0}%
}\dfrac{\,-r(P_{0})+r(P)\,}{s(P_{0})-s(P)}.
\end{equation}
}
\end{lemma}

\begin{proof}
{\normalsize By Theorem \ref{Q(T) neq 0 iff. r,s notin Z/2}, we have
$(r(P_{0}),s(P_{0}))\in\tfrac{1}{2}\mathbb{Z}^{2}$. The polynomial
$Q(T)\in\mathbb{C}[T]$ is nontrivial with $Q(T_{0})=0$. Choose any sequence
$P_{k}=(T_{k},C_{k})\in\Gamma(\tau,p)$ with $C_{k}^{2}=Q(T_{k})\neq0$ such
that $T_{k}\to T_{0}$. Then the Baker-Akhiezer functions $\psi(P_{k};z)$
converge uniformly to $\psi(P_{0};z)$ on compact subsets of $E_{\tau}%
\setminus\{0,\pm p\}$. Consequently,
\begin{equation}
\label{0818equ6}\left(  r(P_{k}),\,s(P_{k})\right)  \to\left(  r(P_{0}%
),\,s(P_{0})\right)  \in\frac{1}{2}\mathbb{Z}^{2},\quad\text{as }\,
k\to+\infty.
\end{equation}
For each $k$, set
\begin{equation}
\label{0818equ7}f_{k}(z):=\dfrac{\psi(P_{k};z)}{\,\psi(P_{k}^{*};z)\,},
\end{equation}
where $P_{k}^{*}=(T_{k},-C_{k})\in\Gamma(\tau,p)$ is the dual point of
$P_{k}=(T_{k},C_{k})$. Using \eqref{r(P*),s(P*)}, one obtains
\begin{equation}
\label{0818equ10}%
\begin{split}
f_{k}(z+1)  &  =\dfrac{e^{-2\pi i\,s(P_{k})}\,\psi(P_{k};z)}{e^{-2\pi
i\,s(P_{k}^{*})}\,\psi(P_{k}^{*};z)}=e^{-4\pi i\,s(P_{k})}\,f_{k}(z),\\
f_{k}(z+\tau)  &  =\dfrac{e^{2\pi i\,r(P_{k})}\,\psi(P_{k};z)}{e^{2\pi
i\,r(P_{k}^{*})}\,\psi(P_{k}^{*};z)}=e^{4\pi i\,r(P_{k})}\,f_{k}(z).
\end{split}
\end{equation}
Moreover, by \eqref{properties for BA,1} and \eqref{properties for BA,2}, we
have
\begin{equation}
\label{0818equ9}%
\begin{split}
\left(  \log f_{k}(z)\right)  ^{\prime}  &  = \dfrac{\,\psi^{\prime}%
(P_{k};z)\,}{\psi(P_{k};z)}-\dfrac{\,\psi^{\prime}(P_{k}^{*};z)\,}{\psi
(P_{k}^{*};z)}\\
&  = \dfrac{\,W(\psi(P_{k};z),\,\psi(P_{k}^{*};z))\,}{\psi(P_{k}%
;z)\,\psi(P_{k}^{*};z)} = \dfrac{2i\,C_{k}}{\,\Phi_{e}(z;T_{k})\,}.
\end{split}
\end{equation}
Integrating \eqref{0818equ9} along the cycles $\ell_{j}$, it follows from
\eqref{0818equ10} that
\begin{equation}
\label{0822equ2}%
\begin{split}
\chi_{1}(T_{k})  &  =\frac{\,-2\pi\,s(P_{k})+\pi\,m_{1}(T_{k})\,}{C_{k}},\\
\chi_{2}(T_{k})  &  =\frac{\,2\pi\,r(P_{k})+\pi\,m_{2}(T_{k})\,}{C_{k}},
\end{split}
\end{equation}
for some $m_{1}(T_{k}),m_{2}(T_{k})\in\mathbb{Z}$. Since $\chi_{j}(z;T)$ is
locally holomorphic, it follows that
\begin{equation}
\label{0822equ3}%
\begin{cases}
~\chi_{1}(T_{k})\rightarrow\chi_{1}(T_{0})\in\mathbb{C},\\
~\chi_{2}(T_{k})\rightarrow\chi_{2}(T_{0})\in\mathbb{C},
\end{cases}
\qquad\text{as }\,k\rightarrow+\infty.
\end{equation}
Since $Q(T_{k})\to Q(T_{0})=0$, we have $C_{k}\to0$, and hence
\begin{equation}
\label{0822equ4}%
\begin{cases}
-2\pi\,s(P_{k})+\pi\,m_{1}(T_{k})\rightarrow0,\\
\textcolor{white}{-}2\pi\,r(P_{k})+\pi\,m_{2}(T_{k})\rightarrow0,
\end{cases}
\qquad\text{as }\,k\rightarrow+\infty.
\end{equation}
Combining \eqref{0818equ6} and \eqref{0822equ4} yields
\begin{equation}
\label{0822equ5}m_{1}(T_{k})=2s(P_{0})\quad\text{and}\quad m_{2}%
(T_{k})=-2r(P_{0}),
\end{equation}
for $k$ sufficiently large. Finally, by Lemma \ref{0822lemma1} together with
\eqref{0822equ5},
\begin{align*}
\mathcal{D}  &  = \lim_{k\to+\infty}\dfrac{\,\chi_{2}(T_{k})\,}{\chi_{1}%
(T_{k})}\\
&  = \lim_{k\to+\infty}\dfrac{\,2\pi\,r(P_{k})+\pi\,m_{2}(T_{k})\,}%
{-2\pi\,s(P_{k})+\pi\,m_{1}(T_{k})}\\
&  = \lim_{k\to+\infty}\dfrac{\,-r(P_{0})+r(P_{k})\,}{s(P_{0})-s(P_{k})}.
\end{align*}
This establishes \eqref{monodromy data, non-completely reducible} and thereby
completes the proof. }
\end{proof}

\subsection{The Proofs}

{\normalsize In this subsection, we present the proofs of Theorem
\ref{Main Theorem1}, Theorem \ref{Main Theorem at singular} and Theorem
\ref{Main Theorem3}. }

\begin{proof}
[Proof of Theorem \ref{Main Theorem1}]{\normalsize Let $\tilde{\Gamma}(\tau)$
and $\Gamma(\tau,p)$ be the spectral curves of \eqref{Lame 1} and
\eqref{GLE 4}$_{p}$, respectively. Recall the correspondence
\begin{equation}
\label{0826equ20}%
\begin{split}
\Gamma(\tau,p)  &  \rightarrow\tilde{\Gamma}(\tau)\\
P=(T(P),C(P))  &  \mapsto\tilde{P}=(\tilde{B}(P),\tilde{C}(P)),
\end{split}
\end{equation}
which are related by
\begin{equation}
\tilde{B}(\tilde{P})=T^{2}(P)-2\wp(p).
\end{equation}
}

{\normalsize Suppose the generalized Lam\'{e}-type equation \eqref{GLE 4}$_{p}%
$ at parameter $T_{0}$ is non-completely reducible with monodromy data
$\mathcal{D}\in\mathbb{C}\cup\{\infty\}$. Let $P_{0}=(T_{0},0)\in\Gamma
(\tau,p)$, and let $\tilde{P}_{0}$ denote its corresponding point on
$\tilde{\Gamma}(\tau)$ via \eqref{0826equ20}. Namely,
\[
\tilde{P}_{0}=(\tilde{B}_{0},0)\in\tilde{\Gamma}(\tau),\qquad\tilde{B}%
_{0}=T_{0}^{2}-2\wp(p).
\]
Choose a sequence $P_{k}=(T(P_{k}),C(P_{k}))\in\Gamma(\tau,p)$ with
$C(P_{k})\neq0$ and $P_{k}\to P_{0}$ as $k\to+\infty$. Let $\tilde{P}%
_{k}=(\tilde{B}(P_{k}),\tilde{C}(P_{k}))\in\tilde{\Gamma}(\tau)$ denote the
corresponding points of $P_{k}\in\Gamma(\tau,p)$ under \eqref{0826equ20}. Then
$\tilde{C}(P_{k})\neq0$ and $\tilde{P}_{k}\to\tilde{P}_{0}$. Moreover,
\[%
\begin{split}
\left(  r(P_{k}),s(P_{k})\right)   &  =\left(  r(\tilde{P}_{k}),s(\tilde
{P}_{k})\right)  \notin\frac{1}{2}\mathbb{Z}^{2},\\
\left(  r(P_{0}),s(P_{0})\right)   &  =\left(  r(\tilde{P}_{0}),s(\tilde
{P}_{0})\right)  \in\frac{1}{2}\mathbb{Z}^{2}.
\end{split}
\]
By Lemma \ref{0822lemma2}, the monodromy data $\mathcal{D}$ is then given by
\begin{align*}
\lim\limits_{k\to+\infty}\dfrac{\,-r(P_{0})+r(P_{k})\,}{s(P_{0})-s(P_{k}%
)}=\mathcal{D}=\lim\limits_{k\to+\infty}\dfrac{\,-r(\tilde{P}_{0})+r(\tilde
{P}_{k})\,}{s(\tilde{P}_{0})-s(\tilde{P}_{k})}.
\end{align*}
The converse follows by the same argument. Therefore, in the non-completely
reducible case, the classical Lam\'{e} equation \eqref{Lame 1} and the
generalized Lam\'{e}-type equation \eqref{GLE 4}$_{p}$ are monodromy
equivalent. This completes the proof. }
\end{proof}

\begin{proof}
[Proof of Theorem \ref{Main Theorem at singular}]{\normalsize As in Corollary
\ref{Main Corollary1}, let the parameters $T(p),B(p)$ be given by
\eqref{correspondence} and \eqref{B condition 2}, respectively, so that
\begin{equation}
\label{correspondence'}\tilde{B}=T^{2}(p)-2\wp(p).
\end{equation}
Relation \eqref{correspondence'} implies
\[
T(p)=0\,\Leftrightarrow\,\tilde{B}=-2\wp(p).
\]
We now show that $T(p)=0$ if and only if $p=\pm p_{r,s}^{(1)}(\tau)$. }

{\normalsize Assume $T(p)=0$. Then \eqref{GLE 4}$_{p}$ reduces to the
even-symmetric case. Since the monodromy data are $(r,s)$, Theorem B yields
$p=p_{r,s}^{(1)}(\tau)$. Conversely, suppose $p=\pm p_{r,s}^{(1)}(\tau)$. Then
$p$ satisfies \eqref{PV6-1}. Since $(r,s)$ is the monodromy data of
\eqref{GLE 4}$_{p}$, Theorem \ref{main theorem for Lame-type ODE} implies
that
\[
\wp(r+s\tau)=T^{2}(p)-2\wp(p)\quad\text{and}\quad Z(r,s,\tau)=0.
\]
Combining these with \eqref{PV6-1}, we obtain
\[
-2\wp(p)=\wp(r+s\tau)=T^{2}(p)-2\wp(p),
\]
and hence $T(p)=0$. This completes the proof. }
\end{proof}

{\normalsize We investigate the limiting behavior as $p\to\tfrac{\omega_{k}%
}{2}$ for $k=0,1,2,3$. }

\begin{lemma}
{\normalsize \label{convergence lemma} Let $(T(p),B(p))$ be as in Corollary
\ref{Main Corollary1}. Namely,
\begin{equation}
\label{correspondence'''''}T(p)^{2}-2\wp(p)=\tilde{B},
\end{equation}
with $B(p)$ determined by \eqref{B condition 2}. }

\begin{enumerate}
\item[(i)] {\normalsize As $p\to0$, $q(z;T(p))$ converges uniformly to
$2\wp(z)+\tilde{B}$ on every compact subset of $E_{\tau}\setminus\{0\}$. }

\item[(ii)] {\normalsize For each $k=1,2,3$, as $p\to\frac{\omega_{k}}{2}$,
$q(z;T(p))$ converges uniformly on every compact subset of $E_{\tau}%
\setminus\{0,\tfrac{\omega_{k}}{2}\}$ to
\begin{equation}
\label{0825equ5}2\left(  \wp(z)+\wp\left(  z-\frac{\omega_{k}} {2}\right)
\right)  +2\tilde{T}_{k}\left(  \zeta\left(  z-\frac{\omega_{k}}{2}\right)
-\zeta(z)\right)  +\tilde{B}_{k}%
\end{equation}
where
\[
\tilde{T}_{k}=\lim_{p\rightarrow\frac{\omega_{k}}{2}}T(p)\quad\text{and}%
\quad\tilde{B}_{k}=\tilde{T}_{k}^{2}+\eta_{k}\,\tilde{T}_{k}-e_{k}.
\]
}
\end{enumerate}
\end{lemma}

\begin{proof}
{\normalsize Since $\tilde{B}$ is independent of $p$,
\eqref{correspondence'''''} implies
\begin{equation}
\label{0825equ6}\lim_{p\to0}\bigl(T(p)^{2}-2\wp(p)\bigr)=\tilde{B}.
\end{equation}
Hence, as $p\to0$,
\begin{equation}
\label{behavior of T(p) as p -> 0}%
\begin{split}
T(p)^{2}  &  =\dfrac{2}{\,p^{2}\,}+\tilde{B}+O(p).
\end{split}
\end{equation}
Moreover, as $p\to0$,
\begin{equation}
\label{0825equ7}%
\begin{split}
\wp(z+p)+\wp(z-p)  &  =2\wp(z)+O(p^{2}),\\
\zeta(z+p)+\zeta(z-p)-2\zeta(z)  &  =O(p^{2}),\\
\zeta(z+p)-\zeta(z-p)  &  =-2\wp(z)\,p+O(p^{2}),
\end{split}
\end{equation}
and therefore
\begin{equation}
\label{0825equ8}%
\begin{split}
B(p)  &  =T^{2}(p)+\dfrac{\wp^{\prime\prime}(p)}{\,2\wp^{\prime}(p)\,}%
\zeta(p)-\dfrac{1}{2}\wp(p)\\
&  =\dfrac{2}{\,p^{2}\,}+\tilde{B}+\dfrac{6/p^{4}}{\,2\left(  -2/p^{3}\right)
p\,}-\dfrac{1}{\,2p^{2}\,}+O(p)\\
&  =\tilde{B}+O(p).
\end{split}
\end{equation}
It follows from \eqref{behavior of T(p) as p -> 0}-\eqref{0825equ8} that, as
$p\to0$,
\begin{equation}
\label{0826equ3}%
\begin{split}
q(z;T(p))  &  = \left[
\begin{array}
[c]{c}%
2\wp(z)+\dfrac{3}{4}\left(  \wp(z+p)+\wp(z-p)\right) \\[4pt]%
+T(p)\left(  \zeta(z+p)+\zeta(z-p)-2\zeta(z)\right) \\[4pt]%
-\dfrac{1}{4}\dfrac{\wp^{\prime\prime}(p)}{\wp^{\prime}(p)}\left(
\zeta(z+p)-\zeta(z-p)\right)  +B(p)
\end{array}
\right] \\
&  =\dfrac{7}{2}\wp(z)+\dfrac{6/p^{4}}{\,-4\left(  -2/p^{3}\right)  \,}\left(
-2\wp(z)\,p\right)  +\tilde{B}+O(p)\\
&  =2\wp(z)+\tilde{B}+O(p).
\end{split}
\end{equation}
This proves part \textit{(i)}. For part \textit{(ii)}. Set $p_{k}%
:=p-\tfrac{\omega_{k}}{2}$.
As $p\to\tfrac{\omega_{k}}{2}$,
\begin{equation}
\label{0825equ10}%
\begin{split}
\wp(z+p)+\wp(z-p)  &  =2\wp\left(  z-\frac{\omega_{k}}{2}\right)  +O(p_{k}%
^{2}),\\
\zeta(z+p)+\zeta(z-p)-2\zeta(z)  &  =\eta_{k}-2\zeta(z)+2\zeta\left(
z-\frac{\omega_{k}}{2}\right)  +O(p_{k}^{2}),\\
\zeta(z+p)-\zeta(z-p)  &  =\eta_{k}-2\wp\left(  z-\frac{\omega_{k}}{2}\right)
p_{k}+O(p_{k}^{2}),
\end{split}
\end{equation}
and consequently
\begin{equation}
\label{0825equ11}%
\begin{split}
B(p)  &  =\tilde{T}_{k}^{2}+\dfrac{\wp^{\prime\prime}(\omega_{k}/2)}%
{\,2\wp^{\prime\prime}(\omega_{k}/2)\,p_{k}\,}\left(  \dfrac{\eta_{k}}%
{2}-e_{k}\,p_{k}\right)  -\dfrac{\,e_{k}\,}{2}+O(p_{k})\\
&  = \dfrac{\eta_{k}}{\,4p_{k}\,}+\tilde{T}_{k}^{2}-e_{k}+O(p_{k}).
\end{split}
\end{equation}
Using \eqref{0825equ10} and \eqref{0825equ11}, as $p\to\frac{\omega_{k}}{2}$,
\begin{equation}
\label{0826equ4}%
\begin{split}
q(z;T(p))  &  = \left[
\begin{array}
[c]{c}%
2\wp(z)+\dfrac{3}{4}\left(  \wp(z+p)+\wp(z-p)\right) \\[4pt]%
+T(p)\left(  \zeta(z+p)+\zeta(z-p)-2\zeta(z)\right) \\[4pt]%
-\dfrac{1}{4}\dfrac{\wp^{\prime\prime}(p)}{\wp^{\prime}(p)}\left(
\zeta(z+p)-\zeta(z-p)\right)  +B(p)
\end{array}
\right] \\
&  =\left[
\begin{array}
[c]{c}%
2\wp(z)+\dfrac{3}{2}\wp\left(  z-\dfrac{\omega_{k}}{2}\right) \\[6pt]%
+ \tilde{T}_{k}\left(  \eta_{k}-2\zeta(z)+2\zeta\left(  z-\dfrac{\omega_{k}%
}{2}\right)  \right) \\[6pt]%
- \dfrac{\wp^{\prime\prime}(\omega_{k}/2)}{4\wp^{\prime\prime}(\omega
_{k}/2)\,p_{k}}\left(  \eta_{k}-2\wp\left(  z-\dfrac{\omega_{k}}{2}\right)
p_{k}\right) \\[12pt]%
+ \dfrac{\eta_{k}}{\,4p_{k}\,}+\tilde{T}_{k}^{2}-e_{k}+O(p_{k})
\end{array}
\right] \\
&  = 2\left(  \wp(z)+\wp\left(  z-\frac{\omega_{k}} {2}\right)  \right)
+2\tilde{T}_{k}\left(  \zeta\left(  z-\frac{\omega_{k}}{2}\right)
-\zeta(z)\right)  +\tilde{B}_{k},
\end{split}
\end{equation}
which establishes part \textit{(ii)}. }
\end{proof}

\begin{proof}
[Proof of Theorem \ref{Main Theorem3}]{\normalsize Let $(T(p),B(p))$ be as in
Corollary \ref{Main Corollary1}. Let $\psi_{p}(z)$ denote the the
Baker-Akhiezer function of the generalized Lam\'{e}-type equation
\eqref{GLE 4}$_{p}$ with parameters $(T(p),B(p))$. Since $\psi_{p}(z)$ is
elliptic of the second kind, by \eqref{BA function written in Hermite} we
have
\begin{equation}
\label{0826equ8}\psi_{p}(z)=e^{c_{p}z}\dfrac{\sigma\left(  z-a_{1}%
^{(p)}\right)  \sigma\left(  z-a_{2}^{(p)}\right)  }{\,\sigma(z)\sqrt
{\sigma(z+p)}\sqrt{\sigma(z-p)}\,},\quad%
\begin{array}
[c]{l}%
\text{where \,} c_{p}\in\mathbb{C},\\[4pt]%
a_{1}^{(p)},\,a_{2}^{(p)}\in E_{\tau}\setminus\{0\}.
\end{array}
\end{equation}
From \eqref{algebraic relations between a1(P),a2(P),r(P),s(P),c(P)} we have,
for all $p\in E_{\tau}\setminus\{\tfrac{\omega_{k}}{2}:k=0,1,2,3\}$,
\begin{equation}
\label{0826equ10}%
\begin{cases}
~r+s\tau=a_{1}^{(p)}+a_{2}^{(p)},\\
~r\eta_{1}+s\eta_{2}=c^{(p)}.
\end{cases}
\end{equation}
Passing to a subsequence if necessary, we may assume $a_{j}^{(p)}\to a_{j}$ as
$p\to\tfrac{\omega_{k}}{2}$, $j=1,2$. Let $p\to\frac{\omega_{k}}{2}$, we have
\begin{equation}
\label{0826equ23}\psi_{p}(z)\to\psi_{k}(z)=e^{cz}\dfrac{\sigma(z-a_{1}%
)\,\sigma(z-a_{2})}{\,\sigma(z)\sqrt{\sigma\left(  z+\frac{\omega_{k}}%
{2}\right)  \,}\sqrt{\sigma\left(  z-\frac{\omega_{k}}{2}\right)  \,}\,},
\end{equation}
uniformly on compact subsets of $E_{\tau}\setminus\{0,\frac{\omega_{k}}{2}\}$,
as $p\to\tfrac{\omega_{k}}{2}$, where
\begin{equation}
\label{0826equ9}%
\begin{split}
c  &  =r\eta_{1}+s\eta_{2},\\
a_{j}  &  =\lim_{p\to\frac{\omega_{k}}{2}}a_{j}^{(p)},\quad j=1,2.
\end{split}
\end{equation}
As $p\to\frac{\omega_{k}}{2}$, it follows from \eqref{0826equ10} and
\eqref{0826equ9} that
\begin{equation}
\label{0826equ10*}%
\begin{cases}
~r+s\tau=a_{1}+a_{2},\\
~r\eta_{1}+s\eta_{2}=c.
\end{cases}
\end{equation}
If $k=0$, the limiting function
\[
\psi_{0}(z)=e^{cz}\dfrac{\,\sigma(z-a_{1})\,\sigma(z-a_{2})\,}{\sigma^{2}(z)}
\]
is precisely the Baker-Akhiezer function of \eqref{Lame 1} used to define the
monodromy data $(r(\tilde{P}),s(\tilde{P}))$ for $\tilde{P}=(\tilde{B}%
,\tilde{C})\in\tilde{\Gamma}(\tau)$, hence
\[
(r(\tilde{P}),s(\tilde{P}))=(r,s).
\]
For $k\in\{1,2,3\}$, the Baker-Akhiezer function for the limiting equation
\eqref{limitting k}$_{\tilde{B}_{k}}$, used to define the monodromy data
$(r(\tilde{B}_{k}),s(\tilde{B}_{k}))$, is given by
\[
\hat{\psi}_{k}(z)=e^{c_{k}z}\dfrac{\,\sigma(z-a_{1})\,\sigma(z-a_{1}%
)\,}{\sigma(z)\,\sigma\left(  z-\frac{\omega_{k}}{2}\right)  },
\]
which satisfies
\begin{equation}
\label{0826equ14}%
\begin{cases}
~r(\tilde{B}_{k})+s(\tilde{B}_{k})=a_{1}+a_{2}-\dfrac{\omega_{k}}{2},\\[4pt]%
~r(\tilde{B}_{k})\,\eta_{1}+s(\tilde{B}_{k})\,\eta_{2}=c_{k}.
\end{cases}
\end{equation}
Comparing $\psi_{k}(z)$ with $\hat{\psi}_{k}(z)$ and using the transformation
law of the Weierstrass $\sigma$-function
\eqref{transformation law of the Weierstrass sigma function}, we obtain
\[
c_{k}=c-\dfrac{\eta_{k}}{2}.
\]
Together with \eqref{0826equ10*}, this yields
\[%
\begin{cases}
~r(\tilde{B}_{k})+s(\tilde{B}_{k})=r+s\tau-\dfrac{\omega_{k}}{2},\\[4pt]%
~r(\tilde{B}_{k})\,\eta_{1}+s(\tilde{B}_{k})\,\eta_{2}=r\eta_{1}+s\eta
_{2}-\dfrac{\eta_{k}}{2},
\end{cases}
\]
and hence
\[
\left(  r(\tilde{B}_{k}),\,s(\tilde{B}_{k})\right)  =
\begin{cases}
~\left(  r+\frac{1}{2},\,s\right)  & \text{if }\,k=1,\\
~\left(  r,\,s+\frac{1}{2}\right)  & \text{if }\,k=2,\\
~\left(  r+\frac{1}{2},\,s+\frac{1}{2}\right)  & \text{if }\,k=3.
\end{cases}
\]
}
\end{proof}

\end{document}